\documentclass[10pt,reqno]{amsart}
\usepackage{amsfonts,amsmath,latexsym,verbatim,amscd,mathrsfs,color,array}
\usepackage[colorlinks=true,citecolor=blue]{hyperref}

\usepackage{appendix}
\usepackage{amsmath,amssymb,amsthm,amsfonts,graphicx,color}
\usepackage{amsthm} 
\usepackage[nameinlink]{cleveref} 
\usepackage{amsthm}

\usepackage{amsmath}

\DeclareFontFamily{U}{mathx}{}
\DeclareFontShape{U}{mathx}{m}{n}{<-> mathx10}{}
\DeclareSymbolFont{mathx}{U}{mathx}{m}{n}
\DeclareMathAccent{\widehat}{0}{mathx}{"70}
\DeclareMathAccent{\widecheck}{0}{mathx}{"71}

\newtheorem{theorem}{Theorem}[section]  

\newtheorem{lemma}{Lemma}[section]
\newtheorem{prop}{Proposition}[section]
\newtheorem{corollary}{Corollary}[section]
\newtheorem{remark}{Remark}[section]
\newcommand{\bremark}{\begin{remark} \em}
	\newcommand{\eremark}{\end{remark} }

\allowdisplaybreaks \numberwithin{equation}{section}
\usepackage{color}
\usepackage{cases}

\numberwithin{equation}{section}

\theoremstyle{definition}

\theoremstyle{remark}

\newcommand{\ep}{\varepsilon}

\newcommand{\R}{\mathbb{R}}

\begin{document}

\title[KP-I Lump Solution to Travelling wave of Shadow Water wave]
{From KP-I Lump Solution to Travelling wave of 3D Gravity Capillary Water wave problem}

\author[C. Gui]{Changfeng Gui}
\address{\noindent  Changfeng Gui, Department of Mathematics, Faculty of Science and Technology, University of Macau, Taipa, Macao SAR, China.}
\email{changfenggui@um.edu.mo}

\author[S. Lai]{Shanfa Lai}
\address{\noindent   Shanfa Lai, Department of Mathematics, Faculty of Science and Technology, University of Macau, Taipa, Macao SAR, China.}
\email{laishanfa@amss.ac.cn}

\author[Y. Liu]{Yong Liu}
\address{\noindent  Yong Liu, School of Mathematics and Statistics, Beijing Technology and Business University, Beijing, China.}
\email{yliumath@btbu.edu.cn}

\author[J. Wei]{Juncheng Wei}
\address{\noindent  Juncheng Wei, Department of Mathematics, The Chinese University of Hong Kong, Shatin, Hong Kong.
}  \email{wei@math.cuhk.edu.hk}

\author[W. Yang]{Wen Yang}
\address{\noindent  Wen Yang, Department of Mathematics, Faculty of Science and Technology, University of Macau, Taipa, Macao SAR, China.
}  \email{wenyang@um.edu.mo}


\date{\today \\[2ex]
	\vspace*{2em}
} 

\begin{abstract}  
	In this paper, we study the three-dimensional gravity-capillary water wave problem involving an irrotational, perfect fluid with gravity and surface tension. We focus on steady waves propagating uniformly in one direction. Assuming constant wave speed and water depth, we analyze the fluid's velocity potential and boundary conditions. Using the Kadomtsev-Petviashvili (KP)-I equation as a simplified model, we show that, within a specific parameter range, the problem admits a fully localized solitary-wave solution resemble  lump type solutions of the KP-I equation.
\end{abstract}
\maketitle
\tableofcontents

\addtocontents{toc}{\protect\setcounter{tocdepth}{2}}	
\section{Introduction}

This article explores the classical three-dimensional gravity-capillary water-wave problem, which models the irrotational flow of an inviscid fluid with constant density, subject to gravity and surface tension forces. In the 1750s, L. Euler introduced the first comprehensive mathematical model for fluid dynamics, a framework that remains pivotal in coastal engineering, naval architecture, and environmental science. For an incompressible fluid, the velocity field at a point \(\mathbf{x} = (x_1, x_2, x_3) \in \mathbb{R}^3\) and time \(t\) is denoted by \(\mathbf{u}(\mathbf{x}, t) = (u_1(\mathbf{x}, t), u_2(\mathbf{x}, t), u_3(\mathbf{x}, t))\). The Euler equations, which govern the motion of an inviscid fluid, are expressed as:
\begin{equation}
	\label{eq:euler}
	\frac{\partial \mathbf{u}}{\partial t} + (\mathbf{u} \cdot \nabla) \mathbf{u} + \nabla P = \mathbf{F},
\end{equation}
where \(P(\mathbf{x}, t)\) is the pressure, and \(\mathbf{F}\) represents external forces, such as gravity, acting on the fluid. The fluid domain is typically bounded below by a rigid, horizontal bottom at a fixed depth and above by a free surface, described by the function $\eta(x_1, x_2, t)$. The problem is formulated as follows:
\begin{equation}
	\label{eq:water-wave}
	\begin{cases}
		\frac{\partial \mathbf{u}}{\partial t} + (\mathbf{u} \cdot \nabla) \mathbf{u} + \nabla P = \mathbf{F} & \text{in } \Omega,\\
		\nabla \cdot \mathbf{u} = 0 & \text{in } \Omega,
	\end{cases}
\end{equation}
where the domain is defined as:
\[
\Omega := \left\{ (x_1, x_2, x_3) \in \mathbb{R}^3 : -d < x_3 < \eta(x_1, x_2, t) \right\}.
\]

Assuming the flow is irrotational, the convective term in the Euler equations simplifies to \((\mathbf{u} \cdot \nabla) \mathbf{u} = \frac{1}{2} \nabla |\mathbf{u}|^2\). For gravity-capillary waves, the external force is given by \(\mathbf{F} = g \mathbf{e}_3\), where \(g\) represents the gravitational acceleration. Consequently, the momentum equation in \eqref{eq:water-wave} can be rewritten as:
\begin{equation*}
	\frac{\partial \mathbf{u}}{\partial t} + \nabla \left( \frac{1}{2} |\mathbf{u}|^2 + P + g x_3 \right) = 0.
\end{equation*}
On the free surface \(x_3 = \eta({\bf x}', t)\) with ${\bf x}'=(x_1,x_2)$, the kinematic boundary condition is:
\begin{equation}
	\label{eq:kinematic}
	\partial_t \eta = \sqrt{1 + (\partial_{x_1} \eta)^2 + (\partial_{x_2} \eta)^2} \, \mathbf{u} \cdot \mathbf{n},
\end{equation}
which reflects that the free surface evolves with the motion of the fluid. At the bottom boundary, the impermeability condition is:
\begin{equation}
	\label{eq:impermeability}
	\mathbf{u} \cdot \mathbf{n} = 0,
\end{equation}
where \(\mathbf{n}\) is the unit outward normal to the boundary of the fluid domain in \eqref{eq:kinematic} and \eqref{eq:impermeability}. The irrotationality condition implies that the velocity field is given by \(\mathbf{u} = \nabla \Phi\), where \(\Phi(\mathbf{x}, t)\) is the velocity potential. In terms of \(\Phi\), the system \eqref{eq:water-wave} can be expressed as:
\begin{equation}
	\label{eq:water-wave-phi}
	\begin{cases}
		\nabla \left( \partial_t \Phi + \frac{1}{2} |\nabla \Phi|^2 + P + g x_3 \right) = 0 & \text{in } \Omega, \\
		\Delta \Phi = 0 & \text{in } \Omega,
	\end{cases}
\end{equation}
where the second equation arises from the incompressibility condition. The first equation in \eqref{eq:water-wave-phi} implies:
\begin{equation}
	\label{eq:bernoulli}
	\partial_t \Phi + \frac{1}{2} |\nabla \Phi|^2 + P + g x_3 = C(t),
\end{equation}
where \(C(t)\) is a constant depending only on time. Evaluating \eqref{eq:bernoulli} at the free surface \(x_3 = \eta(x_1, x_2, t)\), we obtain the dynamic boundary condition, known as Bernoulli’s principle:
\begin{equation}
	\label{eq:dynamic}
	\partial_t \Phi + \frac{1}{2} |\nabla \Phi|^2 + P + g \eta = 0.
\end{equation}
When the pressure at the free surface is determined by surface tension, it is given by:
\[
P = \sigma \operatorname{div} \left( \frac{\nabla_{\perp} \eta}{\sqrt{1 + |\nabla_{\perp} \eta|^2}} \right),
\]
where \(\nabla_{\perp} = (\partial_{x_1}, \partial_{x_2})\) and $\sigma$ is the coefficient of the surface tension. Combining these, the boundary value problem for the gravity-capillary water-wave system is:
\begin{equation}
	\label{eq:water-wave-full}
	\begin{cases}
		\Delta \Phi = 0 & \text{in } \Omega, \\
		\partial_n \Phi = 0 & \text{on } x_3 = -d, \\
		\partial_t \eta = \partial_{x_3} \Phi - \nabla_{\perp} \eta \cdot \nabla_{\perp} \Phi & \text{on } x_3 = \eta(x_1, x_2, t), \\
		\partial_t \Phi = -\frac{1}{2} |\nabla \Phi|^2 - g \eta + \sigma \operatorname{div} \left( \frac{\nabla_{\perp} \eta}{\sqrt{1 + |\nabla_{\perp} \eta|^2}} \right) & \text{on } x_3 = \eta(x_1, x_2, t).
	\end{cases}
\end{equation}

An extensive body of research addresses the water wave problem, investigating topics such as well-posedness in the presence of vorticity, solitary wave and front-type solutions, the geometry and properties of the interface, wave regularity, stagnation points and critical layers of the wave, stability and instability, and other related areas. Given the extensive scope of the literature, a complete survey of all contributions is infeasible; we therefore point to a selection of works and refer readers to \cite{amick1982,Amick,bona1984,bri1995,lect,Buffoni,Buffoni2018,Buffoni2022,con-2011,con-2004,con-2010,con-2016,Craig4,Craig1,Craig3,Craig-2009,Cra2000,Craig2,Crag93,dar2003,davila2024,dias2003,Dj,galley2011,Groves2016,Groves,iooss2011,Ka,Ki,Lannes,Lannes2017,mach1993,mach1994,miles1957,mzz2012,mzz2012-1,mzz2013,Sachs,Sat,sun2002,w2009,wa2009,wheeler1,wheeler2,wheeler3,wheeler4,Whi,wu1,wu2,zz2007,zz2008} and the references therein for further details.  

The primary focus of this paper is on traveling wave solutions. For an in-depth review of recent developments in traveling water waves, we direct readers to \cite{H}. Our investigation concentrates on the 3D water wave problem, where the three-dimensional setting introduces a variety of geometric configurations, with the free surface displaying distinct behaviors in different horizontal directions. Two classes of traveling waves have attracted significant attention: doubly periodic waves, which are periodic in two independent horizontal directions, and fully localized waves, which decay in all horizontal directions. The doubly periodic case, often referred to as short-crested waves, was first rigorously analyzed by Reeder-Shinbrot \cite{Reed} using the Crandall-Rabinowitz bifurcation theorem in a fundamental domain of the ``symmetric diamond" type, where the lattice directions have equal length. Earlier formal computations for such waves were provided by Fuchs \cite{f1952} and Sretenskiĭ \cite{s1953}. Subsequently, Craig-Nicholls \cite{Cra2000} extended the analysis to arbitrary fundamental domains using a variational approach. Specifically, doubly periodic traveling waves are critical points of the energy functional
\begin{equation}
\label{1.energy-1}
E(\eta,\Phi)=\int_{\Omega_0}\left(\int_{-d}^{\eta}\frac12|\nabla\Phi|^2dx_3+\frac12g\eta^2+\sigma(\sqrt{1+|\nabla_\perp \eta|^2}-1)\right)dx_1dx_2,
\end{equation}
subject to the constraints
\begin{equation}
\label{1.constraint}
I_j(\eta,\Phi)=\int_{\Omega_0}\eta_{x_j}\Phi(x_1,x_2,\eta(x_1,x_2)) dx_1dx_2={\rm const},~j=1,2,
\end{equation}
where $\Omega_0$ denotes the fundamental domain of the lattice, $\eta$ is the free surface elevation, $\Phi$ is the velocity potential, (g) is the gravitational acceleration, and $\sigma$ is the surface tension coefficient. Since the velocity potential $\Phi$ is harmonic, the problem can be reformulated using the Dirichlet-to-Neumann operator (DNO) to reduce the dimensionality. Define the surface velocity potential as
$$\xi(x_1,x_2,t)=\Phi(x_1,x_2,\eta(x_1,x_2,t),t).$$ 
The DNO, denoted  $G(\eta)$, is defined by 
$$G(\eta)\xi=\partial_{x_3}\Phi-\nabla_{\perp}\eta\cdot\nabla_{\perp}\Phi,$$
where $\nabla_\perp = (\partial_{x_1}, \partial_{x_2})$ and the normal derivative is evaluated at the free surface. Craig and Groves \cite{Craig1} employed the DNO to derive long-wave models, with further developments in \cite{Bona, Craig2, Craig3}. For sufficiently small $C^1$-norm of $\eta$, the DNO is analytic and admits the expansion
\begin{equation}
\label{1.dno}
G(\eta)=G_0+G_1(\eta)+G_2(\eta)+R_3(\eta),
\end{equation}
where the leading terms are
\begin{equation}
\label{1.dno-1}
\begin{aligned}
&G_0=|D|\tanh(|D|),\\
&G_1(\eta)=-G_0\eta G_0-D\eta\cdot D,\\
&G_2(\eta)=-\frac12|D|^2\eta^2G_0-\frac12G_0\eta^2|D|^2+G_0\eta G_0\eta G_0,
\end{aligned}
\end{equation}
with $D = -i \nabla_\perp$ and $|D| = \sqrt{-\Delta_\perp}$. Using the DNO, the kinematic and dynamic boundary conditions of the water wave problem \eqref{eq:water-wave-full} reduce to
\begin{equation}
\label{eq:water-wave-bd}
\begin{cases}
\eta_t=G(\eta)\xi,\\
\xi_t=\frac{1}{2(1+|\nabla_{\perp}\eta|^2)}\left[(G(\eta)\xi+\nabla_{\perp}\eta\cdot\nabla_{\perp}\xi)^2-|\nabla_{\perp}\xi|^2|\nabla_\perp\eta|^2-|\nabla_\perp\xi|^2\right]\\
\quad\quad-g\eta+\sigma\operatorname{div} \left( \frac{\nabla_{\perp} \eta}{\sqrt{1 + |\nabla_{\perp} \eta|^2}} \right).
\end{cases}
\end{equation}
The energy functional \eqref{1.energy-1} can then be rewritten as
\begin{equation}
\label{1.energy-2}
\int_{\Omega_0}\left(\frac12\xi G(\eta)\xi+\frac12g\eta^2+\sigma(\sqrt{1+|\nabla_\perp \eta|^2}-1)\right)dx_1dx_2.
\end{equation}
The existence of doubly periodic waves has been demonstrated using spatial dynamics techniques. Groves and Mielke \cite{gm2001} confirmed their existence in the symmetric diamond configuration, while Groves and Haragus \cite{gh2003} and Nilsson \cite{n2019} generalized these findings to arbitrary fundamental domains. Advances in the study of doubly periodic water waves with vorticity have also emerged. Lokharu, Seth, and Wahlén \cite{lsw2020} provided foundational evidence for small-amplitude waves in this setting. Later, Groves et al. \cite{gnpw2024} investigated three-dimensional waves on Beltrami flows through a generalized Dirichlet-Neumann operator. Most recently, Seth, Varholm, and Wahlén \cite{svw2024} furthered these findings by confirming the existence of symmetric gravity-capillary waves with small vorticity, achieving a key breakthrough in the field.

For fully localized 3D solitary waves, where the free surface $\eta(x_1, x_2) \to 0$ as $|(x_1, x_2)| \to \infty$, spatial dynamics methods, successful in 2D, are less effective in 3D due to the presence of two unbounded variables. Instead, two major approaches are employed: variational methods and the implicit function theorem. As the periodic case, the Fully localised solitary waves are critical points of the energy functional
\begin{equation*}
\begin{aligned}
E(\eta,\xi)=~&\int_{\R^2}\left(\frac12\xi G(\eta)\xi+\frac12g\eta^2+\sigma(\sqrt{1+|\nabla\eta|^2}-1)\right)d{\bf x}'\\
=~&\int_{\Omega}\frac12|\nabla\phi|^2d{\bf x}+\int_{\R^2}\left(\frac12g\eta^2+\sigma(\sqrt{1+|\nabla\eta|^2}-1)\right)d{\bf x}',
\end{aligned}
\end{equation*}
subject to the momentum constraint in the $x_1$-direction,
\begin{equation*}
I_1(\eta,\xi):=\int_{\R^2}\eta_x\xi d{\bf x}'={\rm const}.
\end{equation*}
Groves-Sun constructed solutions using a mountain-pass approach for the augmented energy $E_c(\eta, \xi) = E(\eta, \xi) - c I_1(\eta, \xi),$ under the conditions $\frac{\sigma}{g d^2} > \frac{1}{3}$ and wave speed $c$ close to $\sqrt{g d}$. Buffoni-Groves-Sun-Wahl\'en \cite{Buffoni} employed a variational method minimizing $E(\eta, \xi)$ subject to a small momentum constraint, with further results for weak surface tension in \cite{Buffoni2018}.

The implicit function theorem offers an alternative approach for 3D solitary waves. The water wave problem is closely related to model equations such as the Davey-Stewartson system and the Kadomtsev-Petviashvili (KP) equation. For the gravity-capillary problem on infinite depth, the Davey-Stewartson system reduces to a 2D elliptic nonlinear Schr\"odinger equation, enabling Buffoni, Groves, and Wahlén \cite{Buffoni2022} to construct traveling wave solutions via elegant perturbation arguments. The connection between the KP equation and the water wave problem is explored in Ablowitz-Segur \cite{Abl2}, with additional details available in \cite[Section 6]{Craig1} and \cite{mzz2013}. To clarify this relationship, we consider the linearized boundary conditions
\begin{equation*}
\begin{cases}
\eta_t=\Phi_{x_3},\\
\Phi_t=-g\eta+\sigma\Delta\eta.
\end{cases}
\end{equation*}
at the free surface. Solving these with the Laplace equation and bottom boundary condition yields the dispersion relation
\[\omega^{2}=\left(  g\kappa+\sigma\kappa^{3}\right)  \tanh\left(  \kappa
d\right).\]
where $\kappa = \sqrt{k^2 + l^2}$ is the wave number. For shallow water ($\kappa h \ll 1$) and weak transverse variation ($\frac{l}{k} \ll 1$), the dispersion relation simplifies to
\[\omega^2 \approx g d \kappa^2 \left( 1 + \frac{\sigma}{g} \kappa^2 \right) \left( 1 - \frac{1}{3} (\kappa d)^2 \right).\]
Assuming $l^2 = O(k^4)$, the dispersion relation yields the linear part of the KP equation:
\[\frac{1}{c_0} \eta_{tx} + \eta_{xx} + \frac{1}{2} \eta_{yy} + \frac{d^2}{6} \left( 1 - \frac{3 \sigma}{g d^2} \right) \eta_{xxxx} = 0,\]
where $c_0 = \sqrt{g d}$. For $\frac{\sigma}{g d^2} > \frac{1}{3}$, this corresponds to the KP-I equation, supporting small-amplitude solitary waves. In fact, there is an explicit family of solutions with algebraic decay, first found by Manakov et al. \cite{Ma2}. The lump solutions have been shown to be nondegenerate and orbital stability by Liu-Wei \cite{Liu1}. While for $\frac{\sigma}{g d^2} < \frac{1}{3}$, it is KP-II, where such solutions are not expected. Further discussions are found in Lannes \cite{Lannes} and numerical studies by Părău, Vanden-Broeck and Cooker \cite{parau}.  Notably, in the context of the 2D water wave problem, the Korteweg-de Vries (KdV) equation, rather than the KP-I equation, is typically employed as an approximating model. For traveling wave solutions, Groves \cite{Groves2} utilized solutions to the stationary KdV equation to establish the existence of solutions for the 2D water wave problem when the surface tension parameter satisfies $\frac{\sigma}{g d^2} > \frac{1}{3}$. In this paper, we aim to provide a definitive resolution for the 3D water wave problem.

To streamline our analysis, we take the following normalization:
\begin{equation}
\label{1.parameter}
\mbox{Gravitational acceleration constant}~\ g = 1\mbox{ and the water depth}~\ d = 1.
\end{equation}
Under these assumptions, the corresponding  KP-I equation  takes the elliptic form
\begin{equation}
\label{1.KP}
\left(\sigma-\frac13\right)\partial_x^4q-\partial_x^2q-\partial_y^2q-\frac32\partial_x((\partial_xq)^2)=0,\quad \sigma>\frac13.
\end{equation}
In \cite{Liu1}, Liu-Wei proved that the lump solution of the equation \eqref{1.KP} is non-degenerated. Based on this finding, we aim to establish

\begin{theorem}\label{Thm1.1}
	Let $q(x, y)$ be a non-degenerate lump solution to the  KP-I equation \eqref{1.KP}. For any sufficiently small $\varepsilon>0$, there exists a solution pair $(\eta_{\ep}, \xi_{\ep})$ to the water wave problem \eqref{eq:water-wave-full} with traveling speed $c=\frac{1}{\sqrt{1+\ep^2}}$, satisfying the asymptotic expansions:
	\begin{align*}
		&\eta_{\ep}(x_1, x_2,t)=\ep^2 Q(\ep(x_1-ct),\ep^2x_2)+\mathcal{O}(\ep^3),\\
		&\xi_{\varepsilon}(x_1,x_2, t)= \varepsilon q(\ep(x_1-ct),\ep^2x_2)+\mathcal{O}(\varepsilon^2),
	\end{align*}
	where $Q(x,y) := \partial_x q(x,y)$ denotes the spatial derivative of the lump profile.
\end{theorem}

In the following, we briefly outline the main steps of this article. Initially, we employ the Dirichlet-to-Neumann operator to transform the full problem \eqref{eq:water-wave-full} into \eqref{eq:water-wave-bd}. We seek solutions to \eqref{eq:water-wave-bd} of the form
\[
\eta(x_1, x_2, t) = \varepsilon^2 h(\varepsilon(x_1 - ct), \varepsilon^2 x_2), ~ \xi(x_1, x_2, t) = \varepsilon f(\varepsilon(x_1 - ct), \varepsilon^2 x_2),~ c = \frac{1}{\sqrt{1 + \varepsilon^2}}.
\]
Substituting this ansatz into \eqref{eq:water-wave-bd}, we find that the leading-order term of the compatibility condition links the problem to the KP-I equation, which is similar to the traveling wave solutions of the Gross-Pitaevskii equation in \cite{Liu2}. However, as \cite{Liu2}, constructing solutions via the non-degeneracy of the lump solution and the classical Lyapunov-Schmidt reduction is not easy. Two challenges arise: first, perturbation terms complicate the decay properties and the nonlinear problem; second, the presence of the Dirichlet-to-Neumann operator (DNO) requires careful analysis of its leading-order term. Instead of handling a constant anisotropic differential operator as in the Gross-Pitaevskii equation, we address a more complicated operator:
\begin{align}
	\label{1.lin}
	\mathcal{L}_\varepsilon = & \, A \partial_1^4 - \partial_1^2 - (1 + \varepsilon^2) \partial_2^2 + \left(2A + \frac{1}{3}\right) \varepsilon^2 \partial_1^2 \partial_2^2 + \sigma (1 + \varepsilon^2) \varepsilon^4 \partial_2^4 \notag \\
	& + \varepsilon^{-4} (1 + \varepsilon^2) (1 + \varepsilon^2 \sigma \mathcal{P}) (\mathcal{Q} G_0 - \varepsilon^2 \mathcal{P}) - \frac{3}{c} \partial_1 (\partial_1 q \partial (\cdot)),
\end{align}
where
\begin{equation}
	\label{1.lin-f}
	A = (1 + \varepsilon^2) \sigma - \frac{1}{3}, \quad \mathcal{P} = -\partial_1^2 - \varepsilon^2 \partial_2^2, \quad \mathcal{Q} = 1 + \frac{1}{3} \varepsilon^2 \mathcal{P}, \quad G_0 = |D| \tanh(|D|).
\end{equation}
The term involving \( G_0 = |D| \tanh(|D|) \) poses significant challenges in studying the linearized problem. A Taylor expansion is typically used, but to preserve the elliptic nature of the differential operator in Proposition \ref{Prop4-1}, we restrict the surface tension constant \( \sigma \) to a narrower range. Instead, we utilize the continued fraction representation of the hyperbolic tangent:
\begin{equation}
	\label{1.tanh}
	\tanh(x) = \cfrac{1}{\frac{1}{x} + \cfrac{1}{\frac{3}{x} + \cfrac{1}{\frac{5}{x} + \cfrac{1}{\frac{7}{x} + \ddots}}}},
\end{equation}
which yields the inequality
\begin{equation}
	\label{1.tanh-ine}
	\frac{x (15 + x^2)}{15 + 6 x^2} > \tanh x > \frac{10x^3+105x}{x^4+45x^2 +105}.
\end{equation}
This inequality aids the analysis of the linearized operator in Proposition \ref{Prop4-1}. Having developed the \( L^2 \)-theory for the linearized problem, we analyze the linear problem in a suitable Sobolev space to tackle the nonlinear problem. In \cite{Liu2}, a crucial step for the traveling wave solutions of the Gross-Pitaevskii equation involves deriving the asymptotic behavior of the Green function for the linearized operator at infinity and near zero. However, these techniques are not directly applicable here. Instead, we apply the classical Fourier analysis results, including the Hausdorff-Young inequality and the Hardy-Littlewood-Paley theorem, to complete the linear theory within certain Sobolev spaces. Subsequently, we address the nonlinear system by proving that it is weakly decoupled, which allows us to reduce it into a single nonlinear problem. This problem includes several nonlocal operators stemming from the DNO operator. To handle these terms, we utilize results from Craig-Schanz-Sulem \cite{Craig2} regarding two types of singular integral operators:
\begin{align*}
	S_l(\eta_1, \dots, \eta_l) \xi({\bf x}') &= \int_{\mathbb{R}^2} K({\bf x}' - {\bf y}') \left( \prod_{j=1}^l P(\eta_j) \right) \xi({\bf y}') \, d{\bf y}', \\
	S_{l,*}(\eta, \dots, \eta) \xi({\bf x}') &= \int_{\mathbb{R}^2} K_*({\bf x}' - {\bf y}') \kappa_*^r Q^l(\eta) \xi({\bf y}') \, d{\bf y}',
\end{align*}
where \( K \) is a Calder\'on-Zygmund kernel satisfying standard estimates, and the functions \( P(\eta) \), \( Q(\eta) \), \( \kappa_*({\bf x}', {\bf y}') \), and \( K_* \) are defined as:
\[
P(\eta) = \frac{\eta({\bf x}') - \eta({\bf y}')}{|{\bf x}' - {\bf y}'|}, \quad Q(\eta) = \frac{\eta({\bf x}') + \eta({\bf y}')}{(|{\bf x}' - {\bf y}'|^2 + 4)^{\frac{1}{2}}}, \quad \kappa_*({\bf x}', {\bf y}') = \frac{4}{(|{\bf x}' - {\bf y}'|^2 + 4)^{\frac{1}{2}}},
\]
\[
K_*({\bf x}') = \frac{1}{(|{\bf x}'|^2 + 4)^{\rho/2}} \prod_{\ell=1}^2 \left( \frac{x_\ell}{(|{\bf x}'|^2 + 4)^{\frac{1}{2}}} \right)^{m_\ell} ~\text{for } l + \rho + r > 2~\mbox{and}~m_\ell\geq0,~\ell=1,2.
\]
Using Lemma \ref{LemmaA3} and Lemma \ref{LemmaA4}, we obtain refined estimates for the differential operators \( A(\eta) \) and \( B(\eta) \), which appear in the DNO as:
\[
(1 - B(\eta)) G(\eta) \cdot = |D| \tanh(|D|) \cdot + A(\eta) \cdot.
\]
With these results, we can manage the nonlinear perturbations and subsequently use a Banach fixed-point argument to establish the existence of solutions to \eqref{eq:water-wave-bd}.

We conclude the introduction with the following comments. First, the non-degeneracy of the lump solution $ q $ to equation \eqref{1.KP} and its Morse index of 1 are critical for analyzing the linearized operator. Using this non-degeneracy, we can address the linear problem in a functional space with appropriate symmetry, where the linearized operator is invertible. We strongly believe that the methods developed in this article can be extended to excited solutions of the KP-I equation. Recent work in \cite{lwy} has established the uniqueness of even solutions to the KP-I equation. If one can prove that any excited solution is non-degenerate and such solution has finite Morse index, then we can apply a similar argument as in Section 4 to develop the linear theory for the corresponding linearized operator. This could lead to the discovery of additional traveling wave solutions to the water wave problem. Second, in analyzing the linear theory, we must consider the nonlocal operator $ G(\eta) $. Employing the framework from Craig-Schanz-Sulem \cite{Craig2}, we obtain a priori estimates for the linearized operator of the KP-I lump solution, carrying the nonlocal term. This analysis holds independent interest. We suspect that Littlewood-Paley theory could provide pointwise estimates for the associated Green kernel and its derivatives, facilitating a more straightforward approach to studying the linear problem. We intend to pursue this in future research. Finally, compared to findings in \cite{Buffoni,Groves-Arma}, our method offers a more precise characterization of the asymptotic behavior of solutions as the traveling speed approaches the critical value. This could yield deeper insights into the spectrum, uniqueness, and dynamical properties of traveling wave solutions, which are valuable for further exploration.

The article is organized as follows: In Section 2, we reformulate the problem \eqref{eq:water-wave-bd} and establish its connection to the KP-I equation. In Section 3, we analyze the Dirichlet-to-Neumann operator and obtain refined estimates for related operators. In Section 4, we investigate the linear problem by studying the linearized KP-I operator with a nonlocal term around the lump solution. In Section 5, we solve the nonlinear problem using the contraction mapping principle, completing the proof of Theorem \ref{Thm1.1}. Some technical tools from Fourier analysis are provided in the Appendix.

\medskip
\begin{center}
{\bf  Notations}
\end{center}

\indent The following mathematical notations and conventions are adopted throughout this paper: 
\begin{itemize}
	\item \textbf{Spatial coordinates}: 
	The three-dimensional position vector is denoted by $\mathbf{x} = (x_1, x_2, x_3)$, where $x' = (x_1, x_2)$ represents the horizontal coordinates in the Cartesian plane. The scaled coordinates are defined as:
	\begin{align*}
		x = \varepsilon x_1, \quad 
		y = \varepsilon^2 x_2
	\end{align*}
	where $\varepsilon > 0$ is a dimensionless scaling parameter.
	
	\item \textbf{Differential operators}:
	\begin{align*}
		\partial_{{\bf x'}} &= (\partial_{x_1}, \partial_{x_2}) && \text{(horizontal gradient)} \\
		\nabla_{\varepsilon} &= (\partial_x, \varepsilon \partial_y) && \text{(scaled gradient operator)} \\
		D &= (-i\partial_{x_1}, -i\partial_{x_2})  && \text{(Fourier multiplier operator)}
	\end{align*}
	Here $\partial_x \equiv \partial_1$ and $\partial_y \equiv \partial_2$ denote partial derivatives with respect to the scaled coordinates $x$ and $y$ respectively. The magnitude satisfies $|D| = (D \cdot D)^{1/2} = (-\Delta_{{\bf x'}})^{1/2}$, where $\Delta_{{\bf x'}} = \partial_{x_1}^2 + \partial_{x_2}^2$ is the horizontal Laplacian.
	\item \textbf{Function spaces}:
	\begin{align*}
		\|\cdot\|_{L^{\infty}(\mathbb{R}^2; {\bf x'})} &:= \underset{{\bf x'} \in \mathbb{R}^2}{\mathrm{ess\,sup}}\, |\cdot| && \text{(supremum norm)} \\
		\|\cdot\|_{W^{k,p}(\mathbb{R}^2; {\bf x'})} &:= \left( \sum_{|\alpha| \leq k} \|D^\alpha \cdot\|^p_{L^p} \right)^{1/p} && \text{(Sobolev norm)}
	\end{align*}
	where $k \in \mathbb{N}_0$ is the order of weak derivatives, $1 \leq p \leq \infty$ is the integrability exponent, and differentiation is with respect to ${\bf x'}$ coordinates. The notation $\|\cdot\|_{L^{\infty}(\mathbb{R}^2)}$ and $\|\cdot\|_{W^{k,p}(\mathbb{R}^2)}$ implicitly refers to the standard Cartesian coordinates $(x, y)$.
	
	\item \textbf{Function dependencies}:
	$\eta = \eta(x_1, x_2)$ and $\xi = \xi(x_1, x_2)$ are scalar functions of horizontal coordinates, $\psi = \psi(x,y)$, $\phi = \phi(x,y)$, $f = f(x,y)$, $g = g(x,y)$ are functions defined on the scaled coordinate system.
\end{itemize}



\section{From water wave to KP-I }\label{sec2}

\subsection{Dirichlet-Neumann operators}
The Dirichlet-Neumann operator (DNO), which maps Dirichlet boundary data to the normal derivative of the solution to Laplace's equation, plays a fundamental role in water wave theory. 

Defining the surface velocity potential as $\xi(\mathbf{x'},t) = \Phi(\mathbf{x'}, \eta(\mathbf{x'},t),t)$, the DNO $G(\eta)$ is given by
\begin{equation*}
	G(\eta)\xi=\Phi_{x_3}-\partial_{{\bf x'}}\eta \cdot \partial_{{\bf x'}} \Phi.
\end{equation*}
Using geometric relations on the free surface
\begin{align*}
	\Phi_{x_3} - \eta_{x_1} \Phi_{x_1} - \eta_{x_2} \Phi_{x_2} &= G(\eta) \xi, \\
	\Phi_{x_1} + \eta_{x_1} \Phi_{x_3} &= \xi_{x_1}, \\
	\Phi_{x_2} + \eta_{x_2} \Phi_{x_3} &= \xi_{x_2},
\end{align*}
we solve for the gradient components
	\begin{align*}
		\Phi_{x_1} & =\frac{\left(1+\eta_{x_2}^2\right) \xi_{x_1}-\eta_{x_1} \eta_{x_2} \xi_{x_2}-\eta_{x_1} G(\eta) \xi}{1+\left|\partial_{{\bf x'}} \eta\right|^2} ,\\
		\Phi_{x_2} & =\frac{\left(1+\eta_{x_1}^2\right) \xi_{x_2}-\eta_{x_1} \eta_{x_2} \xi_{x_1}-\eta_{x_2} G(\eta) \xi}{1+\left|\partial_{{\bf x'}} \eta\right|^2} ,\\
		\Phi_{x_3} & =\frac{G(\eta) \xi+\partial_{{\bf x'}} \eta \cdot \partial_{{\bf x'}} \xi}{1+\left|\partial_{{\bf x'}} \eta\right|^2}.
	\end{align*}
This formulation yields the exact water wave equations
\begin{equation}\label{2-4}
	\begin{aligned}
		\eta_t  &=G(\eta) \xi, \\
		\xi_t  &=\frac{1}{2(1+|\partial_{{\bf x'}} \eta|^2)}\left[\left(G(\eta) \xi+\partial_{{\bf x'}}\eta\cdot \partial_{{\bf x'}}\xi\right)^2
		-\left|\partial_{{\bf x'}} \xi\right|^2\left|\partial_{{\bf x'}} \eta\right|^2 -\left|\partial_{{\bf x'}} \xi\right|^2\right]\\
		&\quad -g\eta+\sigma  \operatorname{div}\Big[\frac{\partial_{{\bf x'}} \eta}{\sqrt{1+|\partial_{{\bf x'}} \eta|^2}}\Big].
	\end{aligned}
\end{equation}

\subsection{Asymptotic derivation from water wave equations to KP-I}
In this subsection, we establish the asymptotic connection between the water wave system \eqref{2-4} and the KP-I equation \eqref{1.KP}. The derivation proceeds through the following steps:

Introduce a small parameter $\ep$ and define rescaled spatial variables
$$x=\ep x_1,\quad y=\ep^2 x_2.$$
Define the perturbation term $\Pi$ as
\begin{equation*}
	\begin{split}
		\ep^6 \Pi:&=\frac{1}{2}\xi_{x_1}^2+\frac{1}{2(1+|\partial_{{\bf x'}} \eta|^2)}\Big[\left(G(\eta) \xi+\partial_{{\bf x'}}\eta\cdot \partial_{{\bf x'}}\xi\right)^2
		-\left|\partial_{{\bf x'}} \xi\right|^2\left|\partial_{{\bf x'}} \eta\right|^2 -\left|\partial_{{\bf x'}} \xi\right|^2\Big] \\
		&\quad +\sigma  \Big(\mathrm{div}\Big[\frac{\partial_{{\bf x'}} \eta}{\sqrt{1+|\partial_{{\bf x'}}  \eta|^2}}\Big]-\Delta_{{\bf x'}}\eta\Big).
	\end{split}
\end{equation*}
We consider asymptotic solutions of the form
\begin{align}
\eta(x_1,x_2,t) &= \ep^2 h\big(\ep(x_1 - ct),\, \ep^2 x_2\big), \label{eta-ansatz} \\
\xi(x_1,x_2,t) &= \ep f\big(\ep(x_1 - ct),\, \ep^2 x_2\big), \label{xi-ansatz}
\end{align}
with wave speed $c = (1 + \ep^2)^{-1/2}$, gravitational acceleration $g=1$, 
and surface tension coefficient $\sigma > \frac{1}{3}$. Here $\ep > 0$ is a sufficiently small parameter.

Substituting \eqref{eta-ansatz} and \eqref{xi-ansatz} into \eqref{2-4} yields
\begin{align}
-c\partial_1h &= \ep^{-2}G(\ep^2h)f, \label{2-5} \\
-c\partial_1f &= -h - \tfrac{1}{2}\ep^2(\partial_1f)^2 + \sigma \ep^2(\partial_1^2 + \ep^2\partial_2^2)h + \ep^4 \Pi, \label{2-6}
\end{align}
where $\partial_1 \equiv \partial/\partial x$ and $\partial_2 \equiv \partial/\partial y$ denote derivatives with respect to the scaled variables.

Applying $\partial_1$ to \eqref{2-6} and combining with \eqref{2-5} yields
\begin{equation}\label{2-7}
	\begin{split}
		0 &= c^2\partial_1^2f + \left(1 - \sigma\ep^2(\partial_{1}^2 + \ep^2\partial_2^2)\right)\big(\ep^{-2}G(\ep^2h)f\big) \\ 
		&\quad - \frac{c}{2}\ep^2\partial_1(\partial_1f)^2 + c\ep^4\partial_1\Pi.
	\end{split}
\end{equation}
Define $\mathcal{Q} := 1 + \frac{1}{3}\ep^2\mathcal{P}$ where $\mathcal{P} := -\partial_1^2 - \ep^2\partial_2^2$. 
Applying $\mathcal{Q}$ to \eqref{2-7} gives
\begin{equation*}
    \begin{split}
        0 &= c^2\left(1 + \tfrac{1}{3}\ep^2\mathcal{P}\right)\partial_1^2f 
            + (1 + \sigma\ep^2\mathcal{P})\big(\ep^{-2}\mathcal{Q}G(\ep^2h)f\big) \\
          &\quad - \frac{c}{2}\ep^2\mathcal{Q}\partial_1(\partial_1f)^2 
            + c\ep^4\mathcal{Q}\partial_1\Pi.
    \end{split}
\end{equation*}
Using the wave speed $c^2 = \frac{1}{1+\ep^2}$, we derive
\begin{align*}
0 & = \left(1 + \tfrac{1}{3}\ep^2\mathcal{P}\right)\partial_1^2f 
    + (1+\ep^2)(1 + \sigma\ep^2\mathcal{P})\left(\ep^{-2}\mathcal{Q}G(\ep^2h)f\right) \\
  &\quad - \frac{\ep^2}{2c}\mathcal{Q}\partial_1(\partial_1f)^2 
    + \frac{\ep^4}{c}\mathcal{Q}\partial_1\Pi \\
& = (1+\ep^2)(1 + \sigma\ep^2\mathcal{P})\left[\ep^{-2}\mathcal{Q}(G(\ep^2h) - G_0)f\right]  + (1+\ep^2)(1 + \sigma\ep^2\mathcal{P})\left[\ep^{-2}\mathcal{Q}G_0f - \mathcal{P}f\right] \\
  &\quad + (1+\ep^2)(1 + \sigma\ep^2\mathcal{P})\mathcal{P}f 
    + \left(1 + \tfrac{1}{3}\ep^2\mathcal{P}\right)\partial_1^2f \\
  &\quad - \frac{\ep^2}{2c}\mathcal{Q}\partial_1(\partial_1f)^2 
    + \frac{\ep^4}{c}\mathcal{Q}\partial_1\Pi \\
& = (1+\ep^2)(1 + \sigma\ep^2\mathcal{P})\left[\ep^{-2}\mathcal{Q}(G(\ep^2h) - G_0)f\right] + (1+\ep^2)(1 + \sigma\ep^2\mathcal{P})\left[\ep^{-2}\mathcal{Q}G_0f - \mathcal{P}f\right] \\
  &\quad + \ep^2\Bigl[-\partial_1^2f - (1+\ep^2)\partial_2^2f 
        + \left(\sigma(1+\ep^2) - \tfrac{1}{3}\right)\partial_1^4 f  + \left(2\sigma(1+\ep^2) - \tfrac{1}{3}\right)\ep^2\partial_1^2\partial_2^2f \\
  &\quad\qquad + \sigma(1+\ep^2)\ep^4\partial_2^4 f \Bigr]  - \frac{\ep^2}{2c}\mathcal{Q}\partial_1(\partial_1f)^2 
    + \frac{\ep^4}{c}\mathcal{Q}\partial_1\Pi.
\end{align*}
Define the operators
\begin{equation*}
\begin{aligned}
\mathcal{L}_1 &:= A\partial_1^4 - \partial_1^2 - (1+\ep^2)\partial_2^2 
                + \left(2A + \tfrac{1}{3}\right)\ep^2\partial_1^2\partial_2^2 
                + \sigma(1+\ep^2)\ep^4\partial_2^4, \\  
\mathcal{L}_2 &:= \ep^{-4}(1+\ep^2)(1 + \ep^2\sigma\mathcal{P})
                (\mathcal{Q}G_0 - \ep^2\mathcal{P}), \\
A &:= \sigma(1+\ep^2) - \tfrac{1}{3} > 0.
\end{aligned}
\end{equation*}
Applying the expansion $G(\ep^2h)f = G_0f + G_1(\ep^2h)f + R_2(\ep^2 h)f$ yields
\begin{align*}
0 &= \mathcal{L}_{1}f + \mathcal{L}_2f 
     - \frac{1}{2c}\mathcal{Q}\partial_1(\partial_1f)^2 
     + \frac{\ep^2}{c}\mathcal{Q}\partial_1\Pi \\
  &\quad + \ep^{-2}(1+\ep^2)(1 + \sigma\ep^2\mathcal{P})
        \left[\ep^{-2}\mathcal{Q}(G(\ep^2h) - G_0)f\right] \\
&= \mathcal{L}_{1}f + \mathcal{L}_2f 
     - \frac{1}{2c}\mathcal{Q}\partial_1(\partial_1f)^2 
     + \frac{\ep^2}{c}\mathcal{Q}\partial_1\Pi \\
  &\quad + \ep^{-4}(1+\ep^2)(1 + \sigma\ep^2\mathcal{P})\mathcal{Q}R_2(\ep^2h)f  + \ep^{-2}(1+\ep^2)(1 + \sigma\ep^2\mathcal{P})\mathcal{Q} G_1(h)f.
\end{align*}
From \eqref{2-6}, we calculate
\begin{align*}
& \ep^{-2}(1+\ep^2)(1+\sigma \ep^2 \mathcal{P})\mathcal{Q} G_1(h)f \\
&= \ep^{-2}(1+\ep^2)\Bigl(-G_0 h G_0 f - \ep^4\partial_2(h\partial_2f) - \ep^2\partial_1(h\partial_1f) \Bigr) \\
&\quad + (1+\ep^2)\Bigl[ \left(\tfrac{1}{3} + \sigma\right) \mathcal{P} G_1(h)f + \tfrac{\sigma}{3}\ep^2\mathcal{P}^2 G_1(h)f \Bigr] \\
&= -\ep^{-2}(1+\ep^2)\Bigl(G_0 h G_0 f + \ep^4\partial_2(h\partial_2f)\Bigr)  + (1+\ep^2)\Bigl[ \left(\tfrac{1}{3} + \sigma\right) \mathcal{P} G_1(h)f + \tfrac{\sigma}{3}\ep^2\mathcal{P}^2 G_1(h)f \Bigr] \\
&\quad - (1+\ep^2)\Bigl( c\partial_1(\partial_1f)^2 - \tfrac{\ep^2}{2}\partial_1 ((\partial_1f)^3) + \ep^4 \partial_1 ((\partial_1f)\Pi)  - \sigma\ep^2 \partial_1((\partial_1f)\mathcal{P}h) \Bigr).
\end{align*}
This leads to
\begin{align*}
& \mathcal{L}_{1}f + \mathcal{L}_2f - \frac{1}{2c}\mathcal{Q}\partial_1(\partial_1f)^2 \\
&= -\ep^{-4}(1+\ep^2)(1+\sigma \ep^2 \mathcal{P})\mathcal{Q}R_2(\ep^2h)f  + \ep^{-2}(1+\ep^2)\Bigl(G_0 h G_0 f + \ep^4\partial_2(h\partial_2f)\Bigr) \\
&\quad - (1+\ep^2)\Bigl[ \left(\tfrac{1}{3} + \sigma\right)\mathcal{P}G_1(h)f + \tfrac{\sigma}{3}\ep^2\mathcal{P}^2G_1(h)f \Bigr] - \frac{\ep^2}{c}\mathcal{Q}\partial_1\Pi \\
&\quad + (1+\ep^2)\Bigl( c\partial_1(\partial_1f)^2 - \tfrac{\ep^2}{2}\partial_1 ((\partial_1f)^3) + \ep^4 \partial_1 ((\partial_1f)\Pi)  - \sigma\ep^2 \partial_1((\partial_1f)\mathcal{P}h) \Bigr),
\end{align*}
where the nonlinear term expands as
\begin{align}
\frac{1}{2c}\mathcal{Q}\partial_1(\partial_1f)^2 
	&= \frac{1}{2c}\partial_1(\partial_1f)^2 + \frac{\ep^2}{6c}\mathcal{P}\partial_1(\partial_1f)^2, \label{nonlinear-expansion}
\end{align}
the remainder terms satisfy
\begin{align}
& \ep^{-4}(1+\sigma \ep^2\mathcal{P})(\mathcal{Q}{R}_2(\ep^2h)f) + \tfrac{1}{3}\sigma\ep^2\mathcal{P}^2G_1(h)f \nonumber \\
&= \ep^{-4}G_2(\ep^2h)f + \ep^{-4}R_3(\ep^2h)f + \ep^{-2}\left(\sigma + \tfrac{1}{3}\right)\mathcal{P}G_2(\ep^2 h)f \nonumber \\
&\quad + \ep^{-2}\left(\sigma + \tfrac{1}{3}\right)\mathcal{P}{R}_3(\ep^2 h)f + \tfrac{1}{3}\sigma \mathcal{P}^2{R}_1(\ep^2h)f, \label{remainder-expansion}
\end{align}
and the principal operator acts as
\begin{align}
-\mathcal{P}G_1(h)f 
	&= \mathcal{P}\Bigl( G_0(hG_0f) + \ep^2\partial_1(h\partial_1f) + \ep^4\partial_2(h\partial_2 f) \Bigr). \label{principal-action}
\end{align}
Finally, we obtain the asymptotic equation
\begin{equation}\label{2-8}
	\mathcal{L}_{1}f + \mathcal{L}_2f - \frac{3}{2c}\partial_1(\partial_1f)^2 = \sum_{i=1}^4 P_i,
\end{equation}
where the perturbation terms are defined as follows
\begin{align*}
P_1 &:= -\frac{\ep^2}{c}\partial_1(\mathcal{Q}\Pi) 
        - \frac{\ep^2}{2c^2}\partial_1((\partial_1f)^3)
        - \frac{\sigma \ep^2}{c^2}\partial_1\bigl((\partial_1f)(\mathcal{P}h)\bigr)
        + \frac{\ep^4}{c^2}\partial_1\bigl((\partial_1f)\Pi\bigr) \\
    &\quad + \frac{\ep^2}{6c}\mathcal{P}\partial_1(\partial_1f)^2 
        + \frac{\ep^2}{c^2}\left(\tfrac{1}{3} + \sigma\right)\mathcal{P}\partial_1(h\partial_1f), \\
P_2 &:= \frac{\ep^2}{c^2}\partial_2(h\partial_2f) 
        + \frac{\ep^4}{c^2}\left(\tfrac{1}{3} + \sigma\right)\mathcal{P}\partial_2(h\partial_2f), \\
P_3 &:= -\frac{\ep^{-4}}{c^2}G_2(\ep^2h)f 
        - \frac{\ep^{-4}}{c^2}R_3(\ep^2h)f 
        - \frac{\ep^{-2}}{c^2}\left(\sigma + \tfrac{1}{3}\right)\mathcal{P}G_2(\ep^2 h)f \\
    &\quad - \frac{\ep^{-2}}{c^2}\left(\sigma + \tfrac{1}{3}\right)\mathcal{P}{R}_3(\ep^2 h)f 
        - \frac{1}{3c^2}\sigma \mathcal{P}^2{R}_1(\ep^2h)f, \\
P_4 &:= \frac{\ep^{-2}}{c^2}(G_0 h G_0 f) 
        + \frac{1}{c^2}\left(\tfrac{1}{3} + \sigma\right)\mathcal{P}G_0hG_0f.
\end{align*}

\begin{remark}
	The systematic analysis of Dirichlet-to-Neumann operator (DNO) estimation will be presented in Section 3. Specifically:
	\begin{itemize}
		\item Subsections 3.1 and 3.2 establish spectral decompositions of the principal operators $G_0$, $G_1$, and $G_2$ using Fourier analysis techniques
		\item Subsection 3.3 characterizes the Sobolev regularity properties of the residual operator $R_3$
	\end{itemize}
\end{remark}

Regarding the left-hand side of \eqref{2-8}, we seek a solution of the form \(f = q + \phi\), where \(q = q_{\varepsilon}\) is a lump solution of the KP-I equation
\begin{align*}  
A\partial_x^4 q - \partial_x^2 q - (1 + \varepsilon^2)\partial_y^2 q - \frac{3}{2c}\partial_x((\partial_x q)^2) = 0.  
\end{align*}
Throughout the proof, we use the standard lump solution  
\begin{align}  
q_{\varepsilon}(x, y) = -\frac{8\sqrt{1+\varepsilon^2}A  x}{y^2 + (1+\varepsilon^2)(x^2 + 3A)}, \label{2-10}  
\end{align}  
whose nondegeneracy is established in \cite{Liu1}.  

The unique negative eigenvalue and corresponding eigenfunction for \(q_{\varepsilon}\) are closely approximated by those of the limiting equation (as \(\varepsilon \to 0\)):  
\begin{align}
	(\sigma-\tfrac{1}{3})\partial_x^4 q- \partial_x^2 q- \partial_y^2 q-\frac{3}{2} \partial_x((\partial_x q)^2)=0,\label{2-11}
\end{align} 
with solution \(q_0(x, y) = -\dfrac{8\left(\sigma - \frac{1}{3}\right)x}{x^2 + y^2 + 3\sigma - 1}\). Consequently, the linearized problem for \eqref{2-8} is well-approximated by that of the KP-I equation.

\section{Methodological adaptation for DNO operator estimation}\label{secA}
In this section, we will estimate the Dirichlet-to-Neumann (DNO) operator. The primary methodology draws inspiration from Craig's foundational framework \cite{Craig2}, with targeted modifications to enhance compatibility with subsequent nonlinear term analysis. Specifically, we refine the original operator estimation scheme by introducing a norm adaptation strategy, enabling higher flexibility in bounding conditions while addressing the inherent complexity trade-offs. This adaptation ensures that the operators $G_0, G_1, G_2$ and $R_3$
satisfy the regularity requirements for the nonlinear regime in Section \ref{sec5}.

To facilitate the subsequent proofs in this section, we first introduce the description of the Dirichlet-Neumann operator (DNO) from \cite{Craig2, Schanz}.  Let ${\bf x}', {\bf y}' \in \mathbb R^2$, and consider functions $\eta: \mathbb R^2 \to \mathbb R$  and $\xi: \mathbb R^2 \to \mathbb R$.
An exact implicit formula for the Dirichlet-Neumann operator is given by
\begin{align*}
	(1-B(\eta)) G(\eta) \xi=|D| \tanh (|D|) \xi + A(\eta) \xi,
\end{align*}
where
\begin{align*}
		A(\eta)\xi&=-|D|(1+e^{-2|D|})^{-1}\int_{\mathbb R^2} m({\bf x}' ,{\bf y}' ) \xi({\bf y}' ) d {\bf y}' \equiv (A_1+\ldots+A_L+A_{L+1}^R) \xi,\\
		B(\eta)\xi &=-|D|(1+e^{-2|D|})^{-1} \int_{\mathbb R^2}l({\bf x}' ,{\bf y}' ) \xi({\bf y}' ) d {\bf y}'  \equiv (B_1+\ldots+B_L+B_{L+1}^R) \xi,
\end{align*}
with
	\begin{align*}
		m({\bf x}' ,{\bf y}' )&=\frac{1}{2 \pi} \frac{-({\bf x}'-{\bf y}') \cdot \partial_{{\bf y}'} P}{|{\bf x}'-{\bf y}'|^2} \frac{1}{(1+P^2)^{3 / 2}} +\frac{1}{2 \pi}\Big[\frac{({\bf x}'-{\bf y}') \cdot \partial_{{\bf y}' } Q}{|{\bf x}' -{\bf y}' |^2+4}+\frac{4  Q}{(|{\bf x}' -{\bf y}' |^2+4)^2}\Big]\\
		&\quad\times \frac{1}{(1+\frac{4}{(|{\bf x}' -{\bf y}' |^2+4)^{1 / 2}} Q+Q^2)^{3 / 2}} \\ 
		&\quad+\frac{1}{2\pi} \frac{2}{(|{\bf x}' -{\bf y}' |^2+4 )^{3 / 2}}\Big[\frac{1}{(1+\frac{4}{(\left|{\bf x}' -{\bf y}' \right|^2+4)^{1 / 2}} Q+Q^2)^{3 / 2}}-1\Big],\\
		l({\bf x}',{\bf y}')&=\frac{1}{2 \pi} \frac{1}{|{\bf x}' -{\bf y}' |}\Big(\frac{1}{(1+P^2)^{1 / 2}}-1\Big)\\
		&\quad+\frac{1}{2 \pi} \frac{1}{(\left|{\bf x}' -{\bf y}' \right|^2+4 )^{1 / 2}}\Big[\frac{1}{(1+\frac{4 }{(|{\bf x}' -{\bf y}' |^2+4)^{1 / 2}} Q+Q^2)^{1 / 2}}-1\Big],
	\end{align*}
where
\begin{align*}
	P(\eta)=\frac{\eta\left({\bf x}'\right)-\eta\left({\bf y}'\right)}{\left|{\bf x}'-{\bf y}'\right|}, \quad  Q(\eta)=\frac{\eta\left({\bf x}'\right)+\eta\left({\bf y}'\right)}{(\left|{\bf x}'-{\bf y}'\right|^2+4 )^{1 / 2}},  \quad \kappa_h\left({\bf x}', {\bf y}'\right)=\frac{4}{(\left|{\bf x}'-{\bf y}'\right|^2+4 )^{1 / 2}}, 
\end{align*}
and 
	\begin{align*}
		A_l\xi&=-|D|(1+e^{-2|D|})^{-1} \frac{1}{2 \pi} \\
		&\quad \int_{\mathbb R^2}\Big[\frac{({\bf x}' -{\bf y}' ) \cdot \partial_{{\bf y}' } \tilde{p}_l(P)}{|{\bf x}' -{\bf y}' |^2}+\Big(\frac{({\bf x}' -{\bf y}' )\cdot \partial_{{\bf y}' }Q}{|{\bf x}' -{\bf y}' |^2+4}+\frac{4Q}{(|{\bf x}' -{\bf y}' |^2+4)^2}\Big)\Tilde{q}_{l-1}(Q,\kappa_h)\\
		&\quad+\frac{2}{(|{\bf x}' -{\bf y}' |^2+4)^{3/2}}\Tilde{q}_l(Q,\kappa_h)\Big]\xi({\bf y}' )d{\bf y}' ,\\
		A_{L+1}^R\xi&=-|D|(1+e^{-2|D|})^{-1} \frac{1}{2 \pi} \\
		&\quad \int_{\mathbb R^2}\Big[\frac{({\bf x}' -{\bf y}' ) \cdot \partial_{{\bf y}' } \tilde{p}^R_{L+1}(P)}{|{\bf x}' -{\bf y}' |^2}+\Big(\frac{({\bf x}' -{\bf y}' )\cdot \partial_{{\bf y}' }Q}{|{\bf x}' -{\bf y}' |^2+4}+\frac{4Q}{(|{\bf x}' -{\bf y}' |^2+4)^2}\Big)\Tilde{q}^R_{L}(Q,\kappa_h)\\
		&\quad+\frac{2}{(|{\bf x}' -{\bf y}' |^2+4)^{3/2}}\Tilde{q}^{R}_{L+1}(Q,\kappa_h)\Big]\xi({\bf y}' )d{\bf y}' ,\\
		B_l \xi&=-|D|(1+e^{-2|D|})^{-1} \frac{1}{2 \pi}\int_{\mathbb R^2}\Big(\frac{p_l(P)}{\left|{\bf x}' -{\bf y}' \right|}+\frac{q_l\left(Q, \kappa_h\right)}{(|{\bf x}' -{\bf y}' |^2+4 )^{1 / 2}}\Big) \xi d {\bf y}' ,\\
		B_{L+1}^R\xi&=-|D|(1+e^{-2|D|})^{-1} \frac{1}{2 \pi}\int_{\mathbb R^2}\Big(\frac{p_{L+1}^R(P)}{|{\bf x}' -{\bf y}' |}+\frac{q_{L+1}^R(Q, \kappa_h)}{(|{\bf x}' -{\bf y}' |^2+4 )^{1 / 2}}\Big) \xi d {\bf y}' .
	\end{align*}
Here, ${p}_l (\text{for } l \geq 2 \text{ for } l \text{ even})$  and $q_l(\text{for } l \geq 1)$ are homogeneous polynomials of order $l$ in $\eta$,  derived from the expansion of
$$
\frac{1}{(1+\sigma)^{1 / 2}}-1=-\frac{1}{2} \sigma+\frac{3}{8} \sigma^2-\ldots
$$
where $\sigma$ corresponds to $P^2$ and $\kappa_{h} Q+Q^2$, respectively. Similarly, $\tilde{p}_l(\text{for } l \geq 1 \text{ for } l \text{ odd})$ and $\tilde{q}_l(\text{for } l \geq 1)$ are homogeneous polynomials of order $l$ in $\eta$, with $\tilde{p}_l$
additionally involving derivatives of $\eta$. 
The remainders of these expansions, denoted by
$p_{L+1}^R, q_{L+1}^R, \tilde{p}_{L+1}^R$ and $\tilde{q}_{L+1}^R$, are expressions of order  $\mathcal{O}\left(\eta^{L+1}\right)$ for small $\eta$.

\begin{remark}\label{LemmaA1}
	We observe that the operators \((1 + e^{-2|D|})^{-1}\) and the Riesz potential \(R_j(D) = i D_j / |D|\) in the previously defined \(A_l\) and \(B_l\) are bounded on the Sobolev space \(W^{s,q}(\mathbb{R}^n)\) for all \(1 < q < \infty\) and non-negative integers \(s \geq 0\). This boundedness follows from the Mikhlin multiplier theorem (see Appendix A for details), as both symbols satisfy the Mikhlin condition:
	\[
	\sup_{\xi \in \mathbb{R}^n \setminus \{0\}} \big| \xi^{|\alpha|} \partial^\alpha \sigma(\xi) \big| \leq C_\alpha, \quad \forall |\alpha| \leq 2,
	\]
	where \(\sigma\) denotes the symbol of either operator. This property is particularly advantageous for subsequent a priori estimates in our analysis, as it ensures uniform control of higher-order derivatives.
\end{remark}

Since the most fundamental part of the DNO operator estimation involves singular integral operators, we first define them as follows. Let two operators be defined as:
\begin{equation*}
	\begin{split}
		S_l(\eta_1, \ldots, \eta_l) \xi({\bf x}')&=\int_{\mathbb R^2} K({\bf x}'-{\bf y}')\Big(\prod_{j=1}^l P(\eta_j)\Big) \xi({\bf y}') d {\bf y}',\\
		S_{l, *}(\eta, \ldots, \eta)&=\int_{\mathbb R^2} K_*\left({\bf x}'-{\bf y}'\right) \kappa_h^{r} Q^l(\eta) \xi({\bf y}') d {\bf y}'
	\end{split}
\end{equation*}
where $K$ is a Calder\'on–Zygmund kernel satisfying  standard estimates, and  $K_*$ takes the form:
\begin{equation*}
	K_*({\bf x}')=\frac{1}{(|{\bf x}'|^2+4)^{\rho/ 2}} \prod_{l=1}^2\Big(\frac{x_l}{(|{\bf x}'|^2+4)^{1 / 2}}\Big)^{m_l}, \text{ for } l+\rho+r>2.
\end{equation*}

To establish the desired estimates, we require the following lemmas, which can be found in \cite{Craig2, Schanz}. In what follows, the constants $C$, which appear below, are related only to $s$, and are independent of $\eta$ and $\xi$.

\begin{lemma}\label{LemmaA3}
	Suppose that $\eta \in C^{s+1}\left(\mathbb{R}^2\right)$, and consider the mapping properties of the operator $S_l(\eta, \cdots, \eta)$. Let $1<q<\infty$, then we have the estimate that
	$$
	\left\|\partial_{x}^s S_l(\eta, \cdots, \eta) \xi\right\|_q \leq C l^{M+s}\|\eta\|_{C^1}^{l-1}\big(\|\eta\|_{C^1}\left\|\partial_{x}^s \xi\right\|_q+\|\eta\|_{C^{s+1}}\|\xi\|_q\big) .
	$$
\end{lemma}

\begin{lemma}\label{LemmaA4}
	Suppose $l+\rho+r>2$, then we have the estimates
	$$
	\|\partial_{x}^s S_{l, *}(\eta, \ldots, \eta) \xi\|_q \leq C l^s\|\eta\|_{L^{\infty}}^{l-1}\|\eta\|_{C^s}\|\xi\|_q.
	$$
\end{lemma}

\subsection{Estimates for the operator $B_l$}
\label{subsec:B_l_estimates}

We now establish refined estimates for the operators $B_l$, which are crucial for our subsequent analysis.

\begin{theorem}\label{ThmA2}
    Let $1 < q < \infty$ with $\frac{1}{q} = \frac{1}{q_1} + \frac{1}{q_2}$, $s \in \mathbb{Z}_{\geq 0}$, and $\|\eta\|_{C^1} < \frac{1}{4}$. For integers $l \geq 1$ and $L \geq 1$, the operators $B_l$ and $B_L^R$ satisfy:
    \begin{align*}
        \big\||D|^{-1}B_1(\eta)\xi\big\|_{s,q} 
            &\leq C\|\eta\|_{s,q_1}\|\xi\|_{q_2}, \\
        \big\||D|^{-1}B_l(\eta)\xi\big\|_{s,q} 
            &\leq C\, 2^l l^{M+s} \|\eta\|_{s,q_1} \|\eta\|_{C^1}^{l-2}
            \Big(\|\eta\|_{C^1}\|\xi\|_{s,q_2} + \|\eta\|_{C^{s+1}}\|\xi\|_{q_2}\Big), 
            \quad l \geq 2,
    \end{align*}
    and for the remainder term,
    \begin{align*}
        \big\||D|^{-1}B_L^R(\eta)\xi\big\|_{s,q} 
            &\leq C\, 2^L L^{M+s} \|\eta\|_{s,q_1} \|\eta\|_{C^1}^{L-2}
            \Big(\|\eta\|_{C^1}\|\xi\|_{s,q_2} + \|\eta\|_{C^{s+1}}\|\xi\|_{q_2}\Big), 
            \quad L \geq 2.
    \end{align*}
    Furthermore, the full operator satisfies
    \[
        \big\||D|^{-1}B(\eta)\xi\big\|_{s,q}  
            \leq C\|\eta\|_{s,q_1}\Big(\|\xi\|_{q_2} + \|\eta\|_{C^1}\|\xi\|_{s,q_2} + \|\eta\|_{C^{s+1}}\|\xi\|_{q_2}\Big).
    \]
\end{theorem}

\begin{proof}
Recall the definitions of the operators:
\begin{align*}
    |D|^{-1}B(\eta)\xi(\mathbf{x}') 
    &= -(1 + e^{-2|D|})^{-1} \frac{1}{2\pi} \int_{\mathbb{R}^2} \frac{1}{|\mathbf{x}' - \mathbf{y}'|} \left( \frac{1}{(1 + P^2)^{1/2}} - 1 \right) \xi(\mathbf{y}')  d\mathbf{y}' \\
    &\quad - (1 + e^{-2|D|})^{-1} \frac{1}{2\pi} \int_{\mathbb{R}^2} \frac{1}{(|\mathbf{x}' - \mathbf{y}'|^2 + 4)^{1/2}} \left( \frac{1}{(1 + \kappa_h Q + Q^2)^{1/2}} - 1 \right) \xi(\mathbf{y}')  d\mathbf{y}', \\[2ex]
    |D|^{-1}B_1(\eta)\xi(\mathbf{x}') 
    &= (1 + e^{-2|D|})^{-1} \frac{1}{\pi} \int_{\mathbb{R}^2} \frac{\eta(\mathbf{x}') + \eta(\mathbf{y}')}{(|\mathbf{x}' - \mathbf{y}'|^2 + 4)^{3/2}} \xi(\mathbf{y}')  d\mathbf{y}', \\[2ex]
    |D|^{-1}B_l(\eta)\xi(\mathbf{x}') 
    &= -\big(1 + e^{-2|D|}\big)^{-1} \frac{1}{2\pi} \int_{\mathbb{R}^2} \left( \frac{p_l(P)}{|\mathbf{x}' - \mathbf{y}'|} + \frac{q_l(Q, \kappa_h)}{(|\mathbf{x}' - \mathbf{y}'|^2 + 4)^{1/2}} \right) \xi(\mathbf{y}')  d\mathbf{y}' \quad \text{for} \quad l \geq 2,
\end{align*}
where $P = P(\eta)$ and $Q = Q(\eta)$ are auxiliary functions depending on $\eta$. The series expansions are given by
$$
	\frac{1}{(1+P^2)^{1/2}}-1=\sum_{\substack{l \geq 2 \\ l \text{ even}}}^{\infty}p_{l}(P),\quad p_l(P)=a_lP^{l},\quad a_{l}=\Big(\frac{-1}{4}\Big)^{l/2}\binom{l}{l/2},
	$$
	and
	\begin{align*}
		&\frac{1}{(1+\kappa_h Q+Q^2)^{1/2}}-1=\sum_{l=1}^{\infty}q_{l}(Q,\kappa_h),\\
		&q_{l}(Q,\kappa_h)=\sum_{j=0}^{[l/2]}b_{l-j}\kappa_h^{l-2j}Q^l,    	\quad b_{l-j}=\Big(\frac{-1}{4}\Big)^{l-j}\binom{2l-2j}{l-j}\binom{l-j}{j}.
	\end{align*}
By Remark \ref{LemmaA1}, it suffices to analyze the following two types of integral operators in $B_l$.

\vspace{0.5em}
\noindent\textbf{Case 1: Operator $B_1(\eta)$.} 
Consider the convolution kernel $M(\mathbf{x}') = (|\mathbf{x}'|^2 + 4)^{-3/2}$. We estimate
\begin{align*}
    \big\||D|^{-1}B_1(\eta)\xi\big\|_{s,q} 
    &\leq C \big\| \eta (M * \xi) + M * (\eta \xi) \big\|_{s,q} \\
    &\leq C \|\eta\|_{s,q_1} \|M * \xi\|_{s,q_2} + \|M * (\eta\xi)\|_{s,q} \\
    &\leq C \big( \|\eta\|_{s,q_1}\|\xi\|_{q_2} + \|\eta\|_{q_1}\|\xi\|_{q_2} \big) \\
    &\leq C \|\eta\|_{s,q_1} \|\xi\|_{q_2},
\end{align*}
where the last inequality follows from the embedding $\|\eta\|_{q_1} \leq C \|\eta\|_{s,q_1}$ since $s \geq 0$.

\vspace{0.5em}
\noindent\textbf{Case 2: Operators $B_l(\eta)$ for $l \geq 2$.}
We decompose the operator into two components:
\begin{align*}
    P_l(\eta)\xi(\mathbf{x}') 
    &:= \int_{\mathbb{R}^2} \frac{1}{|\mathbf{x}' - \mathbf{y}'|} p_l(P) \xi(\mathbf{y}')  d\mathbf{y}' \\
    &= a_l \eta(\mathbf{x}') \int_{\mathbb{R}^2} \frac{1}{|\mathbf{x}' - \mathbf{y}'|^2} P^{l-1} \xi(\mathbf{y}')  d\mathbf{y}'  - a_l \int_{\mathbb{R}^2} \frac{1}{|\mathbf{x}' - \mathbf{y}'|^2} P^{l-1} \eta(\mathbf{y}')\xi(\mathbf{y}')  d\mathbf{y}'.
\end{align*}
This corresponds to the sum $a_l\eta S_{l-1}\xi + a_l S_{l-1}(\eta\xi)$ with kernel $|\mathbf{x}' - \mathbf{y}'|^{-2}$. Applying Lemma \ref{LemmaA3} and Stirling's approximation $\binom{2n}{n} \sim 2^{2n}/\sqrt{\pi n}$ yields
\begin{align*}
    \|P_l(\eta)\xi\|_{s,q} 
    &\leq \|\eta S_{l-1}\xi\|_{s,q} + \|S_{l-1}(\eta\xi)\|_{s,q} \\
    &\leq C \|\eta\|_{s,q_1} \|S_{l-1}\xi\|_{s,q_2} + C l^{M+s} \|\eta\|_{C^1}^{l-2} \Big( \|\eta\|_{C^1} \|\eta \xi\|_{s,q} + \|\eta\|_{C^{s+1}} \|\eta \xi\|_q \Big) \\
    &\leq C l^{M+s} \|\eta\|_{s,q_1} \|\eta\|_{C^1}^{l-2} \Big( \|\eta\|_{C^1} \|\xi\|_{s,q_2} + \|\eta\|_{C^{s+1}} \|\xi\|_{q_2} \Big).
\end{align*}
Similarly, for the $Q$-component:
\begin{align*}
    Q_{l,*}(\eta)\xi(\mathbf{x}') 
    &:= \int_{\mathbb{R}^2} \frac{1}{(|\mathbf{x}' - \mathbf{y}'|^2 + 4)^{1/2}} q_l(Q, \kappa_h) \xi(\mathbf{y}')  d\mathbf{y}' \\
    &= \sum_{j=0}^{\lfloor l/2 \rfloor} b_{l-j} \eta(\mathbf{x}') \int_{\mathbb{R}^2} \frac{1}{|\mathbf{x}' - \mathbf{y}'|^2 + 4} \kappa_h^{l-2j} Q^{l-1} \xi(\mathbf{y}')  d\mathbf{y}' \\
    &\quad + \sum_{j=0}^{\lfloor l/2 \rfloor} b_{l-j} \int_{\mathbb{R}^2} \frac{1}{|\mathbf{x}' - \mathbf{y}'|^2 + 4} \kappa_h^{l-2j} Q^{l-1} \eta(\mathbf{y}')\xi(\mathbf{y}')  d\mathbf{y}'.
\end{align*}
This is a sum of operators $\eta S_{l-1,*}\xi$ and $S_{l-1,*}(\eta\xi)$. Applying Lemma \ref{LemmaA4} for $l \geq 2$ gives
\begin{align*}
    \|Q_{l,*}(\eta) \xi\|_{s,q} 
    &\leq 2^l \Big( \|\eta S_{l-1,*}\xi\|_{s,q} + \|S_{l-1,*}(\eta\xi)\|_{s,q} \Big) \\
    &\leq C 2^l l^s \|\eta\|_{L^{\infty}}^{l-2} \|\eta\|_{C^s} \|\eta\|_{s,q_1} \|\xi\|_{q_2}.
\end{align*}
Combining both components:
\begin{align*}
    \big\||D|^{-1}B_l(\eta)\xi\big\|_{s,q} 
    &\leq \|P_l(\eta)\xi\|_{s,q} + \|Q_{l,*}(\eta)\xi\|_{s,q} \\
    &\leq C 2^l l^{M+s} \|\eta\|_{s,q_1} \|\eta\|_{C^1}^{l-2} \Big( \|\eta\|_{C^1} \|\xi\|_{s,q_2} + \|\eta\|_{C^{s+1}} \|\xi\|_{q_2} \Big).
\end{align*}
\vspace{0.5em}
\noindent\textbf{Remainder estimate for $B_L^R(\eta)$.} 
For the remainder term, we have
\begin{align*}
    \big\||D|^{-1}B_L^R(\eta)\xi\big\|_{s,q} 
    &\leq \sum_{j=L+1}^{\infty} \big\||D|^{-1}B_j(\eta)\xi\big\|_{s,q} \\
    &\leq C 2^L L^{M+s} \|\eta\|_{s,q_1} \|\eta\|_{C^1}^{L-2} \Big( \|\eta\|_{C^1} \|\xi\|_{s,q_2} + \|\eta\|_{C^{s+1}} \|\xi\|_{q_2} \Big),
\end{align*}
where the summation converges under the condition $\|\eta\|_{C^1} < 1/4$.

\vspace{0.5em}
\noindent This completes the proof of Theorem \ref{ThmA2}.
\end{proof}

In the proof of the above theorem, we establish the following corollary.
\begin{corollary}\label{cor3.1}
    Let $1 < q < \infty$, $\frac{1}{q} = \frac{1}{q_1} + \frac{1}{q_2}$,  $s \in \mathbb{Z}$ with $s \geq 0$, and $\|\eta_i\|_{C^1} < \frac{1}{4}$ for $i=1,2$. For $l \geq 1$, $L \geq 1$, and $\eta_t = \eta_1 + t(\eta_2 - \eta_1)$ with $t \in [0,1]$, the operators $B_l$ and $B_L^R$ satisfy the estimates
    \begin{align*}
        \left\||D|^{-1}\frac{\mathrm{d}}{\mathrm{d}t}B_1(\eta_t)\xi\right\|_{s,q} 
        &\le C\|\eta_2 - \eta_1\|_{s,q_1}\|\xi\|_{q_2}, \\
        \left\||D|^{-1}\frac{\mathrm{d}}{\mathrm{d}t} B_l(\eta_t)\xi \right\|_{s,q} 
        &\le C2^l l^{M+s+1}\|\eta_2 - \eta_1\|_{s,q_1}\|\eta_t\|_{C^1}^{l-2} \\
        &\quad \times \big(\|\eta_t\|_{C^1}\|\xi\|_{s,q_2} + \|\eta_t\|_{C^{s+1}}\|\xi\|_{q_2}\big), \quad  l \geq 2,
    \end{align*}
    and
    \begin{align*}
        \left\||D|^{-1}\frac{\mathrm{d}}{\mathrm{d}t}B_L^R(\eta_t)\xi\right\|_{s,q} 
        &\le C2^L L^{M+s+1}\|\eta_2 - \eta_1\|_{s,q_1}\|\eta_t\|_{C^1}^{L-2} \\
        &\quad \times \big(\|\eta_t\|_{C^1}\|\xi\|_{s,q_2} + \|\eta_t\|_{C^{s+1}}\|\xi\|_{q_2}\big), \quad  L \geq 2.
    \end{align*}
    Moreover, 
    \begin{equation*}
        \left\||D|^{-1}\frac{\mathrm{d}}{\mathrm{d}t}B(\eta_t)\xi\right\|_{s,q} 
        \leq C\|\eta_2 - \eta_1\|_{s,q_1}\Big(\|\xi\|_{q_2} + \|\eta_t\|_{C^1}\|\xi\|_{s,q_2} + \|\eta_t\|_{C^{s+1}}\|\xi\|_{q_2}\Big).
    \end{equation*}
\end{corollary}

\begin{proof}
    We first establish the estimate for $B_1$. Direct calculation shows that
    \begin{align*}
        |D|^{-1}\frac{\mathrm{d}}{\mathrm{d}t}B_1(\eta_t)\xi(\mathbf{x}')
        = (1 + e^{-2|D|})^{-1} \frac{1}{\pi} \int_{\mathbb{R}^2} \frac{(\eta_2 - \eta_1)(\mathbf{x}') + (\eta_2 - \eta_1)(\mathbf{y}')}
        {(|\mathbf{x}' - \mathbf{y}'|^2 + 4)^{3/2}} \xi(\mathbf{y}')  \mathrm{d}\mathbf{y}'.
    \end{align*}
    The kernel $\mathcal{K}(\mathbf{x}', \mathbf{y}') = (|\mathbf{x}' - \mathbf{y}'|^2 + 4)^{-3/2}$ defines a bounded integral operator on the relevant function spaces. Combining this with the boundedness of the Fourier multiplier $(1 + e^{-2|D|})^{-1}$ on Sobolev spaces, we obtain
    \begin{equation*}
        \left\||D|^{-1}\frac{\mathrm{d}}{\mathrm{d}t}B_1(\eta_t)\xi \right\|_{s,q} \leq C\|\eta_2 - \eta_1\|_{s,q_1}\|\xi\|_{q_2}.
    \end{equation*}

    For $l \geq 2$ and $L \geq 2$, the estimates for $B_l$ and $B_L^R$ follow by applying the same techniques used in the proof of Theorem \ref{ThmA2}, where we differentiate the operator series and control the resulting expressions using the norm bounds for $\eta_t$ and $\xi$. 
\end{proof}

For the operator $B_l$, we establish the following estimates.
\begin{theorem}\label{ThmA2'}
    Let $1 < q < \infty$, $\frac{1}{q} = \frac{1}{q_1} + \frac{1}{q_2}$, $s \in \mathbb{Z}$ with $s \geq 0$, and $\|\eta\|_{C^1} < \frac{1}{4}$. For $l \geq 1$ and $L \geq 1$, the following estimates hold:
    \begin{align*}
        \|B_1(\eta)\xi\|_{s,q} 
        &\leq C\left( \|\partial_x\eta\|_{s,q_1}\|\xi\|_{s,q_2} + \|\eta\|_{s,q_1}\|\partial_x\xi\|_{s,q_2} \right), \\
        \|B_l(\eta)\xi\|_{s,q} 
        &\leq C2^l l^{M+s+1} \|\eta\|_{C^1}^{l-2} \|\eta\|_{C^{s+1}} 
        \left( \|\partial_x\eta\|_{s,q_1}\|\xi\|_{s,q_2} + \|\eta\|_{s,q_1}\|\partial_x\xi\|_{s,q_2} \right), \quad l \geq 2, \\
        \|B_L^R(\eta)\xi\|_{s,q} 
        &\leq C2^L L^{M+s+1} \|\eta\|_{C^1}^{L-2} \|\eta\|_{C^{s+1}} 
        \left( \|\partial_x\eta\|_{s,q_1}\|\xi\|_{s,q_2} + \|\eta\|_{s,q_1}\|\partial_x\xi\|_{s,q_2} \right), \quad L \geq 2.
    \end{align*}
    Moreover, for the full operator $B(\eta) = \sum_{l=1}^\infty B_l(\eta)$, we have
    \begin{equation*}
        \|B(\eta)\xi\|_{s,q} 
        \leq C\left(1 + \|\eta\|_{C^{s+1}}\right)
        \left( \|\partial_x\eta\|_{s,q_1}\|\xi\|_{s,q_2} + \|\eta\|_{s,q_1}\|\partial_x\xi\|_{s,q_2} \right).
    \end{equation*}
\end{theorem}

\begin{proof}
Recall the operator expressions:
\begin{equation*}
    \begin{aligned}
        B_1(\eta)\xi(\mathbf{x}') 
        &= |D|(1 + e^{-2|D|})^{-1} \frac{1}{\pi} 
        \int_{\mathbb{R}^2} \frac{\eta(\mathbf{x}') + \eta(\mathbf{y}')}
        {(|\mathbf{x}' - \mathbf{y}'|^2 + 4)^{3/2}} \xi(\mathbf{y}')  \mathrm{d}\mathbf{y}', \\
        B_l(\eta)\xi(\mathbf{x}') 
        &= -|D|(1 + e^{-2|D|})^{-1} \frac{1}{2\pi} 
        \int_{\mathbb{R}^2} \left( \frac{p_l(P)}{|\mathbf{x}' - \mathbf{y}'|} 
        + \frac{q_l(Q, \kappa_h)}{(|\mathbf{x}' - \mathbf{y}'|^2 + 4)^{1/2}} \right) \xi(\mathbf{y}')  \mathrm{d}\mathbf{y}', 
    \end{aligned}
\end{equation*}
where $p_l(P) = a_l P^{l}$ and $q_l(Q, \kappa_h) = \sum_{j=0}^{\lfloor l/2 \rfloor} b_{l-j} \kappa_h^{l-2j} Q^l$.

\vspace{0.5em}
\noindent\textbf{Part 1: Estimate for $B_1(\eta)$} \\
Denote the kernel $M(\mathbf{x'}) = (|\mathbf{x'}|^2 + 4)^{-3/2}$. 
We have
\begin{align*}
    \|B_1(\eta)\xi\|_{s,q} 
    &\leq C \big\| \partial_{\mathbf{x}'} \big[\eta (M * \xi)\big] \big\|_{s,q} 
    + C \big\| \partial_{\mathbf{x}'} \big[M * (\eta \xi)\big] \big\|_{s,q} \\
    &\leq C \|\partial_{\mathbf{x}'} \eta\|_{s,q_1} \|M * \xi\|_{s,q_2} 
    + C \|\eta\|_{s,q_1} \|\partial_{\mathbf{x}'} (M * \xi)\|_{s,q_2} \\
    &\quad + C \|M * \partial_{\mathbf{x}'} (\eta \xi)\|_{s,q} \\
    &\leq C \|\partial_{\mathbf{x}'} \eta\|_{s,q_1} \|\xi\|_{s,q_2} 
    + C \|\eta\|_{s,q_1} \|\partial_{\mathbf{x}'} \xi\|_{s,q_2}.
\end{align*}

\vspace{0.5em}
\noindent\textbf{Part 2: Estimate for $B_l(\eta)$ with $l \geq 2$} \\
We decompose $B_l(\eta)\xi = -|D|(1 + e^{-2|D|})^{-1} \frac{1}{2\pi} \big[ P_l(\eta)\xi + Q_{l,*}(\eta)\xi \big]$ where:
\begin{align*}
    P_l(\eta)\xi(\mathbf{x}') &= \int_{\mathbb{R}^2} |\mathbf{x}' - \mathbf{y}'|^{-1} p_l(P) \xi(\mathbf{y}')  \mathrm{d}\mathbf{y}', \\
    Q_{l,*}(\eta)\xi(\mathbf{x}') &= \int_{\mathbb{R}^2} (|\mathbf{x}' - \mathbf{y}'|^2 + 4)^{-1/2} q_l(Q, \kappa_h) \xi(\mathbf{y}')  \mathrm{d}\mathbf{y}'.
\end{align*}

\noindent\textit{Step 2.1: Derivative of $P_l(\eta)\xi$} \\
Differentiating $P_l(\eta)\xi$ yields
\begin{align*}
    \partial_{\mathbf{x}'} P_l(\eta)\xi(\mathbf{x}')
    &= l a_l \int_{\mathbb{R}^2} |\mathbf{x}' - \mathbf{y}'|^{-2} P^{l-1} 
    \big(\partial_{\mathbf{x}'}\eta(\mathbf{x}') - \partial_{\mathbf{y}'}\eta(\mathbf{y}')\big) \xi(\mathbf{y}')  \mathrm{d}\mathbf{y}' \\
    &\quad + a_l \int_{\mathbb{R}^2} |\mathbf{x}' - \mathbf{y}'|^{-2} P^{l-1} 
    \big(\eta(\mathbf{x}') - \eta(\mathbf{y}')\big) \partial_{\mathbf{y}'}\xi(\mathbf{y}')  \mathrm{d}\mathbf{y}'.
\end{align*}
This decomposes into four operators:
\begin{align*}
    \mathcal{T}_1 &= (\partial_{\mathbf{x}'}\eta) S_{l-1}\xi, \quad
    \mathcal{T}_2 = S_{l-1}((\partial_{\mathbf{y}'}\eta)\xi), \\
    \mathcal{T}_3 &= \eta S_{l-1}(\partial_{\mathbf{y}'}\xi), \quad
    \mathcal{T}_4 = S_{l-1}(\eta \partial_{\mathbf{y}'}\xi).
\end{align*}
Applying Lemma \ref{LemmaA3} and using $\|\eta\|_{C^1} < 1/4$ gives
\begin{align*}
    \|\partial_{\mathbf{x}'} P_l(\eta)\xi\|_{s,q} 
    &\leq Cl \big( \|\mathcal{T}_1\|_{s,q} + \|\mathcal{T}_2\|_{s,q} + \|\mathcal{T}_3\|_{s,q} + \|\mathcal{T}_4\|_{s,q} \big) \\
    &\leq C l^{M+s+1} \|\eta\|_{C^1}^{l-2} \|\eta\|_{C^{s+1}} 
    \Big( \|\partial_{\mathbf{x}'}\eta\|_{s,q_1} \|\xi\|_{s,q_2} + \|\eta\|_{s,q_1} \|\partial_{\mathbf{x}'}\xi\|_{s,q_2} \Big).
\end{align*}

\noindent\textit{Step 2.2: Derivative of $Q_{l,*}(\eta)\xi$} \\
Differentiating $Q_{l,*}(\eta)\xi$ yields
\begin{align*}
    &\partial_{\mathbf{x}'} Q_{l,*}(\eta)\xi(\mathbf{x}')\\
    &= l \sum_{j=0}^{\lfloor l/2 \rfloor} b_{l-j} \int_{\mathbb{R}^2} (|\mathbf{x}' - \mathbf{y}'|^2 + 4)^{-1} \kappa_h^{l-2j} Q^{l-1}  \big(\partial_{\mathbf{x}'}\eta(\mathbf{x}') + \partial_{\mathbf{y}'}\eta(\mathbf{y}')\big) \xi(\mathbf{y}')  \mathrm{d}\mathbf{y}' \\
    &\quad + \sum_{j=0}^{\lfloor l/2 \rfloor} b_{l-j} \int_{\mathbb{R}^2} (|\mathbf{x}' - \mathbf{y}'|^2 + 4)^{-1} \kappa_h^{l-2j} Q^{l-1}  \big(\eta(\mathbf{x}') + \eta(\mathbf{y}')\big) \partial_{\mathbf{y}'}\xi(\mathbf{y}')  \mathrm{d}\mathbf{y}'.
\end{align*}
This decomposes into:
\begin{align*}
    \mathcal{U}_1 &= (\partial_{\mathbf{x}'}\eta) S_{l-1,*}\xi, \quad
    \mathcal{U}_2 = S_{l-1,*}(\partial_{\mathbf{y}'}\eta \xi), \\
    \mathcal{U}_3 &= \eta S_{l-1,*}(\partial_{\mathbf{y}'}\xi), \quad
    \mathcal{U}_4 = S_{l-1,*}(\eta \partial_{\mathbf{y}'}\xi).
\end{align*}
Applying Lemma \ref{LemmaA4} and using $\|\eta\|_{C^1} < 1/4$ gives
\begin{align*}
    \|\partial_{\mathbf{x}'} Q_{l,*}(\eta)\xi\|_{s,q} 
    &\leq C 2^l l \big( \|\mathcal{U}_1\|_{s,q} + \|\mathcal{U}_2\|_{s,q} + \|\mathcal{U}_3\|_{s,q} + \|\mathcal{U}_4\|_{s,q} \big) \\
    &\leq C 2^l l^{s+1} \|\eta\|_{C^1}^{l-2} \|\eta\|_{C^s} 
    \Big( \|\partial_{\mathbf{x}'}\eta\|_{s,q_1} \|\xi\|_{q_2} + \|\eta\|_{s,q_1} \|\partial_{\mathbf{y}'}\xi\|_{q_2} \Big).
\end{align*}

\noindent\textit{Step 2.3: Combining estimates} \\
Since $(1 + e^{-2|D|})^{-1}$ is bounded on $W^{s,q}$, we combine both components:
\begin{align*}
    \|B_l(\eta)\xi\|_{s,q} 
    &\leq C \big( \|\partial_{\mathbf{x}'} P_l(\eta)\xi\|_{s,q} + \|\partial_{\mathbf{x}'} Q_{l,*}(\eta)\xi\|_{s,q} \big) \\
    &\leq C 2^l l^{M+s+1} \|\eta\|_{C^1}^{l-2} \|\eta\|_{C^{s+1}} 
    \Big( \|\partial_{\mathbf{x}'}\eta\|_{s,q_1} \|\xi\|_{s,q_2} + \|\eta\|_{s,q_1} \|\partial_{\mathbf{x}'}\xi\|_{s,q_2} \Big).
\end{align*}

\vspace{0.5em}
\noindent\textbf{Part 3: Estimate for $B_L^R(\eta)$} \\
The bound for $B_L^R(\eta)$ follows identically to $B_l(\eta)$ using the methodology from Theorem \ref{ThmA2}, as the residual terms satisfy analogous kernel estimates.

\vspace{0.5em}
\noindent\textbf{Part 4: Estimate for $B(\eta)$} \\
Summing over $l \geq 1$ and using $\|\eta\|_{C^1} < 1/4$ for convergence:
\begin{align*}
    \|B(\eta)\xi\|_{s,q} 
    &\leq \sum_{l=1}^\infty \|B_l(\eta)\xi\|_{s,q} \\
    &\leq C \big(1 + \|\eta\|_{C^{s+1}}\big) 
    \Big( \|\partial_{\mathbf{x}'}\eta\|_{s,q_1} \|\xi\|_{s,q_2} + \|\eta\|_{s,q_1} \|\partial_{\mathbf{x}'}\xi\|_{s,q_2} \Big),
\end{align*}
where the series converges absolutely since $\|\eta\|_{C^1} < 1/4$ implies geometric decay. 
\end{proof}

\begin{remark} \label{LemmaA5}
    In \cite{Craig2, Schanz}, the estimates for the operators $B_l$ and $B_L^R$ are given as follows:  
    Let $1<q<\infty$ and $0 \leq s \in \mathbb{Z}$. The operators $B_l$ and $B_L^R$ satisfy  
$$
    \left\|B_l \xi\right\|_{s, q} \leq C\|\eta\|_{C^1}^{l-1}\left(\|\eta\|_{C^{s+1}}\|\xi\|_q + \|\eta\|_{C^1}\|\xi\|_{s, q}\right)
$$
    and    
$$
    \|B_L^R \xi\|_{s, q} \leq C\|\eta\|_{C^1}^{L-1}\left(\|\eta\|_{C^{s+1}}\|\xi\|_q + \|\eta\|_{C^1}\|\xi\|_{s, q}\right).
$$
    The lemma established above enables us to refine these estimates for practical applications.  
\end{remark}

In estimating the remainder term $R_3$, we require the following estimates for $B^m$, which follow directly from Remark \ref{LemmaA5}:  
\begin{lemma}\label{LemmaA7}  
    The powers of $B$ satisfy  
$$
    \|B^m(\eta) \xi\|_{s, q} \leq C\left(1+C_0\right)^{m-1}\|\eta\|_{C^1}^{m-1}\big(\|\eta\|_{C^{s+1}}\|\xi\|_q + \|\eta\|_{C^1}\|\xi\|_{s, q}\big).
$$
\end{lemma}

\subsection{Estimates for the operator $A_l$}
We now estimate the operator $A_l$, which is derived from the representation introduced in Section~\ref{secA} and plays a pivotal role in bounding the residual term $R_3$.

First, we state a lemma essential for estimating $A_l$; detailed proofs can be found in \cite{Craig2, Schanz}.
\begin{lemma}\label{LemmaA2}
	For ${\bf x'}\in \mathbb{R}^{2}$, $f(P)$ is an old continuous function of $P$ and $\eta \in$ $C^1\left(\mathbb{R}^2\right)$, we have
	$$
	\int_{\mathbb{R}^2} \frac{{\bf x'}-{\bf y'}}{\left|{\bf x'}-{\bf y'}\right|^2} \cdot \partial_{{\bf y'}}(f(P)) \xi\left({\bf y'}\right) d {\bf y'}=-\int_{\mathbb{R}^2} \frac{{\bf x'}-{\bf y'}}{\left|{\bf x'}-{\bf y'}\right|^{ 2}} f(P) \cdot \partial_{{\bf y'}} \xi\left({\bf y'}\right) d {\bf y'} .
	$$
\end{lemma}

\begin{theorem}
	\label{ThmA1} Let $\frac{1}{q} = \frac{1}{q_1} + \frac{1}{q_2}$ and $\|\eta\|_{L^{\infty}}<1/4$. The operators $A_l$ and $A_L^R$satisfy the following estimates for all integers $l \geq 2$ and $L \geq 2$:
	\begin{equation*}
		\left\||D|^{-1}A_l(\eta) \xi\right\|_{s, q} \leq C 2^ll^{M+s} \|\eta\|_{C^1}^{l-2} \|\eta\|_{\max\{s,1\},q_1} \left( \|\eta\|_{C^1} \left\|\partial_{{\bf x'}} \xi\right\|_{s,q_2} + \|\eta\|_{C^{s+1}} \| \xi\|_{1, q_2} \right),
	\end{equation*}
	and 
	\begin{equation*}
		\left\||D|^{-1}A_L^R(\eta) \xi\right\|_{s, q} \leq C 2^LL^{M+s} \|\eta\|_{C^1}^{L-2} \|\eta\|_{\max\{s,1\},q_1} \left( \|\eta\|_{C^1} \left\|\partial_{{\bf x'}} \xi\right\|_{s,q_2} + \|\eta\|_{C^{s+1}} \| \xi\|_{1, q_2} \right).
	\end{equation*}
\end{theorem}

\begin{proof}
We recall the integral representation of $|D|^{-1}A(\eta)\xi$:
\begin{align*}
|D|^{-1} A(\eta)\xi 
&= (1+e^{-2|D|})^{-1}\frac{1}{2\pi} \Bigg[ \int_{\mathbb{R}^2} 
\frac{({\bf x'}-{\bf y'}) \cdot \partial_{{\bf y'}} P}{|{\bf x'}-{\bf y'}|^2} 
\frac{\xi({\bf y'})}{(1+P^2)^{3/2}}  d{\bf y'} \\
&\quad - \int_{\mathbb{R}^2} \left( \frac{({\bf x'}-{\bf y'}) \cdot \partial_{{\bf y'}} Q}{|{\bf x'} -{\bf y'}|^2 + 4} 
+ \frac{4Q}{(|{\bf x'} -{\bf y'}|^2 + 4)^2} \right) 
\frac{\xi({\bf y'})}{(1+\kappa_h Q + Q^2)^{3/2}}  d{\bf y'} \\
&\quad - \int_{\mathbb{R}^2} \frac{2}{(|{\bf x'} -{\bf y'}|^2 + 4)^{3/2}}
\left[ \frac{1}{(1+\kappa_h Q + Q^2)^{3/2}} - 1 \right] \xi({\bf y'})  d{\bf y'} \Bigg].
\end{align*}
For $l \geq 2$, we decompose $|D|^{-1}A_l \xi$ as
\begin{equation*}
|D|^{-1}A_l \xi({\bf x'}) = -(1 + e^{-2|D|})^{-1} \frac{1}{2\pi} 
\Big[ \tilde{P}_l(\eta)\xi({\bf x'}) + \tilde{Q}_{l,*}(\eta) \xi({\bf x'}) + \tilde{R}_{l,*}(\eta) \xi({\bf x'}) \Big],
\end{equation*}
where the component operators are defined by
\begin{align*}
\tilde{P}_l(\eta) \xi({\bf x'}) 
&:= \int_{\mathbb{R}^2} 
\frac{{\bf x'} - {\bf y'}}{|{\bf x'} - {\bf y'}|^2} 
\cdot \partial_{{\bf y'}} \tilde{p}_l(P)  \xi({\bf y'})  d{\bf y'}, \\
\tilde{Q}_{l,*}(\eta) \xi({\bf x'}) 
&:= \int_{\mathbb{R}^2} 
\frac{{\bf x'} - {\bf y'}}{|{\bf x'} - {\bf y'}|^2 + 4} 
\cdot (\partial_{{\bf y'}} Q)  \tilde{q}_{l-1}(Q, \kappa_h)  \xi({\bf y'})  d{\bf y'}, \\
\tilde{R}_{l,*}(\eta) \xi({\bf x'}) 
&:= \int_{\mathbb{R}^2} \Bigg( \frac{4 Q}{(|{\bf x'} - {\bf y'}|^2 + 4)^2} \tilde{q}_{l-1}(Q, \kappa_h) \\
&\quad + \frac{2}{(|{\bf x'} - {\bf y'}|^2 + 4)^{3/2}} \tilde{q}_l(Q, \kappa_h) \Bigg) \xi({\bf y'})  d{\bf y'}.
\end{align*}
The series expansions are given by
\begin{align*}
\frac{\partial_{\bf y'} P}{(1+P^2)^{3/2}} 
&= \partial_{\bf y'} \left( \frac{P}{(1+P^2)^{1/2}} \right) 
= \sum_{\substack{l \geq 1 \\ l \text{ odd}}} \partial_{\bf y'} \tilde{p}_l(P), \quad
\tilde{p}_l(P) := \tilde{a}_l P^{l}, \\
\tilde{a}_l &= \left( -\frac{1}{4} \right)^{(l-1)/2} \binom{l-1}{(l-1)/2}, \\
\frac{1}{(1+\kappa_h Q+Q^2)^{3/2}} 
&= \sum_{l=1}^{\infty} \tilde{q}_{l-1}(Q,\kappa_h), \quad
\tilde{q}_{l-1}(Q,\kappa_h) 
:= \sum_{j=0}^{\lfloor (l-1)/2 \rfloor} \tilde{b}_{l-1-j} \kappa_h^{l-1-2j} Q^{l-1}, \\
\tilde{b}_{l-1-j} 
&= (-1)^{l-1-j} \frac{(2l-1-2j)!!}{2^{l-1-j} (l-1-j)!} \binom{l-1-j}{j}.
\end{align*}
By Remark \ref{LemmaA1}, we focus on three integral operators in $A_l$. Applying Lemma \ref{LemmaA2} to $\tilde{P}_l$ yields
\begin{align*}
\tilde{P}_l(\eta) \xi({\bf x'}) 
&= -\int_{\mathbb{R}^2} 
\frac{{\bf x'} - {\bf y'}}{|{\bf x'} - {\bf y'}|^2} 
\cdot \tilde{p}_l(P) \partial_{{\bf y'}} \xi({\bf y'})  d{\bf y'} \\
&= -\int_{\mathbb{R}^2} 
\frac{{\bf x'} - {\bf y'}}{|{\bf x'} - {\bf y'}|^3} 
\cdot \tilde{a}_l P^{l-1} (\eta({\bf x'}) - \eta({\bf y'})) \partial_{{\bf y'}} \xi({\bf y'})  d{\bf y'},
\end{align*}
which decomposes into $\eta({\bf x'})S_{l-1} \partial_{\bf x'} \xi$ and $S_{l-1} (\eta \partial_{\bf x'} \xi)$. Applying Lemma \ref{LemmaA3}:
\begin{align*}
\|\tilde{P}_l(\eta) \xi\|_{s,q} 
&\leq \|\eta({\bf x'})S_{l-1} \partial_{\bf x'} \xi\|_{s,q} 
+ \|S_{l-1} (\eta \partial_{\bf x'} \xi)\|_{s,q} \\
&\leq C l^{M+s} \|\eta\|_{C^1}^{l-2} \|\eta\|_{s,q_1}
\Big( \|\eta\|_{C^1} \|\partial_{{\bf x'}} \xi\|_{s,q_2} 
+ \|\eta\|_{C^{s+1}} \|\partial_{\bf x'} \xi\|_{q_2} \Big).
\end{align*}
For $\tilde{Q}_{l,*}$, we compute
\begin{align*}
\partial_{{\bf y'}} Q 
&= \frac{\partial_{{\bf y'}} \eta({\bf y'})}{(|{\bf x'} - {\bf y'}|^2 + 4)^{1/2}} 
+ \frac{(\eta({\bf x'}) + \eta({\bf y'}))({\bf x'} - {\bf y'})}{(|{\bf x'} - {\bf y'}|^2 + 4)^{3/2}},
\end{align*}
which gives
\begin{align*}
\tilde{Q}_{l,*}(\eta) \xi({\bf x'}) 
&= \int_{\mathbb{R}^2} 
\frac{{\bf x'} - {\bf y'}}{(|{\bf x'} - {\bf y'}|^2 + 4)^{3/2}} 
\cdot \partial_{{\bf y'}} \eta({\bf y'}) \tilde{q}_{l-1}(Q, \kappa_h) \xi({\bf y'})  d{\bf y'} \\
&\quad + \int_{\mathbb{R}^2} 
\frac{|{\bf x'} - {\bf y'}|^2}{(|{\bf x'} - {\bf y'}|^2 + 4)^{5/2}} 
(\eta({\bf x'}) + \eta({\bf y'})) \tilde{q}_{l-1}(Q, \kappa_h) \xi({\bf y'})  d{\bf y'}.
\end{align*}
These correspond to $S_{l-1, *} (\partial_{\bf x'} \eta \cdot \xi)$, $\eta({\bf x'})S_{l-1, *} \xi$, and $S_{l-1, *} (\eta \xi)$. By Lemma \ref{LemmaA4}:
\begin{align*}
\|\tilde{Q}_{l,*}(\eta) \xi\|_{s,q} 
&\leq C 2^l \left( \|S_{l-1, *} (\partial_{\bf x'} \eta \cdot \xi)\|_{s,q} 
+ \|\eta({\bf x'})S_{l-1, *} \xi\|_{s,q}
+ \|S_{l-1, *} (\eta \xi)\|_{s,q} \right) \\
&\leq C 2^l l^s \|\eta\|_{L^{\infty}}^{l-2} \|\eta\|_{C^s} 
\Big( \|\partial_{\bf x'} \eta \cdot \xi\|_{q} 
+ \|\xi\|_{q_2} \|\eta\|_{s,q_1} 
+ \|\eta \xi\|_{q} \Big) \\
&\leq C 2^l l^s \|\eta\|_{L^{\infty}}^{l-2} \|\eta\|_{C^s} 
\|\eta\|_{\max\{s,1\},q_1} \|\xi\|_{q_2}.
\end{align*}
For $\tilde{R}_{l,*}$:
\begin{align*}
\tilde{R}_{l,*}(\eta) \xi({\bf x'}) 
&= \int_{\mathbb{R}^2} 
\frac{4}{(|{\bf x'} - {\bf y'}|^2 + 4)^{5/2}} 
(\eta({\bf x'}) + \eta({\bf y'})) \tilde{q}_{l-1}(Q, \kappa_h) \xi({\bf y'})  d{\bf y'} \\
&\quad + \sum_{j=0}^{\lfloor l/2 \rfloor} \tilde{b}_{l-j} 
\int_{\mathbb{R}^2} \frac{2}{(|{\bf x'} - {\bf y'}|^2 + 4)^{2}} 
\kappa_h^{l-2j} Q^{l-1} (\eta({\bf x'}) + \eta({\bf y'})) \xi({\bf y'})  d{\bf y'},
\end{align*}
yielding $\eta({\bf x'})S_{l-1, *} \xi$ and $S_{l-1, *} (\eta \xi)$. Using Lemma \ref{LemmaA4} and $\|\eta\|_{L^{\infty}} < 1/4$:
\begin{align*}
\|\tilde{R}_{l,*}(\eta) \xi\|_{s,q} 
&\leq C 2^l \left( \|\eta({\bf x'})S_{l-1, *} \xi\|_{s,q} 
+ \|S_{l-1, *} (\eta \xi)\|_{s,q} \right) \\
&\leq C 2^l l^s \|\eta\|_{L^{\infty}}^{l-2} \|\eta\|_{C^s} 
\|\xi\|_{q_2} \|\eta\|_{s,q_1}.
\end{align*}
Combining estimates for $l \geq 2$:
\begin{align*}
&\||D|^{-1}A_l(\eta) \xi\|_{s,q} \\
&\leq \|\tilde{P}_l(\eta) \xi\|_{s, q} 
+ \|\tilde{Q}_{l,*}(\eta) \xi\|_{s, q} 
+ \|\tilde{R}_{l,*}(\eta) \xi\|_{s, q} \\
&\leq C l^{M+s} \|\eta\|_{C^1}^{l-2} \|\eta\|_{s,q_1} 
\Big( \|\eta\|_{C^1} \|\partial_{{\bf x'}} \xi\|_{s,q_2} 
+ \|\eta\|_{C^{s+1}} \|\partial_{\bf x'} \xi\|_{q_2} \Big) \\
&\quad + C 2^l l^s \|\eta\|_{L^{\infty}}^{l-2} \|\eta\|_{C^s} 
\|\eta\|_{\max\{s,1\},q_1} \|\xi\|_{q_2} \\
&\leq C 2^l l^{M+s} \|\eta\|_{C^1}^{l-2} \|\eta\|_{\max\{s,1\},q_1} 
\Big( \|\eta\|_{C^1} \|\partial_{{\bf x'}} \xi\|_{s,q_2} 
+ \|\eta\|_{C^{s+1}} \|\xi\|_{1,q_2} \Big).
\end{align*}
For the remainder operator $A_L^R$:
\begin{align*}
&\||D|^{-1}A_L^R(\eta) \xi\|_{s, q} \\
&\leq \sum_{l=L+1}^{\infty} \||D|^{-1}A_l(\eta) \xi\|_{s, q} \\
&\leq C 2^L L^{M+s} \|\eta\|_{C^1}^{L-2} \|\eta\|_{\max\{s,1\},q_1} 
\Big( \|\eta\|_{C^1} \|\partial_{{\bf x'}} \xi\|_{s,q_2} 
+ \|\eta\|_{C^{s+1}} \|\xi\|_{1,q_2} \Big).
\end{align*}
This completes the proof.
\end{proof}

During the proof of the preceding theorem, the following corollary is easily obtained.

\begin{corollary} \label{Cor DA}
    Let $\frac{1}{q} = \frac{1}{q_1} + \frac{1}{q_2}$ with $\|\eta_1\|_{L^{\infty}}, \|\eta_2\|_{L^{\infty}} < \frac{1}{4}$. For $\eta_t = \eta_1 + t(\eta_2 - \eta_1)$ where $t \in [0,1]$, the operators $A_l$ and $A_L^R$ satisfy
    \begin{align*}
        \left\||D|^{-1} \frac{d}{dt} A_l(\eta_t) \xi\right\|_{s, q} 
        &\leq C  2^l  l^{M+s+1}  \|\eta_t\|_{C^1}^{l-2}  \|\eta_2 - \eta_1\|_{\max\{s,1\},q_1} \\
        &\quad \times \left( \|\eta_t\|_{C^1} \|\partial_{\mathbf{x'}} \xi\|_{s,q_2} + \|\eta_t\|_{C^{s+1}} \|\xi\|_{1,q_2} \right),
    \end{align*}
    and
    \begin{align*}
        \left\||D|^{-1} \frac{d}{dt} A_L^R(\eta_t) \xi\right\|_{s, q} 
        &\leq C  2^L  L^{M+s+1}  \|\eta_t\|_{C^1}^{L-2}  \|\eta_2 - \eta_1\|_{\max\{s,1\},q_1} \\
        &\quad \times \left( \|\eta_t\|_{C^1} \|\partial_{\mathbf{x'}} \xi\|_{s,q_2} + \|\eta_t\|_{C^{s+1}} \|\xi\|_{1,q_2} \right),
    \end{align*}
    for all integers $l \geq 2$ and $L \geq 2$, where $C > 0$ is a constant.
\end{corollary}
\begin{proof}
For $l \geq 2$, we first analyze operator $A_l$:
\begin{equation*}
|D|^{-1}\frac{d}{dt}A_l(\eta_t) \xi({\bf x'}) = -(1 + e^{-2|D|})^{-1} \frac{1}{2\pi} 
\Big[ \frac{d}{dt}\tilde{P}_l(\eta_t)\xi({\bf x'}) + \frac{d}{dt}\tilde{Q}_{l,*}(\eta_t) \xi({\bf x'}) + \frac{d}{dt}\tilde{R}_{l,*}(\eta_t) \xi({\bf x'}) \Big].
\end{equation*}
\textbf{Analysis of $\frac{d}{dt}\tilde{P}_l(\eta_t)$:}
From Lemma \ref{LemmaA2}, we derive
\begin{align*}
\frac{d}{dt}\tilde{P}_l(\eta_t) \xi({\bf x'}) 
&= - \frac{d}{dt}\int_{\mathbb{R}^2} 
\frac{{\bf x'} - {\bf y'}}{|{\bf x'} - {\bf y'}|^2} 
\cdot  (\tilde{a}_l P^{l}(\eta_t)) \partial_{{\bf y'}}\xi({\bf y'})  d{\bf y'} \\
&= -\int_{\mathbb{R}^2} 
\frac{{\bf x'} - {\bf y'}}{|{\bf x'} - {\bf y'}|^3} 
\cdot  (l \tilde{a}_l P^{l-1}(\eta_t))(\eta_2-\eta_1) \partial_{{\bf y'}}\xi({\bf y'})  d{\bf y'}.
\end{align*}
Applying Lemma \ref{LemmaA3} yields
\begin{align*}
\left\|\frac{d}{dt}\tilde{P}_l(\eta_t) \xi\right\|_{s,q} 
&\leq C l^{M+s+1} \|\eta_t\|_{C^1}^{l-2} \|\eta_2-\eta_1\|_{s,q_1}
\Big( \|\eta_t\|_{C^1} \|\partial_{{\bf x'}} \xi\|_{s,q_2} 
+ \|\eta_t\|_{C^{s+1}} \|\partial_{\bf x'} \xi\|_{q_2} \Big).
\end{align*}
\textbf{Analysis of $\frac{d}{dt} \tilde{Q}_{l,*}(\eta_t)$:}
We have 
\begin{align*}
&\frac{d}{dt}\tilde{Q}_{l,*}(\eta_t) \xi({\bf x'}) \\
&= \int_{\mathbb{R}^2} 
\frac{{\bf x'} - {\bf y'}}{(|{\bf x'} - {\bf y'}|^2 + 4)^{3/2}} 
\cdot \partial_{{\bf y'}} (\eta_2-\eta_1)({\bf y'}) \tilde{q}_{l-1}(Q, \kappa_h) \xi({\bf y'})  d{\bf y'} \\
&\quad + (l-1)\sum_{j=0}^{\lfloor (l-1)/2 \rfloor} \tilde{b}_{l-1-j}\int_{\mathbb{R}^2} 
\frac{{\bf x'} - {\bf y'}}{(|{\bf x'} - {\bf y'}|^2 + 4)^{2}} 
\cdot \partial_{{\bf y'}} \eta_t({\bf y'}) \\
&\qquad\times \Big[\kappa_h^{l-1-2j} Q^{l-2}\big( (\eta_2-\eta_1)({\bf x'})+(\eta_2-\eta_1)({\bf y'})\big) \Big] \xi({\bf y'})  d{\bf y'} \\
&\quad + \int_{\mathbb{R}^2} 
\frac{|{\bf x'} - {\bf y'}|^2}{(|{\bf x'} - {\bf y'}|^2 + 4)^{5/2}} 
\big((\eta_2-\eta_1)({\bf x'}) + (\eta_2-\eta_1)({\bf y'})\big) \tilde{q}_{l-1}(Q, \kappa_h) \xi({\bf y'})  d{\bf y'} \\
&\quad + (l-1)\sum_{j=0}^{\lfloor (l-1)/2 \rfloor} \tilde{b}_{l-1-j} \int_{\mathbb{R}^2} 
\frac{|{\bf x'} - {\bf y'}|^2}{(|{\bf x'} - {\bf y'}|^2 + 4)^{3}} 
\big(\eta_t({\bf x'}) + \eta_t({\bf y'})\big) \\
&\qquad\times \Big[ \kappa_h^{l-1-2j} Q^{l-2}\big((\eta_2-\eta_1)({\bf x'})+(\eta_2-\eta_1)({\bf y'})\big)\Big] \xi({\bf y'})  d{\bf y'}.
\end{align*}
Lemma \ref{LemmaA4} provides the bound
\begin{align*}
\left\|\frac{d}{dt}\tilde{Q}_{l,*}(\eta_t) \xi\right\|_{s,q} 
&\leq C 2^l l^{s+1} \|\eta_t\|_{L^{\infty}}^{l-2} \|\eta_t\|_{C^s} 
\|\eta_2-\eta_1\|_{\max\{s,1\},q_1} \|\xi\|_{q_2}.
\end{align*}
\textbf{Analysis of $\frac{d}{dt} \tilde{R}_{l,*}(\eta_t)$:}
Similarly,
\begin{align*}
&\frac{d}{dt}\tilde{R}_{l,*}(\eta_t) \xi({\bf x'}) \\
&= \int_{\mathbb{R}^2} 
\frac{4}{(|{\bf x'} - {\bf y'}|^2 + 4)^{5/2}} 
\big((\eta_2-\eta_1)({\bf x'}) + (\eta_2-\eta_1)({\bf y'})\big) \tilde{q}_{l-1}(Q, \kappa_h) \xi({\bf y'})  d{\bf y'} \\
&\quad + (l-1)\sum_{j=0}^{\lfloor (l-1)/2 \rfloor} \tilde{b}_{l-1-j} \int_{\mathbb{R}^2} 
\frac{4}{(|{\bf x'} - {\bf y'}|^2 + 4)^{5/2}} 
(\eta_t({\bf x'}) + \eta_t({\bf y'})) \\
&\qquad\times \Big[\kappa_h^{l-1-2j} Q^{l-2}\big((\eta_2-\eta_1)({\bf x'})+(\eta_2-\eta_1)({\bf y'})\big) \Big] \xi({\bf y'})  d{\bf y'} \\
&\quad + l\sum_{j=0}^{\lfloor l/2 \rfloor} \tilde{b}_{l-j} 
\int_{\mathbb{R}^2} \frac{2}{(|{\bf x'} - {\bf y'}|^2 + 4)^{2}} 
\kappa_h^{l-2j} Q^{l-1} \big((\eta_2-\eta_1)({\bf x'}) + (\eta_2-\eta_1)({\bf y'})\big) \xi({\bf y'})  d{\bf y'}.
\end{align*}
Using Lemma \ref{LemmaA4} and the condition $\|\eta_t\|_{L^{\infty}} < 1/4$:
\begin{align*}
\left\|\frac{d}{dt}\tilde{R}_{l,*}(\eta_t) \xi\right\|_{s,q} 
&\leq C 2^l l^{s+1} \|\eta_t\|_{L^{\infty}}^{l-2} \|\eta_t\|_{C^s} 
\|\xi\|_{q_2} \|\eta_2-\eta_1\|_{s,q_1}.
\end{align*}
Combining these estimates yields the final bound for $A_l$. The estimate for $A_L^R$ follows similarly. 
\end{proof}

While estimates for the operator $A_l$ are available in Craig \cite{Craig2, Schanz}, we adapt these results to align with our nonlinear estimates and enhance their applicability.

\begin{theorem} \label{ThmA1'} 
    Let $\frac{1}{q} = \frac{1}{q_1} + \frac{1}{q_2}$. The operators $A_1$, $A_2$, $A_l$, and $A_L^R$ satisfy the following estimates for some constant $C > 0$:
    \begin{align*}
        \|A_1(\eta)\xi\|_{s,q} 
        &\leq C \Big( \|\partial_{\mathbf{x'}} \eta\|_{s,q_1} \|\partial_{\mathbf{x'}} \xi\|_{s,q_2} + \|\eta\|_{s,q_1} \|\partial_{\mathbf{x'}}^2 \xi\|_{s,q_2} \Big), \\
        \|A_2(\eta)\xi\|_{s,q} 
        &\leq C \|\eta\|_{C^{s+1}} \Big( \|\partial_{\mathbf{x'}} \eta\|_{s+1,q_1} \|\xi\|_{s+1,q_2} + \|\eta\|_{s+1,q_1} \|\partial_{\mathbf{x'}} \xi\|_{s+1,q_2} \Big),
    \end{align*}
    and for all integers $l \geq 3$, $L \geq 3$,
    \begin{align*}
        \|A_{l}(\eta)\xi\|_{s,q} 
        &\leq C l^{M+s+1} 2^l \|\eta\|_{C^1}^{l-3} \|\eta\|_{C^{s+1}}^2 \\
        &\quad \times \Big( \|\partial_{\mathbf{x'}} \eta\|_{s+1,q_1} \|\xi\|_{s+1,q_2} + \|\eta\|_{s+1,q_1} \|\partial_{\mathbf{x'}} \xi\|_{s+1,q_2} \Big), \\
        \|A_{L}^{R}(\eta)\xi\|_{s,q} 
        &\leq C L^{M+s+1} 2^L \|\eta\|_{C^1}^{L-3} \|\eta\|_{C^{s+1}}^2 \\
        &\quad \times \Big( \|\partial_{\mathbf{x'}} \eta\|_{s+1,q_1} \|\xi\|_{s+1,q_2} + \|\eta\|_{s+1,q_1} \|\partial_{\mathbf{x'}} \xi\|_{s+1,q_2} \Big).
    \end{align*}
\end{theorem}

\begin{proof}
    Recall the operator definition:
    \[
    A_l \xi(\mathbf{x'}) = -|D|(1+e^{-2|D|})^{-1} \frac{1}{2\pi} \Big( \tilde{P}_l(\eta)\xi(\mathbf{x'}) + \tilde{Q}_{l,*}(\eta)\xi(\mathbf{x'}) + \tilde{R}_{l,*}(\eta)\xi(\mathbf{x'}) \Big),
    \]
    where the components are defined as
    \begin{align*}
        \tilde{P}_l(\eta)\xi(\mathbf{x'}) 
        &= \int_{\mathbb{R}^2} \frac{\mathbf{x'} - \mathbf{y'}}{|\mathbf{x'} - \mathbf{y'}|^2} \cdot \partial_{\mathbf{y'}} \tilde{p}_l(P) \xi(\mathbf{y'}) \, d\mathbf{y'},  && l \text{ odd}, \\
        \tilde{Q}_{l,*}(\eta)\xi(\mathbf{x'}) 
        &= \int_{\mathbb{R}^2} \frac{\mathbf{x'} - \mathbf{y'}}{|\mathbf{x'} - \mathbf{y'}|^2 + 4} \cdot \partial_{\mathbf{y'}} Q \tilde{q}_{l-1}(Q, \kappa_h) \xi(\mathbf{y'}) \, d\mathbf{y'}, && l \geq 1,
    \end{align*}
    and
    \[
    \tilde{R}_{l,*}(\eta)\xi(\mathbf{x'}) = \int_{\mathbb{R}^2} \left( \frac{4Q}{(|\mathbf{x'} - \mathbf{y'}|^2 + 4)^2} \tilde{q}_{l-1}(Q, \kappa_h) + \frac{2}{(|\mathbf{x'} - \mathbf{y'}|^2 + 4)^{3/2}} \tilde{q}_l(Q, \kappa_h) \right) \xi(\mathbf{y'}) \, d\mathbf{y'}.
    \]

    \noindent \textbf{Case 1: \( l = 1 \).}
    The operator simplifies to
    \[
    A_1(\eta)\xi = D\eta \cdot D\xi - |D|(1 + e^{-2|D|})^{-1} \eta(1 + e^{-2|D|}) G_0 \xi,
    \]
    where \( G_0 = |D|\tanh(|D|) \). This implies
    \begin{align*}
        \|A_1(\eta)\xi\|_{s,q} 
        &\leq C \|(\partial_{\mathbf{x'}} \eta)(\partial_{\mathbf{x'}} \xi)\|_{s,q} + C \|\partial_{\mathbf{x'}} (\eta (1 + e^{-2|D|}) G_0 \xi)\|_{s,q} \\
        &\leq C \|\partial_{\mathbf{x'}} \eta\|_{s,q_1} \|\partial_{\mathbf{x'}} \xi\|_{s,q_2} + C \|\partial_{\mathbf{x'}} \eta\|_{s,q_1} \|(1 + e^{-2|D|}) G_0 \xi\|_{s,q_2} \\
        &\quad + C \|\eta\|_{s,q_1} \|(1 + e^{-2|D|}) \partial_{\mathbf{x'}} G_0 \xi\|_{s,q_2}.
    \end{align*}
    Using the kernel \( M_2(\mathbf{x'}) = \frac{1}{\pi (|\mathbf{x'}|^2 + 4)^{3/2}} \) and boundedness of \( \tanh(|D|) \) in $W^{s,p}(\mathbb R^2)$, we obtain
    \[
    \|A_1(\eta)\xi\|_{s,q} \leq C \|\partial_{\mathbf{x'}} \eta\|_{s,q_1} \|\partial_{\mathbf{x'}} \xi\|_{s,q_2} + C \|\eta\|_{s,q_1} \|\partial^2_{\mathbf{x'}} \xi\|_{s,q_2}.
    \]

    \noindent \textbf{Case 2: \( l \geq 2 \), odd.} 
    Analyze \( \partial_{\mathbf{x'}} \tilde{P}_l(\eta) \xi \):
    \begin{align*}
        \partial_{\mathbf{x'}} \tilde{P}_l(\eta) \xi(\mathbf{x'}) 
        &= -l\tilde{a}_l \partial_{\mathbf{x'}} \eta(\mathbf{x'}) \int_{\mathbb{R}^2} \frac{\mathbf{x'} - \mathbf{y'}}{|\mathbf{x'} - \mathbf{y'}|^3} \cdot P^{l-1} \partial_{\mathbf{y'}} \xi(\mathbf{y'}) \, d\mathbf{y'} \\
        &\quad + l\tilde{a}_l \int_{\mathbb{R}^2} \frac{\mathbf{x'} - \mathbf{y'}}{|\mathbf{x'} - \mathbf{y'}|^3} \cdot P^{l-1} \partial_{\mathbf{y'}} \eta \partial_{\mathbf{y'}} \xi \, d\mathbf{y'} \\
        &\quad - \tilde{a}_l \eta(\mathbf{x'}) \int_{\mathbb{R}^2} \frac{\mathbf{x'} - \mathbf{y'}}{|\mathbf{x'} - \mathbf{y'}|^3} \cdot P^{l-1} \partial^2_{\mathbf{y'}} \xi(\mathbf{y'}) \, d\mathbf{y'} \\
        &\quad + \tilde{a}_l \int_{\mathbb{R}^2} \frac{\mathbf{x'} - \mathbf{y'}}{|\mathbf{x'} - \mathbf{y'}|^3} \cdot P^{l-1} \eta(\mathbf{y'}) \partial^2_{\mathbf{y'}} \xi \, d\mathbf{y'}.
    \end{align*}
    By Lemma \ref{LemmaA3}, this yields
    \begin{align*}
        \|\partial_{\mathbf{x'}} \tilde{P}_l(\eta) \xi\|_{s,q} 
        \leq C l^{M+s+1} \|\eta\|_{C^1}^{l-2} \|\eta\|_{C^{s+1}} \Big( \|\partial_{\mathbf{x'}} \eta\|_{s,q_1} \|\partial_{\mathbf{x'}} \xi\|_{s,q_2} + \|\eta\|_{s,q_1} \|\partial_{\mathbf{x'}}^2 \xi\|_{s,q_2} \Big).
    \end{align*}

    \noindent \textbf{Case 3: \( \tilde{Q}_{l,*} \) for \( l \geq 2 \).}
    Rewrite the operator using the substitution \( \mathbf{y'} \mapsto \mathbf{x'} - \mathbf{y'} \):
    \begin{align*}
        \tilde{Q}_{l,*}(\eta) \xi(\mathbf{x'}) 
        &= \int_{\mathbb{R}^2} \frac{\mathbf{y'}}{(|\mathbf{y'}|^2 + 4)^{3/2}} \cdot \partial_{\mathbf{y'}} \eta(\mathbf{x'} - \mathbf{y'}) \tilde{q}_{l-1}(Q, \kappa_h) \xi(\mathbf{x'} - \mathbf{y'}) \, d\mathbf{y'} \\
        &\quad + \int_{\mathbb{R}^2} \frac{|\mathbf{y'}|^2}{(|\mathbf{y'}|^2 + 4)^{5/2}} (\eta(\mathbf{x'}) + \eta(\mathbf{x'} - \mathbf{y'})) \tilde{q}_{l-1}(Q, \kappa_h) \xi(\mathbf{x'} - \mathbf{y'}) \, d\mathbf{y'},
    \end{align*}
    with \( Q = \frac{\eta(\mathbf{x'}) + \eta(\mathbf{x'} - \mathbf{y'})}{(|\mathbf{y'}|^2 + 4)^{1/2}} \), \( \kappa_h = \frac{4}{(|\mathbf{y'}|^2 + 4)^{1/2}} \) and $ \tilde{q}_{l-1}(Q,\kappa_h)=\sum_{j=0}^{[(l-1)/2]}\tilde b_{l-1-j}\kappa_h^{l-1-2j}Q^{l-1}$. Differentiation gives
    \begin{align*}
		\partial_{{\bf x'}} \tilde{Q}_{l, *}(\eta) \xi({\bf x'})
		&= \int_{\mathbb{R}^2}\frac{{\bf y'}}{(| {\bf y'}|^2 + 4)^{3/2}} \cdot \partial_{{\bf y'}}^2 \eta({\bf x'}-{\bf y'})\tilde{q}_{l-1}(Q, \kappa_h)\xi({\bf x'}-{\bf y'})d{\bf y'}\\
		&\quad+\sum_{j=0}^{[(l-1)/2]}\tilde b_{l-1-j} (l-1)\int_{\mathbb{R}^2}\frac{{\bf y'}}{(| {\bf y'}|^2 + 4)^{2}} \cdot \partial_{{\bf y'}} \eta({\bf x'}-{\bf y'})\kappa_h^{l-1-2j}Q^{l-2}\\
		&\quad\times(\partial_{\bf x'} \eta({\bf x'})-\partial_{{\bf y'}}\eta({\bf x'}-{\bf y'}))\xi({\bf x'}-{\bf y'})d{\bf y'}\\
		&\quad-\int_{\mathbb{R}^2}\frac{{\bf y'}}{(| {\bf y'}|^2 + 4)^{3/2}} \cdot \partial_{{\bf y'}} \eta({\bf x'}-{\bf y'})\tilde{q}_{l-1}(Q, \kappa_h)\partial_{\bf y'}\xi({\bf x'}-{\bf y'})d{\bf y'}\\
		&\quad+\int_{\mathbb{R}^2} \frac{| {\bf y'}|^2}{(| {\bf y'}|^2 + 4)^{5/2}} (\partial_{\bf x'}\eta({\bf x'}) -\partial_{\bf y'} \eta({\bf x'}-{\bf y'})) \tilde{q}_{l-1}(Q, \kappa_h)\xi({\bf x'}-{\bf y'})d{\bf y'}\\
		&\quad+\sum_{j=0}^{[(l-1)/2]}\tilde b_{l-1-j} (l-1)\int_{\mathbb{R}^2} \frac{| {\bf y'}|^2}{(| {\bf y'}|^2 + 4)^{3}} (\eta({\bf x'}) + \eta({\bf x'}-{\bf y'})) \kappa_h^{l-1-2j}Q^{l-2}\\
		&\quad\times (\partial_{\bf x'}\eta({\bf x'}) -\partial_{\bf y'} \eta({\bf x'}-{\bf y'}))\xi({\bf x'}-{\bf y'})d{\bf y'}\\
		&\quad-\int_{\mathbb{R}^2} \frac{| {\bf y'}|^2}{(| {\bf y'}|^2 + 4)^{5/2}} (\eta({\bf x'}) + \eta({\bf x'}-{\bf y'})) \tilde{q}_{l-1}(Q, \kappa_h)\partial_{\bf y'}\xi({\bf x'}-{\bf y'})d{\bf y'}.
	\end{align*}
    By Lemma \ref{LemmaA4}, for \( l = 2 \):
    \[
    \|\partial_{\mathbf{x'}} \tilde{Q}_{2,*}(\eta) \xi\|_{s,q} \leq C \|\eta\|_{C^{s+1}} \Big( \|\partial_{\mathbf{x'}} \eta\|_{s+1,q_1} \|\xi\|_{s+1,q_2} + \|\eta\|_{s+1,q_1} \|\partial_{\mathbf{x'}} \xi\|_{s+1,q_2} \Big),
    \]
    and for \( l \geq 3 \):
    \[
    \|\partial_{\mathbf{x'}} \tilde{Q}_{l,*}(\eta) \xi\|_{s,q} \leq C l^{s+1} 2^l \|\eta\|_{L^\infty}^{l-3} \|\eta\|_{C^{s+1}}^2 \Big( \|\partial_{\mathbf{x'}} \eta\|_{s+1,q_1} \|\xi\|_{q_2} + \|\eta\|_{s+1,q_1} \|\partial_{\mathbf{x'}} \xi\|_{q_2} \Big).
    \]

    \noindent \textbf{Case 4: \( \tilde{R}_{l,*} \) for \( l \geq 2 \).}
    After differentiation and simplification:
    \[
    \|\partial_{\mathbf{x'}} \tilde{R}_{l,*}(\eta) \xi\|_{s,q} \leq C l 2^l \|\eta\|_{L^\infty}^{l-2} \|\eta\|_{C^s} \Big( \|\partial_{\mathbf{x'}} \eta\|_{s,q_1} \|\xi\|_{q_2} + \|\eta\|_{s,q_1} \|\partial_{\mathbf{x'}} \xi\|_{q_2} \Big).
    \]

    \noindent \textbf{Final estimates:}
    \begin{itemize}
        \item \textbf{Operator \( A_2 \):}
\begin{align*}
    \|A_2(\eta)\xi\|_{s,q} &\leq C \Big( \|\partial_{\mathbf{x'}} \tilde{Q}_{2,*}(\eta)\xi\|_{s,q} + \|\partial_{\mathbf{x'}} \tilde{R}_{2,*}(\eta)\xi\|_{s,q} \Big)\\
    &\leq C \|\eta\|_{C^{s+1}} \Big( \|\partial_{\mathbf{x'}} \eta\|_{s+1,q_1} \|\xi\|_{s+1,q_2} + \|\eta\|_{s+1,q_1} \|\partial_{\mathbf{x'}} \xi\|_{s+1,q_2} \Big).
\end{align*}
        \item \textbf{Operator \( A_l \) (\( l \geq 3 \)):}
\begin{align*}
     \|A_l(\eta)\xi\|_{s,q} &\leq C \Big( \|\partial_{\mathbf{x'}} \tilde{P}_l(\eta)\xi\|_{s,q} + \|\partial_{\mathbf{x'}} \tilde{Q}_{l,*}(\eta)\xi\|_{s,q} + \|\partial_{\mathbf{x'}} \tilde{R}_{l,*}(\eta)\xi\|_{s,q} \Big) \\
     &\leq C l^{M+s+1} 2^l \|\eta\|_{C^1}^{l-3} \|\eta\|_{C^{s+1}}^2 \Big( \|\partial_{\mathbf{x'}} \eta\|_{s+1,q_1} \|\xi\|_{s+1,q_2} + \|\eta\|_{s+1,q_1} \|\partial_{\mathbf{x'}} \xi\|_{s+1,q_2} \Big).
\end{align*}
        \item \textbf{Operator \( A_L^R \):}
\begin{align*}
     \|A_L^R(\eta)\xi\|_{s,q} &\leq \sum_{j=L+1}^\infty \|A_j(\eta)\xi\|_{s,q} \\
     &\leq C L^{M+s+1} 2^L \|\eta\|_{C^1}^{L-3} \|\eta\|_{C^{s+1}}^2 \Big( \|\partial_{\mathbf{x'}} \eta\|_{s+1,q_1} \|\xi\|_{s+1,q_2} + \|\eta\|_{s+1,q_1} \|\partial_{\mathbf{x'}} \xi\|_{s+1,q_2} \Big).
\end{align*}
    \end{itemize}
    This completes the proof.
\end{proof}

Moreover, the proof of the preceding theorem yields the following uniform continuity estimate for the operator \(A_l\):

\begin{corollary} \label{cor A}
    Let \(\frac{1}{q} = \frac{1}{q_1} + \frac{1}{q_2}\). The operators \(A_1\), \(A_2\), \(A_l\), and \(A_L^R\) satisfy the estimates:
    \begin{align*}
        \left\|\frac{d}{dt} A_1(\eta_t) \xi \right\|_{s,q} 
            &\leq C \Big( \|\partial_{\mathbf{x}'} (\eta_2 - \eta_1)\|_{s,q_1} \|\partial_{\mathbf{x}'} \xi\|_{s,q_2} + \|\eta_2 - \eta_1\|_{s,q_1} \|\partial_{\mathbf{x}'}^2 \xi\|_{s,q_2} \Big), \\
        \left\|\frac{d}{dt} A_2(\eta_t) \xi \right\|_{s,q} 
            &\leq C \|\eta_t\|_{C^{s+1}} \Big( \|\partial_{\mathbf{x}'} (\eta_2 - \eta_1)\|_{s+1,q_1} \|\xi\|_{s+1,q_2} + \|\eta_2 - \eta_1\|_{s+1,q_1} \|\partial_{\mathbf{x}'} \xi\|_{s+1,q_2} \Big),
    \end{align*}
    and for all integers \(l, L \geq 3\),
    \begin{align*}
        \left\|\frac{d}{dt} A_l(\eta_t) \xi \right\|_{s,q} 
            &\leq  Cl^{M+s+2}2^l\|\eta_t\|_{C^1}^{l-3} \|\eta_t\|_{C^{s+1}}^2 \\
            &\quad \times \Big( \|\partial_{\mathbf{x}'} (\eta_2 - \eta_1)\|_{s+1,q_1} \|\xi\|_{s+1,q_2} + \|\eta_2 - \eta_1\|_{s+1,q_1} \|\partial_{\mathbf{x}'} \xi\|_{s+1,q_2} \Big), \\
        \left\|\frac{d}{dt} A_L^R(\eta_t) \xi \right\|_{s,q} 
            &\leq CL^{M+s+2}2^L \|\eta_t\|_{C^1}^{L-3} \|\eta_t\|_{C^{s+1}}^2 \\
            &\quad \times \Big( \|\partial_{\mathbf{x}'} (\eta_2 - \eta_1)\|_{s+1,q_1} \|\xi\|_{s+1,q_2} + \|\eta_2 - \eta_1\|_{s+1,q_1} \|\partial_{\mathbf{x}'} \xi\|_{s+1,q_2} \Big),
    \end{align*}
    where \(\eta_t = \eta_1 + t (\eta_2 - \eta_1)\) for \(t \in [0, 1]\), and \(C > 0\) is a constant.
\end{corollary}
\begin{proof}
As the derivation follows similar arguments to the proof of Theorem \ref{ThmA1'}, we present the detailed calculation only for $A_1$. The estimates for the remaining operators $A_2$, $A_l$ ($l \geq 3$), and $A_L^R$ ($L \geq 3$) follow analogously. From the operator expression:
\begin{align*}
\frac{d}{dt}A_1(\eta_t)\xi = D(\eta_2-\eta_1)\cdot D\xi - |D|(1+e^{-2|D|})^{-1}(\eta_2-\eta_1)(1+e^{-2|D|})G_0\xi,
\end{align*}
we obtain the norm estimate:
\[
\left\|\frac{d}{dt}A_1(\eta_t)\xi\right\|_{s,q} \leq C \Big( \|\partial_{\mathbf{x}'} (\eta_2-\eta_1)\|_{s,q_1} \|\partial_{\mathbf{x}'} \xi\|_{s,q_2} + \|\eta_2-\eta_1\|_{s,q_1} \|\partial_{\mathbf{x}'}^2 \xi\|_{s,q_2} \Big).
\]
\end{proof}

\begin{remark}\label{LemmaA6}
    Existing estimates for \(A_l\) and \(A_L^R\) in \cite{Craig2, Schanz} state that for \(1<q<\infty\) and \(s \in \mathbb{Z}_{\geq 0}\),
    \[
        \|A_l(\eta) \xi\|_{s,q} \leq C\|\eta\|_{C^1}^{l-1} \Big( \|\eta\|_{C^{s+1}} \|\xi\|_{1,q} + \|\eta\|_{C^1} \|\xi\|_{s+1,q} \Big),
    \]
    and
    \[
        \|A_L^R(\eta) \xi\|_{s,q} \leq C\|\eta\|_{C^1}^{L-1} \Big( \|\eta\|_{C^{s+1}} \|\xi\|_{1,q} + \|\eta\|_{C^1} \|\xi\|_{s+1,q} \Big).
    \]
    Theorem \ref{ThmA1'} improves these bounds by establishing sharper norm estimates, which provide enhanced control over the remainder term \(R_3\) in subsequent error analyses.
\end{remark}

\subsection{Estimate of the remainder operators \(R_1\), \(R_2\), and \(R_3\)}

This subsection establishes bounds for the remainder terms in the Dirichlet-Neumann operator expansion, extending Craig's results. We focus particularly on the Taylor remainder \(R_3\).

\begin{theorem} \label{Thm |D| R3}
    Let \(1 < q < \infty\) satisfy \(\frac{1}{q} = \frac{1}{q_1} + \frac{1}{q_2}\), and let \(s \in \mathbb{N}\) with \(s \geq 0\). Suppose \(\eta\) satisfies
    \[
        \|\eta\|_{C^1} < \frac{1}{4(1+C_0)}, \quad \|\eta\|_{C^{s+1}} < \infty, \quad \text{and} \quad \|\eta\|_{s,q_1} < \infty.
    \]
    Then the Taylor remainder \(R_3(\eta)\) from the expansion of the Dirichlet-Neumann operator \(G(\eta)\) satisfies
    \[
        \big\||D|^{-1} R_3(\eta) \xi\big\|_{s,q} \leq C \|\eta\|_{\max\{s,1\},q_1} \|\eta\|_{C^1} \|\eta\|_{C^{s+1}} \|\xi\|_{s+1,q_2},
    \]
    where \(C > 0\) is a constant independent of \(\eta\) and \(\xi\).
\end{theorem}

\begin{proof}
    We recall that 
    \begin{equation*}
        \begin{split}
            R_3 &= \big(B_3^R + B_1 (B_2 + B_3^R) + B_2 (B_1 + B_2 + B_3^R) + B_3^R (B_1 + B_2 + B_3^R)\big) G_0 \\
            &\quad + \big(B_2 + B_3^R\big) A_1 + (1 - B)^{-1} \big(A_3^R + B A_2 + B^2 A_1 + B^3 G_0\big) \\
            &= \big(B_3^R + B_1 B_2^R + B_2^R B_1^R\big) G_0 + B_2^R A_1 \\
            &\quad + (1 - B)^{-1} \big(A_3^R + B (A_2 + B A_1 + B^2 G_0)\big).
        \end{split}
    \end{equation*}
    By Remark \ref{LemmaA5} and Theorem \ref{ThmA2}, we have 
    \begin{align*}
        \||D|^{-1} B_3^{R}(\eta)(G_0 \xi) \|_{s,q} 
        &\le C \|\eta\|_{s,q_1} \|\eta\|_{C^1} \big( \|\eta\|_{C^1} \| \partial_{\mathbf{x}'} \xi \|_{s,q_2} + \|\eta\|_{C^{s+1}} \|\partial_{\mathbf{x}'} \xi\|_{q_2} \big) \\
        &\le C \|\eta\|_{s,q_1} \|\eta\|_{C^1} \|\eta\|_{C^{s+1}} \| \partial_{\mathbf{x}'} \xi \|_{s,q_2},
    \end{align*}
    \begin{equation*}
        \||D|^{-1} B_1(\eta) B_2^R(\eta) G_0 \xi \|_{s,q} \le C \|\eta\|_{s,q_1} \|\eta\|_{C^1}^2 \|\partial_{\mathbf{x}'} \xi\|_{q_2},
    \end{equation*}
    and 
    \begin{align*}
        \||D|^{-1} B_2^R(\eta) B_1^R(\eta) G_0 \xi \|_{s,q} 
        &\le C \|\eta\|_{s,q_1} \big( \|\eta\|_{C^1} \| B_1^R(\eta) G_0 \xi \|_{s,q_2} + \|\eta\|_{C^{s+1}} \| B_1^R(\eta) G_0 \xi \|_{q_2} \big) \\
        &\le C \|\eta\|_{s,q_1} \|\eta\|_{C^1} \|\eta\|_{C^{s+1}} \| \partial_{\mathbf{x}'} \xi \|_{s,q_2}.
    \end{align*}
    Using Remark \ref{LemmaA6} and Theorem \ref{ThmA2}, we obtain 
    \begin{align*}
        \||D|^{-1} B_2^R(\eta) A_1(\eta) \xi \|_{s,q}
        &\le C \|\eta\|_{s,q_1} \big( \|\eta\|_{C^1} \| A_1(\eta) \xi \|_{s,q_2} + \|\eta\|_{C^{s+1}} \| A_1(\eta) \xi \|_{q_2} \big) \\
        &\le C \|\eta\|_{s,q_1} \|\eta\|_{C^1} \big( \|\eta\|_{C^{s+1}} \|\xi\|_{1,q_2} + \|\eta\|_{C^1} \|\xi\|_{s+1,q_2} \big).
    \end{align*}
    From Lemma \ref{LemmaA7}, the Neumann series 
    \begin{equation*}
        |D|^{-1} (1 - B)^{-1} A_3^R = |D|^{-1} A_3^R + |D|^{-1} B \sum_{m=0}^{\infty} B^m A_3^R 
    \end{equation*}
    converges for $\|\eta\|_{C^1} < 1/(1 + C_0)$, since the geometric series $\sum_{m=1}^{\infty} (1 + C_0)^{m-1} \|\eta\|_{C^1}^{m-1}$ and $\sum_{m=1}^{\infty} \|\eta\|_{C^1}^m$ converge in this region. By Theorems \ref{ThmA1} and \ref{ThmA2}, we get
    \begin{align*}
        &\||D|^{-1} (1 - B(\eta))^{-1} A_3^R(\eta) \xi \|_{s,q} \\
        &\le \||D|^{-1} A_3^R(\eta) \xi \|_{s,q} + \Big\| |D|^{-1} B \sum_{m=0}^{\infty} B^m A_3^R(\eta) \xi \Big\|_{s,q} \\
        &\le C \|\eta\|_{C^1} \|\eta\|_{\max\{s,1\},q_1} \big( \|\eta\|_{C^1} \|\partial_{\mathbf{x}'} \xi\|_{s,q_2} + \|\eta\|_{C^{s+1}} \|\xi\|_{1,q_2} \big) \\
        &\quad + C \|\eta\|_{s,q_1} \Big( \Big\| \sum_{m=0}^{\infty} B^m A_3^R(\eta) \xi \Big\|_{q_2} + \|\eta\|_{C^1} \Big\| \sum_{m=0}^{\infty} B^m A_3^R(\eta) \xi \Big\|_{s,q_2} \\
        &\qquad + \|\eta\|_{C^{s+1}} \Big\| \sum_{m=0}^{\infty} B^m A_3^R(\eta) \xi \Big\|_{q_2} \Big).
    \end{align*}
    Lemma \ref{LemmaA7} and Remark \ref{LemmaA6} yield
    \begin{align*}
        \Big\| \sum_{m=0}^{\infty} B^m A_3^R(\eta) \xi \Big\|_{s,q_2} 
        &\le \| A_3^R(\eta) \xi \|_{s,q_2} + \sum_{m=1}^{\infty} \| B^m A_3^R(\eta) \xi \|_{s,q_2} \\
        &\le \| A_3^R(\eta) \xi \|_{s,q_2} + \sum_{m=1}^{\infty} C (1 + C_0)^{m-1} \|\eta\|_{C^1}^{m-1} \\
        &\quad \times \big( \|\eta\|_{C^{s+1}} \| A_3^R(\eta) \xi \|_{q_2} + \|\eta\|_{C^1} \| A_3^R(\eta) \xi \|_{s,q_2} \big) \\
        &\le C \|\eta\|_{C^1}^2 \big( \|\eta\|_{C^{s+1}} \|\xi\|_{1, q_2} + \|\eta\|_{C^1} \|\xi\|_{s+1, q_2} \big).
    \end{align*}
    Therefore,
    \begin{align*}
        &\||D|^{-1} (1 - B(\eta))^{-1} A_3^R(\eta) \xi \|_{s,q} \\
        &\le C \|\eta\|_{C^1} \|\eta\|_{\max\{s,1\},q_1} \big( \|\eta\|_{C^1} \|\partial_{\mathbf{x}'} \xi\|_{s,q_2} + \|\eta\|_{C^{s+1}} \|\xi\|_{1,q_2} \big) \\
        &\quad + C \|\eta\|_{s,q_1} \|\eta\|_{C^1}^3 \big( \|\xi\|_{1, q_2} + \|\eta\|_{C^{s+1}} \|\xi\|_{1, q_2} + \|\eta\|_{C^1} \|\xi\|_{s+1, q_2} \big).
    \end{align*}
    Similarly, Lemma \ref{LemmaA7} implies the Neumann series
    \begin{align*}
        &|D|^{-1} (1 - B)^{-1} B (A_2 + B A_1 + B^2 G_0) \\
        &= |D|^{-1} B (A_2 + B A_1 + B^2 G_0) + |D|^{-1} B \sum_{m=1}^{\infty} B^m (A_2 + B A_1 + B^2 G_0)
    \end{align*}
    converges for $\|\eta\|_{C^1} < 1/(1 + C_0)$. Theorems \ref{ThmA1} and \ref{ThmA2} give
    \begin{align*}
        &\||D|^{-1} (1 - B)^{-1} B (A_2 + B A_1 + B^2 G_0) \xi \|_{s,q} \\
        &\le \||D|^{-1} B (A_2 + B A_1 + B^2 G_0) \xi \|_{s,q} + \Big\| |D|^{-1} B \sum_{m=1}^{\infty} B^m (A_2 + B A_1 + B^2 G_0) \xi \Big\|_{s,q} \\
        &\le C \|\eta\|_{s,q_1} \Big( \| (A_2 + B A_1 + B^2 G_0) \xi \|_{q_2} + \Big\| \sum_{m=1}^{\infty} B^m (A_2 + B A_1 + B^2 G_0) \xi \Big\|_{q_2} \Big) \\
        &\quad + C \|\eta\|_{s,q_1} \|\eta\|_{C^1} \Big( \| (A_2 + B A_1 + B^2 G_0) \xi \|_{s,q_2} + \Big\| \sum_{m=1}^{\infty} B^m (A_2 + B A_1 + B^2 G_0) \xi \Big\|_{s,q_2} \Big) \\
        &\quad + C \|\eta\|_{s,q_1} \|\eta\|_{C^{s+1}} \Big( \| (A_2 + B A_1 + B^2 G_0) \xi \|_{q_2} + \Big\| \sum_{m=1}^{\infty} B^m (A_2 + B A_1 + B^2 G_0) \xi \Big\|_{q_2} \Big).
    \end{align*}
    Remark \ref{LemmaA6} and Lemma \ref{LemmaA7} imply
    \begin{align*}
        \| (A_2 + B A_1 + B^2 G_0) \xi \|_{s,q_2} 
        &\le C \|\eta\|_{C^1} \big( \|\eta\|_{C^{s+1}} \|\xi\|_{1,q_2} + \|\eta\|_{C^1} \|\xi\|_{s+1,q_2} \big),
    \end{align*}
    and
    \begin{align*}
        & \Big\| \sum_{m=1}^{\infty} B^m (A_2 + B A_1 + B^2 G_0) \xi \Big\|_{s,q_2} \\
        &\le C \sum_{m=1}^{\infty} (1 + C_0)^{m-1} \|\eta\|_{C^1}^{m-1} \\
        &\quad \times \big( \|\eta\|_{C^{s+1}} \| (A_2 + B A_1 + B^2 G_0) \xi \|_{q_2} + \|\eta\|_{C^1} \| (A_2 + B A_1 + B^2 G_0) \xi \|_{s,q_2} \big) \\
        &\le C \|\eta\|_{C^1}^2 \big( \|\eta\|_{C^{s+1}} \|\xi\|_{1,q_2} + \|\eta\|_{C^1} \|\xi\|_{s+1,q_2} \big).
    \end{align*}
    Hence,
    \begin{align*}
        &\||D|^{-1} (1 - B)^{-1} B (A_2 + B A_1 + B^2 G_0) \xi \|_{s,q} \\
        &\le C \|\eta\|_{s,q_1} \|\eta\|_{C^1}^2 \big( \|\xi\|_{1,q_2} + \|\eta\|_{C^{s+1}} \|\xi\|_{1,q_2} + \|\eta\|_{C^1} \|\xi\|_{s+1,q_2} \big).
    \end{align*}
    Combining all estimates, we obtain
    \begin{align*}
        \||D|^{-1} R_3(\eta) \xi \|_{s, q} 
        &\leq C \|\eta\|_{s,q_1} \|\eta\|_{C^1} \|\eta\|_{C^{s+1}} \| \partial_{\mathbf{x}'} \xi \|_{s,q_2} \\
        &\quad + C \|\eta\|_{s,q_1} \|\eta\|_{C^1} \big( \|\eta\|_{C^{s+1}} \|\xi\|_{1,q_2} + \|\eta\|_{C^1} \|\xi\|_{s+1,q_2} \big) \\
        &\quad + C \|\eta\|_{C^1} \|\eta\|_{\max\{s,1\},q_1} \big( \|\eta\|_{C^1} \|\partial_{\mathbf{x}'} \xi\|_{s,q_2} + \|\eta\|_{C^{s+1}} \|\xi\|_{1,q_2} \big) \\
        &\quad + C \|\eta\|_{s,q_1} \|\eta\|_{C^1}^3 \big( \|\xi\|_{1, q_2} + \|\eta\|_{C^{s+1}} \|\xi\|_{1, q_2} + \|\eta\|_{C^1} \|\xi\|_{s+1, q_2} \big) \\
        &\quad + C \|\eta\|_{s,q_1} \|\eta\|_{C^1}^2 \big( \|\xi\|_{1,q_2} + \|\eta\|_{C^{s+1}} \|\xi\|_{1,q_2} + \|\eta\|_{C^1} \|\xi\|_{s+1,q_2} \big) \\
        &\le C \|\eta\|_{\max\{s,1\},q_1} \|\eta\|_{C^1} \|\eta\|_{C^{s+1}} \|\xi\|_{s+1,q_2}.
    \end{align*}
    This completes the proof.
\end{proof}

Using estimates similar to those in the previous theorem, we obtain the following uniform continuity result for the remainder term $|D|^{-1} R_3$.

\begin{corollary} \label{cor D R3}
    Let \(1 < q < \infty\) with \(\frac{1}{q} = \frac{1}{q_1} + \frac{1}{q_2}\), and let \(s \in \mathbb{N}\), \(s \geq 0\). Suppose \(\eta_1, \eta_2\) satisfy 
    \[
        \|\eta_i\|_{C^1} < \frac{1}{4(1+C_0)}, \quad \|\eta_i\|_{C^{s+1}} < \infty, \quad \text{and} \quad \|\eta_i\|_{s,q_1} < \infty
    \]
    for \(i = 1,2\). Then for the Taylor remainder \(R_3(\eta)\) in the expansion of the Dirichlet-Neumann operator \(G(\eta)\), we have
    \[
        \left\| \frac{d}{dt} |D|^{-1} R_3(\eta_t) \xi \right\|_{s, q} 
        \leq C \| \eta_2 - \eta_1 \|_{\max\{s,1\},q_1} \| \eta_t \|_{C^1} \| \eta_t \|_{C^{s+1}} \| \xi \|_{s+1,q_2}
    \]
    where \(\eta_t = (1-t)\eta_1 + t\eta_2\) and \(t \in [0,1]\), with \(C > 0\) independent of \(\eta_1, \eta_2, \xi\), and \(t\).
\end{corollary}
\begin{proof}
Using the explicit form of the $R_3$ operator, this result is established by applying 
Corollaries \ref{cor3.1}, \ref{Cor DA}, and \ref{cor A} through the same analytical 
framework developed in the preceding theorem's proof.
\end{proof}

We need to make a slight modification to the norm of $ R_3 $ to align with our nonlinear term estimates in Section \ref{sec5}.

\begin{theorem}\label{Thm R_3}
    Let \(1 < q < \infty\) with \(\frac{1}{q} = \frac{1}{q_1} + \frac{1}{q_2}\), and let \(s \in \mathbb{N}\), \(s \geq 0\). Assume \(\eta\) satisfies 
    \[
        \|\eta\|_{C^1} < \frac{1}{4(1+C_0)}, \quad \|\eta\|_{C^{s+1}} < \infty, \quad \text{and} \quad \|\eta\|_{s,q_1} < \infty.
    \]
    Then the remainder term \(R_3(\eta)\) satisfies the estimate:
    \[
        \|R_3(\eta)\xi\|_{s,q} \leq C \|\eta\|_{C^{s+1}}^2 \left( \|\partial_{\mathbf{x}'} \eta\|_{s+1,q_1} \|\xi\|_{s+1,q_2} + \|\eta\|_{s+1,q_1} \|\partial_{\mathbf{x}'} \xi\|_{s+1,q_2} \right)
    \]
    where \(C > 0\) is independent of \(\eta\) and \(\xi\).
\end{theorem}

\begin{proof}
    Recall the expression for \( R_3 \):
    \[
    R_3 = \big(B_3^R + B_1 B_2^R + B_2^R B_1^R\big) G_0 + B_2^R A_1 + (1 - B)^{-1} \big(A_3^R + B (A_2 + B A_1 + B^2 G_0)\big).
    \]
    Applying Theorem \ref{ThmA2'}, we estimate the first term:
    \begin{align*}
        \|B_3^{R}(\eta)(G_0\xi)\|_{s,q}
        &\le C\|\eta\|_{C^1}\|\eta\|_{C^{s+1}} \left(\|\partial_{\bf x'}\eta\|_{s,q_1}\|G_0\xi\|_{s,q_2} + \|\eta\|_{s,q_1}\|\partial_{\bf x'} G_0\xi\|_{s,q_2}\right) \\
        &\le C\|\eta\|_{C^1}\|\eta\|_{C^{s+1}} \left(\|\partial_{\bf x'}\eta\|_{s,q_1}\|\partial_{\bf x'} \xi\|_{s,q_2} + \|\eta\|_{s,q_1}\|\partial_{\bf x'}^2\xi\|_{s,q_2}\right).
    \end{align*}
    For the second term, combining Remark \ref{LemmaA5} and Theorem \ref{ThmA2'} yields
    \begin{align*}
        \|B_1(\eta)B_2^R(\eta)G_0\xi\|_{s,q}
        &\le C\|\eta\|_{C^{s+1}}\|B_2^R(\eta)G_0\xi\|_{q} + \|\eta\|_{C^1}\|B_2^R(\eta)G_0\xi\|_{s,q} \\
        &\le C\|\eta\|_{C^1}\|\eta\|_{C^{s+1}} \big(\|\partial_{\bf x'}\eta\|_{q_1}\|G_0\xi\|_{q_2} + \|\eta\|_{q_1}\|\partial_{\bf x'} G_0\xi\|_{q_2}\big) \\
        &\quad + C\|\eta\|_{C^1}\|\eta\|_{C^{s+1}} \big(\|\partial_{\bf x'}\eta\|_{s,q_1}\|G_0\xi\|_{s,q_2} + \|\eta\|_{s,q_1}\|\partial_{\bf x'} G_0\xi\|_{s,q_2}\big) \\
        &\le C\|\eta\|_{C^1}\|\eta\|_{C^{s+1}} \big(\|\partial_{\bf x'}\eta\|_{s,q_1}\|\partial_{\bf x'}\xi\|_{s,q_2} + \|\eta\|_{s,q_1}\|\partial_{\bf x'}^2\xi\|_{s,q_2}\big).
    \end{align*}
    For the third term, since \(B_1^R = B_2 + B_2^R\), Remark \ref{LemmaA5} gives
    \begin{align*}
        \|B_2^R B_1^R(G_0\xi)\|_{s,q}
        & \le  C \|\eta\|_{C^1} \left(\|\eta\|_{C^{s+1}}\|B_1^R(G_0\xi)\|_{q} + \|\eta\|_{C^1}\|B_1^R(G_0\xi)\|_{s,q}\right).
    \end{align*}
    Furthermore, Theorem \ref{ThmA2'} implies
    \begin{align*}
        \|B_1^R(G_0\xi)\|_{s,q}
        & \le  \|B_2(G_0\xi)\|_{s,q} + \|B_2^R(G_0\xi)\|_{s,q} \\
        & \le  C\|\eta\|_{C^{s+1}} \left( \|\partial_{\bf x'}\eta\|_{s,q_1}\|\partial_{\bf x'}\xi\|_{s,q_2} + \|\eta\|_{s,q_1}\|\partial_{\bf x'}^2\xi\|_{s,q_2} \right).
    \end{align*}
    Combining these yields
    \begin{align*}
        \|B_2^R B_1^R(G_0\xi)\|_{s,q}
        &\le C \|\eta\|_{C^1}^2\|\eta\|_{C^{s+1}} \left( \|\partial_{\bf x'}\eta\|_{s,q_1}\|\partial_{\bf x'}\xi\|_{s,q_2} + \|\eta\|_{s,q_1}\|\partial_{\bf x'}^2\xi\|_{s,q_2} \right).
    \end{align*}
    Next, using Remark \ref{LemmaA5} and Theorem \ref{ThmA1'}, we estimate \(B_2^R A_1\):
    \begin{align*}
        \|B_2^R(\eta)A_1(\eta)\xi\|_{s,q}
        &\le C\|\eta\|_{C^1}\|\eta\|_{C^{s+1}}\|A_1(\eta)\xi\|_{s,q} \\
        &\le C\|\eta\|_{C^1}\|\eta\|_{C^{s+1}} \big(\|\partial_{\bf x'}\eta\|_{s,q_1}\|\partial_{\bf x'}\xi\|_{s,q_2} + \|\eta\|_{s,q_1}\|\partial_{\bf x'}^2\xi\|_{s,q_2}\big).
    \end{align*}
    We now turn to the terms involving \((1 - B)^{-1}\). From the estimate on \(B^m(\eta)\) in Lemma \ref{LemmaA7}, the Neumann series
    \[
    (1 - B)^{-1} A_3^R = A_3^R + \sum_{m=1}^{\infty} B^m A_3^R
    \]
    converges absolutely for \(\|\eta\|_{C^1} < 1 / (1 + C_0)\), as the series \(\sum_{m=1}^{\infty} (1 + C_0)^{m-1} \|\eta\|_{C^1}^{m-1}\) and \(\sum_{m=1}^{\infty} \|\eta\|_{C^1}^m\) converge within this radius. Applying Lemma \ref{LemmaA7} and Theorem \ref{ThmA1'} gives
    \begin{align*}
        &\|(1 - B(\eta))^{-1} A_{3}^R(\eta)\xi\|_{s,q}\\
        &\le  \|A_{3}^R(\eta)\xi\|_{s,q} + \left\|\sum_{m=1}^{\infty} B^m A_3^R(\eta)\xi\right\|_{s,q} \\
        & \le  \|A_{3}^R(\eta)\xi\|_{s,q} + \sum_{m=1}^{\infty} (1 + C_0)^{m-1}\|\eta\|_{C^1}^{m-1} \left( \|\eta\|_{C^{s+1}} \|A_3^R(\eta)\xi\|_{q} +  \|\eta\|_{C^1} \|A_3^R(\eta)\xi\|_{s,q}  \right) \\
        &\le  C\|A_{3}^R(\eta)\xi\|_{s,q} + C\|\eta\|_{C^{s+1}} \|A_3^R(\eta)\xi\|_{q} \\
        &\le C\|\eta\|_{C^{s+1}}^2 \left(\|\partial_{\bf x'}\eta\|_{s+1,q_1}\|\xi\|_{s+1,q_2} + \|\eta\|_{s+1,q_1}\|\partial_{\bf x'}\xi\|_{s+1,q_2}\right).
    \end{align*}
    Similarly, for the remaining part of the inverse term, Lemma \ref{LemmaA7} and Theorem \ref{ThmA1'} imply
    \begin{align*}
        &\|(1 - B)^{-1} B(A_2 + B A_1 + B^2 G_0)\xi\|_{s,q} \\
        &\le  C\|B(A_2 + B A_1 + B^2 G_0)\xi\|_{s,q} + C\|\eta\|_{C^{s+1}} \|B(A_2 + B A_1 + B^2 G_0)\xi\|_{q} \\
        &\le  C\|\eta\|_{C^1} \left( \|A_2\xi\|_{s,q} + \|B A_1\xi\|_{s,q} + \|B^2 G_0\xi\|_{s,q} \right) \\
        &\quad + C\|\eta\|_{C^{s+1}} \left( \|A_2\xi\|_{q} + \|B A_1\xi\|_{q} + \|B^2 G_0\xi\|_{q} \right),
    \end{align*}
    where the individual estimates are
    \begin{align*}
        \|A_2\xi\|_{s,q} &\le C\|\eta\|_{C^{s+1}} \left( \|\partial_{\bf x'}\eta\|_{s+1,q_1}\|\xi\|_{s+1,q_2} + \|\eta\|_{s+1,q_1}\|\partial_{\bf x'}\xi\|_{s+1,q_2} \right), \\
        \|B A_1\xi\|_{s,q} &\le  C \left( \|\eta\|_{C^{s+1}} \|A_1\xi\|_{q} + \|\eta\|_{C^1} \|A_1\xi\|_{s,q} \right) \\
        &\le  C\|\eta\|_{C^{s+1}} \left( \|\partial_{\bf x'}\eta\|_{s, q_1}\|\partial_{\bf x'}\xi\|_{s, q_2} + \|\eta\|_{s, q_1}\|\partial_{\bf x'}^2\xi\|_{s, q_2} \right), \\
        \intertext{and, using Lemma \ref{LemmaA7} and Theorem \ref{ThmA2'},}
        \|B^2 G_0\xi\|_{s,q} &\le  C \|\eta\|_{C^{s+1}} \|B G_0\xi\|_{q} + C \|\eta\|_{C^{1}} \|B G_0\xi\|_{s,q} \\
        &\le  C \|\eta\|_{C^{s+1}} \left( \|\partial_{\bf x'}\eta\|_{q_1}\|\partial_{\bf x'}\xi\|_{q_2} + \|\eta\|_{q_1}\|\partial_{\bf x'}^2\xi\|_{q_2} \right) \\
        &\quad + C \|\eta\|_{C^{1}} (1 + \|\eta\|_{C^{s+1}}) \left( \|\partial_{\bf x'}\eta\|_{s, q_1}\|\partial_{\bf x'}\xi\|_{s, q_2} + \|\eta\|_{s, q_1}\|\partial_{\bf x'}^2\xi\|_{s, q_2} \right) \\
        &\le  C \|\eta\|_{C^{s+1}} \left( \|\partial_{\bf x'}\eta\|_{s, q_1}\|\partial_{\bf x'}\xi\|_{s, q_2} + \|\eta\|_{s, q_1}\|\partial_{\bf x'}^2\xi\|_{s, q_2} \right).
    \end{align*}
    Combining these estimates leads to
    \begin{align*}
        &\|(1 - B)^{-1} B(A_2 + B A_1 + B^2 G_0)\xi\|_{s,q} \\
        &\le  C \|\eta\|_{C^1} \|\eta\|_{C^{s+1}} \left( \|\partial_{\bf x'}\eta\|_{s+1,q_1}\|\xi\|_{s+1,q_2} + \|\eta\|_{s+1,q_1}\|\partial_{\bf x'}\xi\|_{s+1,q_2} \right).
    \end{align*}
    Summing up all the estimates for the constituent parts of \(R_3\), we conclude that
    \begin{align*}
        \|R_3(\eta)\xi\|_{s,q} \le  C \|\eta\|_{C^{s+1}}^2 \left( \|\partial_{\bf x'}\eta\|_{s+1,q_1}\|\xi\|_{s+1,q_2} + \|\eta\|_{s+1,q_1}\|\partial_{\bf x'}\xi\|_{s+1,q_2} \right),
    \end{align*}
    completing the proof.
\end{proof}

Regarding the equi-continuity property of $R_3$, we establish the following result.

\begin{corollary}\label{CorR3}
    Let \(1 < q < +\infty\) with \(\frac{1}{q} = \frac{1}{q_1} + \frac{1}{q_2}\), and let \(s \in \mathbb{N}\), \(s \geq 0\). 
    Suppose \(\eta_1, \eta_2\) satisfy 
    \[
        \|\eta_i\|_{C^{1}} < \frac{1}{4(1+C_0)}, \quad \|\eta_i\|_{C^{s+1}} < +\infty, \quad \|\eta_i\|_{s,q_1} < +\infty, \quad i=1,2.
    \]
    Then for \(\eta_t = \eta_1 + t(\eta_2 - \eta_1)\) with \(t \in [0,1]\), the following estimate holds:
    \begin{align*}
        \left\|\frac{d}{dt}R_3(\eta_t)\xi\right\|_{s,q} 
        &\le C \|\eta_t\|_{C^{s+1}}^2 \|\partial_{\mathbf{x}'}(\eta_2 - \eta_1)\|_{s+1,q_1} \|\xi\|_{s+1,q_2} \\
        &\quad + C \|\eta_t\|_{C^{s+1}}^2 \|\eta_2 - \eta_1\|_{s+1,q_1} \|\partial_{\mathbf{x}'}\xi\|_{s+1,q_2},
    \end{align*}
    where the constant \(C > 0\) depends only on \(q, s\).
\end{corollary}
\begin{proof}
    Combining Corollaries \ref{cor3.1}, \ref{Cor DA}, and \ref{cor A} with the explicit expression for the $R_3$ operator, 
we complete the proof through direct application of the methodology established above.
\end{proof}

Given the norm estimate for \(R_3\) and the relation \(R_2 = G_2 + R_3\), we derive the following boundedness result for \(|D|^{-1}R_2\).

\begin{theorem}\label{ThmA3}
    Let \(1 < q < +\infty\) with \(\frac{1}{q} = \frac{1}{q_1} + \frac{1}{q_2}\), and let \(s \in \mathbb{N}\), \(s \geq 0\). 
    Assume \(\eta\) satisfies 
    \[
        \|\eta\|_{C^{1}} < \frac{1}{1 + C_0} \quad \text{and} \quad \|\eta\|_{C^{s+1}} < +\infty.
    \]
    Then the Taylor remainder \(R_2(\eta)\) in the expansion of \(G(\eta)\) admits the estimate:
    \begin{align*}
        \||D|^{-1}R_2(\eta)\xi\|_{s,q} 
        &\le C\|\eta\|_{\max\{s,1\},q_1} \left( \|\eta\|_{C^{1}} \|\partial_{\mathbf{x}'}\xi\|_{s,q_2} + \|\eta\|_{C^{s+1}} \|\xi\|_{s+1,q_2} \right) \\
        &\quad + C\|\eta\|_{s,q_1} \|\eta\|_{L^\infty} \|\partial_{\mathbf{x}'}^2\xi\|_{q_2},
    \end{align*}
    where the constant \(C > 0\) depends only on \(q, s\).
\end{theorem}

\begin{proof}
    Starting from the decomposition 
    \[ 
    R_2(\eta)\xi = G_2(\eta)\xi + R_3(\eta)\xi
    \]
    with 
    \[ 
    G_2(\eta) = B_2 G_0 + B_1 A_1 + A_2 + B_1^2 G_0,
    \]
    we estimate each component separately.
    For $|D|^{-1}B_2G_0$, Theorem \ref{ThmA2} yields
    \begin{align*}
        \||D|^{-1}B_2(\eta)G_0\xi\|_{s,q} 
        &\le C\|\eta\|_{s,q_1} \left( \|\eta\|_{C^{1}} \|G_0\xi\|_{s,q_2} + \|\eta\|_{C^{s+1}} \|G_0\xi\|_{q_2} \right) \\
        &\le C\|\eta\|_{s,q_1} \left( \|\eta\|_{C^{1}} \|\partial_{\mathbf{x}'}\xi\|_{s,q_2} + \|\eta\|_{C^{s+1}} \|\partial_{\mathbf{x}'}\xi\|_{q_2} \right).
    \end{align*}
    For $|D|^{-1}B_1A_1$, first apply Theorem \ref{ThmA1'}:
    \[
        \|A_1(\eta)\xi\|_{q_2} \le C \left( \|\partial_{\mathbf{x}'}\eta\|_{L^\infty} \|\partial_{\mathbf{x}'}\xi\|_{q_2} + \|\eta\|_{L^\infty} \|\partial_{\mathbf{x}'}^2\xi\|_{q_2} \right),
    \]
    which implies
\begin{align*}
     \||D|^{-1}B_1(\eta)A_1(\eta)\xi\|_{s,q} &\le C\|\eta\|_{s,q_1} \|A_1(\eta)\xi\|_{q_2} \\
     &\le C\|\eta\|_{s,q_1} \left( \|\partial_{\mathbf{x}'}\eta\|_{L^\infty} \|\partial_{\mathbf{x}'}\xi\|_{q_2} + \|\eta\|_{L^\infty} \|\partial_{\mathbf{x}'}^2\xi\|_{q_2} \right).
\end{align*}
    For $|D|^{-1}A_2$, Theorem \ref{ThmA1} gives
    \begin{align*}
        \||D|^{-1}A_2(\eta)\xi\|_{s,q} \le C \|\eta\|_{\max\{s,1\},q_1} \left( \|\eta\|_{C^{1}} \|\partial_{\mathbf{x}'}\xi\|_{s,q_2} + \|\eta\|_{C^{s+1}} \|\xi\|_{1,q_2} \right).
    \end{align*}
    For $|D|^{-1}B_1^2G_0$, combine Theorem \ref{ThmA2} and Remark \ref{LemmaA5}:
    \[ 
        \||D|^{-1}B_1^2(\eta)G_0\xi\|_{s,q} \le C\|\eta\|_{s,q_1} \|B_1(\eta)G_0\xi\|_{q_2} \le C\|\eta\|_{s,q_1} \|\eta\|_{C^{1}} \|\partial_{\mathbf{x}'}\xi\|_{q_2}.
    \]
    For the remainder term, Theorem \ref{Thm |D| R3} provides
    \begin{align*}
        \||D|^{-1}R_3(\eta)\xi\|_{s,q} \le C \|\eta\|_{\max\{s,1\},q_1} \|\eta\|_{C^{1}} \|\eta\|_{C^{s+1}} \|\xi\|_{s+1,q_2}.
    \end{align*}
    Combining these estimates, we obtain
    \begin{align*}
        \||D|^{-1}R_2(\eta)\xi\|_{s,q} 
        &\le \sum \||D|^{-1}G_2(\eta)\xi\|_{s,q} + \||D|^{-1}R_3(\eta)\xi\|_{s,q} \\
        &\le C \|\eta\|_{\max\{s,1\},q_1} \left( \|\eta\|_{C^{1}} \|\partial_{\mathbf{x}'}\xi\|_{s,q_2} + \|\eta\|_{C^{s+1}} \|\xi\|_{1,q_2} \right) \\
        &\quad + C\|\eta\|_{s,q_1} \left( \|\partial_{\mathbf{x}'}\eta\|_{L^\infty} \|\partial_{\mathbf{x}'}\xi\|_{q_2} + \|\eta\|_{L^\infty} \|\partial_{\mathbf{x}'}^2\xi\|_{q_2} \right) \\
        &\quad + C \|\eta\|_{\max\{s,1\},q_1} \|\eta\|_{C^{1}} \|\eta\|_{C^{s+1}} \|\xi\|_{s+1,q_2}.
    \end{align*}
    Since $\|\xi\|_{1,q_2} \leq \|\xi\|_{s+1,q_2}$, we can consolidate terms to achieve the final estimate:
    \begin{align*}
        \||D|^{-1}R_2(\eta)\xi\|_{s,q} 
        &\le C\|\eta\|_{\max\{s,1\},q_1} \left( \|\eta\|_{C^{1}} \|\partial_{\mathbf{x}'}\xi\|_{s,q_2} + \|\eta\|_{C^{s+1}} \|\xi\|_{s+1,q_2} \right) \\
        &\quad + C\|\eta\|_{s,q_1} \|\eta\|_{L^\infty} \|\partial_{\mathbf{x}'}^2\xi\|_{q_2},
    \end{align*}
    completing the proof.
\end{proof}

\begin{theorem}\label{Thm R_2}
    Let \(1 < q < \infty\) satisfy \(\frac{1}{q} = \frac{1}{q_1} + \frac{1}{q_2}\), and let \(s \in \mathbb{N}\), \(s \geq 0\). 
    Suppose \(\eta\) satisfies 
    \[
        \|\eta\|_{C^{1}} < \frac{1}{1 + C_0} \quad \text{and} \quad \|\eta\|_{C^{s+1}} < \infty.
    \]
    Then the remainder term \(R_2(\eta)\) admits the estimate:
    \begin{align*}
        \|R_2(\eta)\xi\|_{s,q} 
        \le C(1 + \|\eta\|_{C^{s+1}}) \|\eta\|_{C^{s+1}} \Big( & \|\partial_{\mathbf{x}'}\eta\|_{s+1,q_1} \|\xi\|_{s+1,q_2} \\
          & + \|\eta\|_{s+1,q_1} \|\partial_{\mathbf{x}'}\xi\|_{s+1,q_2} \Big),
    \end{align*}
    where \(C > 0\) is a constant depending on \(q, s\).
\end{theorem}

\begin{proof}
    Starting from the decomposition 
    \[ 
    R_2(\eta)\xi = G_2(\eta)\xi + R_3(\eta)\xi
    \]
    with 
    \[ 
    G_2(\eta) = B_2 G_0 + B_1 A_1 + A_2 + B_1^2 G_0,
    \]
    we establish estimates for each component.
   
\vspace{0.5em}
\noindent \textbf{Estimate for \(B_2G_0\xi\):} By Theorem \ref{ThmA2'},
    \begin{align*}
        \|B_2(\eta) G_0 \xi\|_{s,q} 
        \le C \|\eta\|_{C^{s+1}} \Big( \|\partial_{\mathbf{x}'}\eta\|_{s,q_1} \|\partial_{\mathbf{x}'}\xi\|_{s,q_2} + \|\eta\|_{s,q_1} \|\partial_{\mathbf{x}'}^2 \xi\|_{s,q_2} \Big).
    \end{align*}

\vspace{0.5em}
\noindent
    \textbf{Estimate for \(B_1A_1\xi\):} Combining Remark \ref{LemmaA5} and Theorem \ref{ThmA1'} yields
    \begin{align*}
        \|B_1(\eta) A_1(\eta) \xi\|_{s,q} 
        &\le C \|\eta\|_{C^{s+1}} \|A_1(\eta) \xi\|_{s,q} \\
        &\le C \|\eta\|_{C^{s+1}} \Big( \|\partial_{\mathbf{x}'}\eta\|_{s,q_1} \|\partial_{\mathbf{x}'}\xi\|_{s,q_2} + \|\eta\|_{s,q_1} \|\partial_{\mathbf{x}'}^2 \xi\|_{s,q_2} \Big).
    \end{align*}

\vspace{0.5em}
\noindent
    \textbf{Estimate for \(A_2\xi\):} Applying Theorem \ref{ThmA1'} gives
    \begin{align*}
        \|A_2(\eta) \xi\|_{s,q} 
        \le C \|\eta\|_{C^{s+1}} \Big( \|\partial_{\mathbf{x}'}\eta\|_{s+1,q_1} \|\xi\|_{s+1,q_2} + \|\eta\|_{s+1,q_1} \|\partial_{\mathbf{x}'}\xi\|_{s+1,q_2} \Big).
    \end{align*}

\vspace{0.5em}
\noindent
    \textbf{Estimate for \(B_1^2G_0\xi\):} Using Remark \ref{LemmaA5} and Theorem \ref{ThmA2'},
    \begin{align*}
        \|B_1^2(\eta) G_0 \xi\|_{s,q} 
        &\le C \|\eta\|_{C^{s+1}} \|B_1(\eta) G_0 \xi\|_{s,q} \\
        &\le C \|\eta\|_{C^{s+1}} \Big( \|\partial_{\mathbf{x}'}\eta\|_{s,q_1} \|\partial_{\mathbf{x}'}\xi\|_{s,q_2} + \|\eta\|_{s,q_1} \|\partial_{\mathbf{x}'}^2 \xi\|_{s,q_2} \Big).
    \end{align*}

\vspace{0.5em}
\noindent
    \textbf{Estimate for \(R_3\xi\):} Theorem \ref{Thm R_3} provides
    \begin{align*}
        \|R_3(\eta)\xi\|_{s,q} 
        \le C \|\eta\|_{C^{s+1}}^2 \Big( \|\partial_{\mathbf{x}'}\eta\|_{s+1,q_1} \|\xi\|_{s+1,q_2} + \|\eta\|_{s+1,q_1} \|\partial_{\mathbf{x}'}\xi\|_{s+1,q_2} \Big).
    \end{align*}
    Summing these estimates, we obtain
    \begin{align*}
        \|R_2(\eta)\xi\|_{s,q} 
        &\le \|G_2(\eta)\xi\|_{s,q} + \|R_3(\eta)\xi\|_{s,q} \\
        &\le C \|\eta\|_{C^{s+1}} \Big( \|\partial_{\mathbf{x}'}\eta\|_{s+1,q_1} \|\xi\|_{s+1,q_2} + \|\eta\|_{s+1,q_1} \|\partial_{\mathbf{x}'}\xi\|_{s+1,q_2} \Big) \\
        &\quad + C \|\eta\|_{C^{s+1}}^2 \Big( \|\partial_{\mathbf{x}'}\eta\|_{s+1,q_1} \|\xi\|_{s+1,q_2} + \|\eta\|_{s+1,q_1} \|\partial_{\mathbf{x}'}\xi\|_{s+1,q_2} \Big) \\
        &\le C (1 + \|\eta\|_{C^{s+1}}) \|\eta\|_{C^{s+1}} \Big( \|\partial_{\mathbf{x}'}\eta\|_{s+1,q_1} \|\xi\|_{s+1,q_2} + \|\eta\|_{s+1,q_1} \|\partial_{\mathbf{x}'}\xi\|_{s+1,q_2} \Big),
    \end{align*}
    which completes the proof.
\end{proof}

\section{The modified linearized KP-I operator}\label{sec3}
This section investigates the mapping properties of the modified linearized KP-I operator \(\mathcal{L}_{\ep}\). Specifically, we analyze solutions to the equation:
\begin{equation}\label{4-1}
	\mathcal{L}_{\ep}\phi = \partial_1 h_1 + \partial_2 h_2,
\end{equation}
where the operator is defined as
\begin{align*}
	\mathcal{L}_{\ep} := \mathcal{L}_1 + \mathcal{L}_2 - \frac{3}{c} \partial_1 (\partial_1 q \partial_1 (\cdot)).
\end{align*} 
The functions \(\phi\), \(h_1\), and \(h_2\) belong to the following function spaces:
\begin{align*}
	\mathcal{H}_1 &:= \left\{ f \in \mathcal{H}_{\mathrm{ox}} : \|f\|_a < \infty \right\}, \quad
	\mathcal{H}_2 := \left\{ f \in \mathcal{H}_{\mathrm{e}} : \|f\|_b < \infty \right\}, \\
	\mathcal{H}_3 &:= \left\{ f \in \mathcal{H}_{\mathrm{ox}} : \|f\|_c < \infty \right\},
\end{align*}
equipped with the norms
\begin{align*}
	\|f\|_a^2 &:= \int_{\mathbb{R}^2} \left( \ep^2 |\nabla_{\ep}^5 f|^2 + |\nabla_{\ep}^4 f|^2 + |\nabla^3_{\ep} f|^2 + |\nabla^2 f|^2 + |\nabla f|^2 \right) dx\, dy, \\
	\|f\|_b &:= \|f\|_{L^2(\mathbb{R}^2)} + \|\partial_1 f\|_{L^2(\mathbb{R}^2)}, \\
	\|f\|_c &:= \|f\|_{L^2(\mathbb{R}^2)} + \|\partial_2 f\|_{L^2(\mathbb{R}^2)},
\end{align*}
where \(\nabla_{\ep} := (\partial_1, \ep \partial_2)\). The symmetry classes are defined by
\begin{align*}
	\mathcal{H}_{\mathrm{ox}} &:= \left\{ f : f(x, y) = -f(-x, y) = f(x, -y) \right\}, \\
	\mathcal{H}_{\mathrm{e}} &:= \left\{ f : f(x, y) = f(-x, y) = f(x, -y) \right\}.
\end{align*}

The first result establishes an a-priori estimate for the linearized operator.

\begin{prop} \label{Prop4-1}
	Let \(\phi \in \mathcal{H}_1\) solve \eqref{4-1} with \(h_1 \in \mathcal{H}_2\) and \(h_2 \in \mathcal{H}_3\). Then there exist constants \(\varepsilon_0 > 0\) and \(C > 0\) such that for all \(\varepsilon \in (0, \varepsilon_0)\),
	\begin{equation}\label{4.2}
		\|\phi\|_a \leq C \left( \|h_1\|_b + \|h_2\|_c \right),
	\end{equation}
	where \(C\) is independent of \(\varepsilon\).
\end{prop}

\begin{proof}
    Our proof follows the strategy in \cite{Liu2}, adapted to our context. We present the key steps for completeness.

    \textbf{Step 1: Functional setting and spectral decomposition.}
    Define the space $\mathcal{E}$ as the closure of $\partial_1(C_0^{\infty}(\mathbb{R}^2))$ under the norm
    \[
        \|\partial_1 f\|_{\mathcal{E}} := \left( \|\nabla f\|_{L^2(\mathbb{R}^2)}^2 + \|\partial_1^2 f\|_{L^2(\mathbb{R}^2)}^2 \right)^{1/2}.
    \]
    The antiderivative $\partial_1^{-1}$ is well-defined on $\mathcal{E}$ (cf. \cite[Remark 1.1]{de Bouard}). Consider the self-adjoint operator in $\mathcal{E}$:
    \[
        L\phi := A\partial_1^2\phi - \phi - \frac{3}{c}(\partial_1 q \phi) - (1+\ep^2)\partial_1^{-2}\partial_2^2\phi.
    \]
    By \cite[Theorem 2]{Liu1}, $L$ has a unique negative eigenvalue $\lambda_1$ with associated even eigenfunction $\phi_0$ satisfying:
    \[
        A\partial_1^2\phi_0 - \phi_0 - \frac{3}{c}(\partial_1 q_{\ep}\phi_0) - (1+\ep^2)\partial_1^{-2}\partial_2^2\phi_0 + \lambda_1\phi_0 = 0.
    \]
    Moreover, $\int_{\mathbb{R}} \phi_0(x,y) dx = 0$ for each $y$, and $\phi_0$ decays at infinity. Define:
    \[
        \phi_1 := \partial_1^{-1}\phi_0 = \int_{-\infty}^{x} \phi_0(z,y)  dz, \quad
        \partial_1^{-2}\partial_2^{2}\phi_0 := \int_{-\infty}^{x} \int_{-\infty}^{z} \partial_2^{2}\phi_0(w,y)  dw  dz.
    \]
    Then $\phi_1 \in \mathcal{H}_1$ and satisfies
    \[
        A\partial_1^4\phi_1 - \partial_1^2\phi_1 - \frac{3}{c}\partial_1(\partial_1 q \partial_1\phi_1) - (1+\ep^2)\partial_2^2\phi_1 = -\lambda_1\partial_1^{2}\phi_1.
    \]
    As $\ep \to 0$, $\lambda_1$ converges to the unique negative eigenvalue of the linearized equation for \eqref{2-11}.

    \textbf{Step 2: Orthogonal decomposition.}
    Decompose $\phi = a \phi_1 + \phi_2$ with 
    \[
        a = \frac{\int_{\mathbb{R}^2} \partial_1\phi  \partial_1\phi_1  dx  dy}{\int_{\mathbb{R}^2} (\partial_1\phi_1)^2  dx  dy},
    \]
    where $\phi_2$ belongs to the orthogonal complement of $\phi_1$ in $\mathcal{H}_1$ under the inner product:
    \begin{align*}
        ( f, g ) = \int_{\mathbb{R}^2} \Big[ A\partial_1^2 f \partial_1^2 g + \partial_1 f \partial_1 g + \frac{3}{c} \partial_1 q \partial_1 f \partial_1 g + (1+\ep^2) \partial_2 f \partial_2 g \Big]  dx  dy.
    \end{align*}
    For any $\psi$ in this complement, we have
    \[
        ( \psi, \psi ) \geq \lambda_2 \|\partial_1 \psi\|_{L^2(\mathbb{R}^2)}^2,
    \]
    where $\lambda_2 > 0$ is the smallest positive eigenvalue of $L$, uniformly bounded for small $\varepsilon$. Additionally,
    \begin{equation}\label{4-3}
        ( \psi, \psi ) \geq \lambda_{*} \left( \|\nabla \psi\|_{L^2(\mathbb{R}^2)}^2 + A \|\partial_1^2 \psi\|_{L^2(\mathbb{R}^2)}^2 \right),
    \end{equation}
    with $\lambda_{*} \in (0,1)$ independent of $\varepsilon$.

    \textbf{Step 3: Equation for $\phi_2$ and key estimates.}
    Equation \eqref{4-1} becomes
    \begin{equation}\label{4-4}
        \begin{split}
            \mathcal{L}_{\ep} \phi_2 
            &= a\lambda_1 \partial_1^2 \phi_1 - a \left[ \left(2A + \tfrac{1}{3}\right) \ep^2 \partial_1^2 \partial_2^2 \phi_1 + \sigma(1+\ep^2)\ep^4 \partial_2^4 \phi_1 + \mathcal{L}_2 \phi_1 \right] \\
            &\quad + \partial_1 h_1 + \partial_2 h_2.
        \end{split}
    \end{equation}
    Multiply by $a \phi_1$ and integrate by parts
    \begin{equation}\label{4-5}
        \begin{split}
            &a \left[ \left(2A + \tfrac{1}{3}\right) \ep^2 \langle \partial_1 \partial_2 \phi_2, \partial_1 \partial_2 \phi_1 \rangle + \sigma(1+\ep^2) \ep^4 \langle \partial_2^2 \phi_2, \partial_2^2 \phi_1 \rangle + \langle \mathcal{L}_2 \phi_2, \phi_1 \rangle \right] \\
            &= -a^2 \left[ \left(2A + \tfrac{1}{3}\right) \ep^2 \|\partial_1 \partial_2 \phi_1\|_{L^2}^2 + \sigma(1+\ep^2) \ep^4 \|\partial_2^2 \phi_1\|_{L^2}^2 + \langle \mathcal{L}_2 \phi_1, \phi_1 \rangle \right] \\
            &\quad - a^2 \lambda_1 \|\partial_1 \phi_1\|_{L^2}^2 - a \left( \langle h_1, \partial_1 \phi_1 \rangle + \langle h_2, \partial_2 \phi_1 \rangle \right).
        \end{split}
    \end{equation}
    For the symbol of $\mathcal{L}_2$, set $\widehat{-\partial_1^2} = m_1^2$, $\widehat{-\partial_2^2} = m_2^2$, and $\mathbf{k} = (k_1, k_2) = (\ep m_1, \ep^2 m_2)$. From \eqref{1.tanh}:
\begin{align*}
\widehat{\mathcal{L}_2}&=\ep^{-4}\Big(1+\ep^2+\sigma|\mathbf{k}|^2\Big)\Big(1+\frac{1}{3}|\mathbf{k}|^2\Big)|\mathbf{k}|\Big(\tanh{|\mathbf{k}|}-\frac{|\mathbf{k}|}{1+\frac{1}{3}|\mathbf{k}|^2}\Big)\\
		&\ge \ep^{-4}\Big(1+\ep^2+\sigma|\mathbf{k}|^2\Big)\Big(1+\frac{1}{3}|\mathbf{k}|^2\Big)|\mathbf{k}|\Big(\frac{1}{\frac{1}{|\mathbf{k}|}+\frac{1}{\frac{3}{|\mathbf{k}|}+\frac{1}{\frac{5}{|\mathbf{k}|}+\frac{|\mathbf{k}|}{7}}}}-\frac{|\mathbf{k}|}{1+\frac{1}{3}|\mathbf{k}|^2}\Big)\\
		&>0,
	\end{align*}
 and there exist $C, c_0 > 0$ such that
	\begin{align*}
\widehat{\mathcal{L}_2}&=\ep^{-4}\Big(1+\ep^2+\sigma|\mathbf{k}|^2\Big)\Big(1+\frac{1}{3}|\mathbf{k}|^2\Big)|\mathbf{k}|\Big(\tanh{|\mathbf{k}|}-\frac{|\mathbf{k}|}{1+\frac{1}{3}|\mathbf{k}|^2}\Big)\\
		&\le \ep^{-4}\Big(1+\ep^2+\sigma|\mathbf{k}|^2\Big)\Big(1+\frac{1}{3}|\mathbf{k}|^2\Big)|\mathbf{k}|\Big(\frac{1}{\frac{1}{|\mathbf{k}|}+\frac{1}{\frac{3}{|\mathbf{k}|}+\frac{|\mathbf{k}|}{5}}}-\frac{|\mathbf{k}|}{1+\frac{1}{3}|\mathbf{k}|^2}\Big)\\
		&\le C\ep^{2}(m_1^2+\ep^2m_2^2)^3,\\
        \widehat{\mathcal{L}_2} &\geq \ep^{-4} c_0 |\mathbf{k}|^5 \quad \text{as} \quad |\mathbf{k}| \to \infty.
	\end{align*}
    This implies
    \[
        |\langle \mathcal{L}_2 \phi_2, \phi_1 \rangle| \leq C \ep^2 \|\partial_1^2 \phi_2\|_{L^2} \|\nabla_{\ep}^4 \phi_1\|_{L^2} + C \ep^4 \|\partial_2^2 \phi_2\|_{L^2} \|\nabla_{\ep}^4 \phi_1\|_{L^2}.
    \]
    Using \eqref{4-5} and Young's inequality yields
    \begin{equation}\label{4-6}
        \begin{split}
            &a^2 (-\lambda_1 - \delta_{\ep}) \|\partial_1 \phi_1\|_{L^2}^2 \\
            &\leq -\frac{\lambda_{*}}{32\lambda_1} \left( \ep^2 \|\partial_1^2 \phi_2\|_{L^2}^2 + \sigma \ep^4 \|\partial_2^2 \phi_2\|_{L^2}^2 + \left(2A + \tfrac{1}{3}\right) \ep^2 \|\partial_1 \partial_2 \phi_2\|_{L^2}^2 \right) \\
            &\quad - \frac{2}{\lambda_1} \|h_1\|_{L^2}^2 - \frac{2}{\lambda_1} \frac{\|\partial_2 \phi_1\|_{L^2}^2}{\|\partial_1 \phi_1\|_{L^2}^2} \|h_2\|_{L^2}^2,
        \end{split}
    \end{equation}
   where  $\delta_{\varepsilon}$ depends on $\varepsilon$ and can be arbitrarily small as $\varepsilon \rightarrow 0$. In \eqref{4-6} we have also used $\delta(\varepsilon)\left\| \partial_1\phi_1\right\|_{L^2\left(\mathbb{R}^2\right)}^2$ to control the terms $\ep^2\|\partial_1\partial_2\phi_1\|_{L^2(\mathbb{R}^2)}^2$, $\ep^2\|\nabla^4\phi_1\|^2_{L^2\left(\mathbb{R}^2\right)}$,  and $\ep^4\|\partial_2^2 \phi_1\|_{L^2\left(\mathbb{R}^2\right)}^2$, due to the fact that $\phi_1$ is a concrete function. For the same reason, we see that the coefficients before $\|h_2\|_{L^2(\mathbb{R}^2)}$ is specified constant.

    \textbf{Step 4: Energy estimate for $\phi_2$.}
    Multiply \eqref{4-4} by $\phi_2$ and use \eqref{4-3}:
    \begin{equation}\label{4-7}
        \begin{split}
            &\lambda_{*} \|\nabla \phi_2\|_{L^2}^2 + \lambda_{*} A \|\partial_1^2 \phi_2\|_{L^2}^2 + \left(2A + \tfrac{1}{3}\right) \ep^2 \|\partial_1 \partial_2 \phi_2\|_{L^2}^2 + \sigma(1+\ep^2) \ep^4 \|\partial_2^2 \phi_2\|_{L^2}^2 \\
            &\leq \frac{2}{\lambda_{*}} \|h_1\|_{L^2}^2 + \frac{2}{\lambda_{*}} \|h_2\|_{L^2}^2 + 2a^2 \left( \frac{\lambda_1^2}{\lambda_{*}} + \tilde{\delta}_{\ep} \right) \|\partial_1 \phi_1\|_{L^2}^2.
        \end{split}
    \end{equation}
    Combining \eqref{4-6} and \eqref{4-7} gives
    \begin{align*}
        \frac{a^2 \lambda_1^2}{\lambda_{*}} \|\partial_1 \phi_1\|_{L^2}^2 + \lambda_{*} \|\nabla \phi_2\|_{L^2}^2 
        \leq \frac{10}{\lambda_{*}} \|h_1\|_{L^2}^2 + \left( \frac{8}{\lambda_{*}} \frac{\|\partial_2 \phi_1\|_{L^2}^2}{\|\partial_1 \phi_1\|_{L^2}^2} + \frac{2}{\lambda_{*}} \right) \|h_2\|_{L^2}^2,
    \end{align*}
    implying:
    \begin{equation*}
        a^2 \|\nabla \phi_1\|_{L^2}^2 + \|\nabla \phi_2\|_{L^2}^2 \leq C \left( \|h_1\|_{L^2}^2 + \|h_2\|_{L^2}^2 \right).
    \end{equation*}
    For the specific function $\phi_1$, we have
    \begin{equation*}
        \|\partial_1^2 \phi_1\|_{L^2}^2 \leq C \|\partial_1 \phi_1\|_{L^2}^2.
    \end{equation*}
    Thus from \eqref{4-7}:
    \begin{equation*}
        \|\partial_1^2 \phi_2\|_{L^2}^2 \leq C \left( \|h_1\|_{L^2}^2 + \|h_2\|_{L^2}^2 \right).
    \end{equation*}
    Consequently
    \begin{equation}\label{4-10}
        \|\partial_1 (\partial_1 q \partial_1 \phi)\|_{L^2}^2 \leq C \left( \|h_1\|_{L^2}^2 + \|h_2\|_{L^2}^2 \right).
    \end{equation}

    \textbf{Step 5: Higher-order estimates via Fourier analysis.}
    Rewrite \eqref{4-1} as:
    \begin{equation}\label{4-11}
        \mathcal{L}_1 \phi + \mathcal{L}_2 \phi = \partial_1 h_1 + \partial_2 h_2 + 3 \partial_1 (\partial_1 q \partial_1 \phi) =: h.
    \end{equation}
    By \eqref{4-10} and the definitions of $\|\cdot\|_b$, $\|\cdot\|_c$:
    \begin{equation}\label{4-12}
        \|h\|_{L^2}^2 \leq C \left( \|h_1\|_b^2 + \|h_2\|_c^2 \right).
    \end{equation}
    The Fourier transform of \eqref{4-11} is
    \begin{equation}\label{4-13}
        \widehat{\phi}(m_1, m_2) = \frac{\widehat{h}(m_1, m_2)}{A m_1^4 + m_1^2 + (1+\ep^2) m_2^2 + \left(2A + \tfrac{1}{3}\right) \ep^2 m_1^2 m_2^2 + \sigma (1+\ep^2) \ep^4 m_2^4 + \widehat{\mathcal{L}_2}}.
    \end{equation}
    To bound $\|\phi\|_a$, we show each term is controlled by $\|h\|_{L^2}$. For $\ep \|\nabla_{\ep}^5 \phi\|_{L^2}$:
    \begin{align*}
        \ep \|\nabla_{\ep}^5 \phi\|_{L^2}
        &\leq C \left\| \frac{\ep (m_1^2 + \ep^2 m_2^2)^{5/2} \widehat{h}}{A m_1^4 + m_1^2 + (1+\ep^2) m_2^2 + \left(2A + \tfrac{1}{3}\right) \ep^2 m_1^2 m_2^2 + \sigma (1+\ep^2) \ep^4 m_2^4 + \widehat{\mathcal{L}_2}} \right\|_{L^2} \\
        &\leq C \left\| \frac{\ep (m_1^2 + \ep^2 m_2^2)^{5/2}}{m_1^2 + \ep^2 m_2^2 + \ep (m_1^2 + \ep^2 m_2^2)^{5/2}} \widehat{h} \right\|_{L^2} \\
        &\leq C \|h\|_{L^2}.
    \end{align*}
    Other terms in $\|\cdot\|_a$ follow similarly. This completes the proof.
\end{proof}

\begin{remark}
	We would like to emphasize that the use of $tanh$ in operator $\mathcal{L}_2$ employs the method of continued fractions, which is different from previous linearization estimates. Additionally, $\sigma > \frac{1}{3}$ can be observed from the expansion of tanh, and $\sigma > \frac{1}{3}$ enables operator $\mathcal{L}_2 $ to possess properties similar to those of elliptic operators. This is also the reason why we use $ \mathcal{Q} $ in the derivation of equations in Section \ref{sec2}.
\end{remark}

After establishing the $L^2$ theory for equation \eqref{4-1}, we extend the analysis to $L^p$ spaces for $p \in [4, \infty)$. This extension is crucial for studying the nonlinear problem \eqref{2-8}. We define the following Sobolev spaces:
\begin{align*}
    \mathcal{F}_1 &:= \left\{ f \in \mathcal{H}_{\mathrm{ox}} : \|f\|_* < \infty \right\},  \quad
    \mathcal{F}_2 := \left\{ f \in \mathcal{H}_{\mathrm{e}} : \|f\|_{**} < \infty \right\},  \\
    \mathcal{F}_3 &:= \left\{ f \in \mathcal{H}_{\mathrm{ox}} : \|f\|_{***} < \infty \right\}, 
\end{align*}
equipped with the norms
\begin{align*}
    \|f\|_* &:= \|f\|_a + \|f\|_{L^4(\mathbb{R}^2)} + \|f\|_{L^\infty(\mathbb{R}^2)} + \|\nabla f\|_{L^p(\mathbb{R}^2)} \\
    &\quad + \ep^{1/2} \|\nabla_{\ep} f\|_{L^\infty(\mathbb{R}^2)} + \ep^{1/2} \|\nabla_{\ep}^2 f\|_{L^\infty(\mathbb{R}^2)} + \ep^{3/2} \|\nabla_{\ep}^3 f\|_{L^\infty(\mathbb{R}^2)}, \quad p > 2, \\
    \|f\|_{**} &:= \|f\|_b + \|f\|_{L^{4/3}(\mathbb{R}^2)}, \\
    \|f\|_{***} &:= \|f\|_c + \|f\|_{L^{4/3}(\mathbb{R}^2)}.
\end{align*}

The second main result is the following existence and estimate:

\begin{prop}\label{prop4-2}
    Let $h_1 \in \mathcal{F}_2$ and $h_2 \in \mathcal{F}_3$. There exist constants $\varepsilon_0 > 0$ and $C > 0$ such that for all $\varepsilon \in (0, \varepsilon_0)$, the linear equation \eqref{4-1} admits a solution $\phi \in \mathcal{F}_1$ satisfying
    \begin{equation}\label{4-28}
        \|\phi\|_{*} \leq C \left( \|h_1\|_{**} + \|h_2\|_{***} \right).
    \end{equation}
    Moreover, $\phi(x) \to 0$ as $|x| \to \infty$.
\end{prop}

\begin{proof}
    The existence proof follows arguments similar to Proposition 4.2 in \cite{Liu2}. We thus obtain the estimate
    \[
        \|\phi\|_{\mathcal{E}_1} \leq C\big( \|h_1\|_b + \|h_2\|_c \big),
    \]
    where $\mathcal{E}_1$ is the Hilbert space defined by
    \[
        \mathcal{E}_1 := \big\{ f \in \mathcal{H}_{\mathrm{ox}} \mid \|f\|_{\mathcal{E}_1} < +\infty \big\},
    \]
    with norm $\|f\|_{\mathcal{E}_1}^2 = \langle f, f \rangle_{\mathcal{E}_1}$ and inner product
    \[
        \langle f, g \rangle_{\mathcal{E}_1} = \int_{\mathbb{R}^2} \Big( \partial_1^2 f \, \partial_1^2 g + \nabla f \cdot \nabla g + \ep^2 \partial_1 \partial_2 f \, \partial_1 \partial_2 g + \ep^4 \partial_2^2 f \, \partial_2^2 g + (\mathcal{L}_{2} f) g \Big)  dx  dy.
    \]
    Clearly, $\mathcal{E}_1$ is a Hilbert space. Moreover, $\|f\|_{\mathcal{E}_1} = 0$ implies $f = 0$: if $\|f\|_{\mathcal{E}_1} = 0$, then $f$ must be constant, and since $f \in \mathcal{H}_{\mathrm{ox}}$, we conclude $f = 0$.

    The key difference in our approach lies in the use of the $L^p$-norm, which requires modified norm estimates. We now derive the a-priori estimate \eqref{4-28} by analyzing $\|\phi\|_{*}$ term by term. 

    First, we estimate $\|\phi\|_{L^4(\mathbb{R}^2)}$. Using the Green's function representation, we have
    \[
        \widehat{\phi}(m_1, m_2) = \frac{3i m_1 \widehat{\partial_1 q \, \partial_1 \phi} + i m_1 \widehat{h_1} + i m_2 \widehat{h_2}}{A m_1^4 + m_1^2 + (1 + \ep^2) m_2^2 + (2A + \frac{1}{3}) \ep^2 m_1^2 m_2^2 + \sigma(1 + \ep^2) \ep^4 m_2^4 + \widehat{\mathcal{L}_2}}.
    \]
    Applying the Hardy-Littlewood-Paley theorem on $\mathbb{R}^2$ and the Hausdorff-Young inequality, we obtain
    \begin{align*}
        \|\phi\|_{L^4(\mathbb{R}^2)} 
        &= \|\widecheck{\widehat{\phi}}\|_{L^4(\mathbb{R}^2)} 
        = \|\widehat{\widehat{\phi}}(-\xi)\|_{L^4(\mathbb{R}^2)} \\
        &\leq C \|\widehat{\phi}\|_{L^{4/3}(\mathbb{R}^2)} \\
        &\leq C \Big( \big\||m|^{-1} \widehat{\partial_1 q \, \partial_1 \phi}\big\|_{L^{4/3}(\mathbb{R}^2)} 
        + \big\||m|^{-1} \widehat{h}_1 \big\|_{L^{4/3}(\mathbb{R}^2)} 
        + \big\||m|^{-1} \widehat{h_2} \big\|_{L^{4/3}(\mathbb{R}^2)} \Big) \\
        &\leq C \Big( \big\| \partial_1 q \, \partial_1 \phi \big\|_{L^{4/3}(\mathbb{R}^2)} 
        + \big\| h_1 \big\|_{L^{4/3}(\mathbb{R}^2)} 
        + \big\| h_2 \big\|_{L^{4/3}(\mathbb{R}^2)} \Big) \\
        &\leq C \Big( \| h_1 \|_{**} + \| h_2 \|_{***} \Big).
    \end{align*}

    Next, applying the Gagliardo-Nirenberg interpolation inequality
    \[
        \|D^j u\|_{L^{p_1}(\mathbb{R}^n)} \leq C \|D^l u\|_{L^r(\mathbb{R}^n)}^\alpha \|u\|_{L^{p_2}(\mathbb{R}^n)}^{1-\alpha}, \quad \frac{1}{p_1} = \frac{j}{n} + \left( \frac{1}{r} - \frac{l}{n} \right) \alpha + \frac{1 - \alpha}{p_2},
    \]
    we deduce
    \begin{align*}
        \|\phi\|_{L^{p_1}(\mathbb{R}^2)} + \|\nabla \phi\|_{L^{p_2}(\mathbb{R}^2)} 
        &\leq C \Big( \|\nabla^2 \phi\|^{\alpha_1}_{L^2(\mathbb{R}^2)} \|\phi\|^{1-\alpha_1}_{L^4(\mathbb{R}^2)} 
        + \|\nabla^2 \phi\|^{\alpha_2}_{L^2(\mathbb{R}^2)} \|\phi\|^{1-\alpha_2}_{L^4(\mathbb{R}^2)} \Big) \\
        &\leq C \Big( \| h_1 \|_{**} + \| h_2 \|_{***} \Big),
    \end{align*}
    where $\alpha_1 = \frac{1}{3}(1 - \frac{4}{p_1})$, $\alpha_2 = 1 - \frac{4}{3p_2}$, for all $p_1 \in [4, +\infty)$ and $p_2 \in [\frac{8}{3}, +\infty)$. 

    By the Sobolev embedding $W^{1,p}(\mathbb{R}^2) \hookrightarrow C_{0}(\mathbb{R}^2)$ for $p > 2$, it follows that 
    \[
        \|\phi\|_{L^{\infty}(\mathbb{R}^2)} \leq C \big( \| h_1 \|_{**} + \| h_2 \|_{***} \big)
    \]
    and $\phi(x) \to 0$ as $|x| \to \infty$.

    Now define $\widetilde{\phi}(x, y) = \phi(x, \ep y)$. Then
    \begin{align*}
        \|\widetilde{\phi}\|_{L^{\infty}(\mathbb{R}^2)} 
        &= \|\phi\|_{L^{\infty}(\mathbb{R}^2)} 
        \leq C \big( \| h_1 \|_{**} + \| h_2 \|_{***} \big), \\
        \|\nabla^m \widetilde{\phi}\|_{L^2(\mathbb{R}^2)} 
        &= \ep^{-1/2} \|\nabla_{\ep}^m \phi\|_{L^2(\mathbb{R}^2)} 
        \leq C \ep^{-1/2} \big( \| h_1 \|_{**} + \| h_2 \|_{***} \big), \quad m = 1, 2, 3, 4, \\
        \|\nabla^5 \widetilde{\phi}\|_{L^2(\mathbb{R}^2)} 
        &= \ep^{-1/2} \|\nabla_{\ep}^5 \phi\|_{L^2(\mathbb{R}^2)} 
        \leq C \ep^{-3/2} \big( \| h_1 \|_{**} + \| h_2 \|_{***} \big).
    \end{align*}
    Using the Sobolev embedding $H^2(\mathbb{R}^2) \hookrightarrow C_0(\mathbb{R}^2)$, we bound the scaled derivatives:
    \begin{align*}
        \|\nabla_{\ep}^m \phi\|_{C_0(\mathbb{R}^2)} 
        &= \|\nabla^m \widetilde{\phi}\|_{C_0(\mathbb{R}^2)} \\
        &\leq C \Big( \|\nabla^m \widetilde{\phi}\|_{L^2(\mathbb{R}^2)} + \|\nabla^{m+2} \widetilde{\phi}\|_{L^2(\mathbb{R}^2)} \Big) \\
        &\leq \begin{cases} 
            C \ep^{-1/2} \big( \| h_1 \|_{**} + \| h_2 \|_{***} \big), & m = 1, 2, \\
            C \ep^{-3/2} \big( \| h_1 \|_{**} + \| h_2 \|_{***} \big), & m = 3.
        \end{cases}
    \end{align*}
    This completes the proof.
\end{proof}

\section{Solving the nonlinear system}\label{sec5}
This section addresses the solution of the nonlinear systems \eqref{2-5} and \eqref{2-6}. We seek solutions to \eqref{2-8} of the form $f = q + \phi$, where $q = q_{\ep}$ is defined in \eqref{2-10}. Substituting this ansatz yields
\begin{equation}
    \label{5-1}
    \mathcal{L}_{\ep} \phi = \sum_{i=1}^4 P_{i} + \frac{3}{2c} \partial_1 \big( (\partial_1 \phi)^2 \big) - \mathcal{L}_2 q - \left(2A + \tfrac{1}{3}\right) \ep^2 \partial_1^2 \partial_2^2 q - \sigma (1 + \ep^2) \ep^4 \partial_2^4 q.
\end{equation}
From \eqref{2-6}, we express $h$ as $h = c (\partial_1 q + \psi)$, where $\psi$ satisfies 
\begin{equation}
    \label{5-2}
    \psi - \sigma \ep^2 (\partial_1^2 + \ep^2 \partial_2^2) \psi = \partial_1 \phi + \sigma \ep^2 (\partial_1^2 + \ep^2 \partial_2^2) (\partial_1 q) - \frac{\ep^2}{2c} (\partial_1 q + \partial_1 \phi)^2 + \frac{\ep^4}{c} \Pi.
\end{equation}

We solve \eqref{5-1} in the function space
\begin{equation}\label{5-3}
    \mathcal{F}_\phi = \left\{ \phi \in \mathcal{F}_1 \mid \|\phi\|_* \leq C \varepsilon \right\},
\end{equation}
with $C$ sufficiently large. For $\phi \in \mathcal{F}_\phi$, the following estimates hold:
\begin{align}
    \|\xi\|_{L^{\infty}(\mathbb{R}^2; \mathbf{x'})} 
    &\leq C \ep + \ep \|\phi\|_{L^{\infty}(\mathbb{R}^2)} \leq C \ep, \label{5-4} \\
    \|\xi\|_{L^p(\mathbb{R}^2; \mathbf{x'})} 
    &\leq \ep^{1 - \frac{3}{p}} \left( \|q\|_{L^p(\mathbb{R}^2)} + \|\phi\|_{L^p(\mathbb{R}^2)} \right) \leq C \ep^{1 - \frac{3}{p}}, \quad p \in [4, +\infty]. \label{5-5}
\end{align}
Furthermore, for $m = 1, 2, \dots, 5$,
\begin{align}
    \|\partial_{\mathbf{x'}}^m \xi\|_{L^2(\mathbb{R}^2; \mathbf{x'})} 
    &= \ep^{m+1 - \frac{3}{2}} \|\nabla_{\ep}^m f\|_{L^2(\mathbb{R}^2)} \nonumber \\
    &\leq \ep^{m - \frac{1}{2}} \left( \|\nabla_{\ep}^m q\|_{L^2(\mathbb{R}^2)} + \|\nabla_{\ep}^m \phi\|_{L^2(\mathbb{R}^2)} \right) \leq C \ep^{m - \frac{1}{2}}. \label{5-6}
\end{align}
For $m = 1, 2, 3$,
\begin{equation}
    \|\partial_{\mathbf{x'}}^m \xi\|_{L^{\infty}(\mathbb{R}^2; \mathbf{x'})} 
    \leq \ep^{m+1} \left( \|\nabla_{\ep}^m q\|_{L^{\infty}(\mathbb{R}^2)} + \|\nabla_{\ep}^m \phi\|_{L^{\infty}(\mathbb{R}^2)} \right) 
    \leq C \ep^{m+1}. \label{5-7}
\end{equation}
By the Gagliardo-Nirenberg interpolation inequality, for $m = 1, 2, 3, 4$ and $p \in [4, +\infty)$,
\begin{equation}
    \|\partial_{\mathbf{x'}}^m \xi\|_{L^{p}(\mathbb{R}^2; \mathbf{x'})} 
    \leq C \|\partial_{\mathbf{x'}}^{m+1} \xi\|_{L^{2}(\mathbb{R}^2; \mathbf{x'})}^{\theta_1} \|\partial_{\mathbf{x'}}^{m-1} \xi\|_{L^{\infty}(\mathbb{R}^2; \mathbf{x'})}^{1 - \theta_1} 
    \leq C \ep^{m + \frac{1}{2} - \frac{1}{p}}, \label{5-8}
\end{equation}
where $\theta_1 = 1 - \frac{2}{p} \geq \frac{1}{2}$.

Next, we consider the linear equation 
\begin{equation}\label{5.3}
    \psi - \sigma \ep^2 (\partial_1^2 + \ep^2 \partial_2^2) \psi = \mathfrak{h}.
\end{equation}
Define the Hilbert space $H := \left\{ f \in \mathcal{H}_{e} \mid \|f\|_{h} < +\infty \right\}$ with norm
    \begin{align*}
        \|\psi\|_{h} 
        &:= \|\psi\|_{L^2(\mathbb{R}^2)} + \|\nabla_{\ep} \psi\|_{L^2(\mathbb{R}^2)} + \|\nabla_{\ep}^2 \psi\|_{L^2(\mathbb{R}^2)} + \|\nabla_{\ep}^3 \psi\|_{L^2(\mathbb{R}^2)} + \ep \|\nabla_{\ep}^4 \psi\|_{L^2(\mathbb{R}^2)} \\
        &\quad + \ep^2 \|\nabla_{\ep}^5 \psi\|_{L^2(\mathbb{R}^2)} + \ep^3 \|\nabla_{\ep}^6 \psi\|_{L^2(\mathbb{R}^2)} + \|\psi\|_{L^{\infty}(\mathbb{R}^2)}+ \ep^{1/2} \|\nabla_{\ep} \psi\|_{L^{\infty}(\mathbb{R}^2)} \\
        &\quad  + \ep^{3/2} \|\nabla_{\ep}^2 \psi\|_{L^{\infty}(\mathbb{R}^2)} + \ep^{5/2} \|\nabla_{\ep}^3 \psi\|_{L^{\infty}(\mathbb{R}^2)} + \ep^{7/2} \|\nabla_{\ep}^4 \psi\|_{L^{\infty}(\mathbb{R}^2)}.
    \end{align*}
We seek solutions in the space 
\begin{equation}\label{5.10}
    \mathcal{F}_{\psi} := \left\{ \psi \in H \mid \|\psi\|_{h} \leq C \ep \right\}.
\end{equation}

For $\psi \in \mathcal{F}_{\psi}$, we have the estimates:
\begin{align}
    \|\partial_{\mathbf{x'}}^m \eta\|_{L^2(\mathbb{R}^2; \mathbf{x'})} 
    &= \ep^{m+2 - \frac{3}{2}} \|\nabla_{\ep}^{m} h\|_{L^2(\mathbb{R}^2)} \nonumber \\
    &\leq C \ep^{m + \frac{1}{2}} \left( \|\nabla_{\ep}^{m} \partial_1 q\|_{L^2(\mathbb{R}^2)} + \|\nabla_{\ep}^{m} \psi\|_{L^2(\mathbb{R}^2)} \right) \nonumber \\
    &\leq \begin{cases} 
        C \ep^{m + \frac{1}{2}}, & m = 0, 1, 2, \\
        C \ep^{7/2}, & m = 3, 4, 5, 6.
    \end{cases} \label{512}
\end{align}
Similarly, 
\begin{align}
    \|\partial_{\mathbf{x'}}^m \eta\|_{L^{\infty}(\mathbb{R}^2; \mathbf{x'})} 
    &\leq C \ep^{m+2} \left( \|\nabla_{\ep}^m (\partial_1 q)\|_{L^{\infty}(\mathbb{R}^2)} + \|\nabla_{\ep}^m \psi\|_{L^{\infty}(\mathbb{R}^2)} \right) \nonumber \\
    &\leq \begin{cases} 
        C \ep^{2}, & m = 0, \\
        C \ep^{3}, & m = 1, 2, 3, 4.
    \end{cases} \label{5-14}
\end{align}
Moreover, 
\begin{align}
    \|\partial_{\mathbf{x'}}^m \eta\|_{L^p(\mathbb{R}^2; \mathbf{x'})} 
    &\leq C \ep^{2 + m - \frac{3}{p}} \left( \|\nabla_{\ep}^m \partial_1 q\|_{L^p(\mathbb{R}^2)} + \|\nabla_{\ep}^m \psi\|_{L^p(\mathbb{R}^2)} \right) \nonumber\\
    &\leq \begin{cases} 
        C \ep^{2 - \frac{3}{p}}, & m = 0, \\
        C \ep^{3 - \frac{3}{p}}, & m = 1, 2, 3, 4.
    \end{cases} \label{5-15}
\end{align}

We now analyze \eqref{5.3}. For given $\mathfrak{h} \in \mathcal{F}_5$, we solve for $\psi \in \mathcal{F}_{\psi}$, where 
\begin{align*}
     \|\mathfrak{h}\|_{\mathcal{F}_5} &:= \|\mathfrak{h}\|_{L^{\infty}(\mathbb{R}^2)} + \|\mathfrak{h}\|_{L^2(\mathbb{R}^2)} + \|\nabla_{\ep} \mathfrak{h}\|_{L^2(\mathbb{R}^2)} + \|\nabla_{\ep}^2 \mathfrak{h}\|_{L^2(\mathbb{R}^2)} \\
     &\ \quad+ \|\nabla_{\ep}^3 \mathfrak{h}\|_{L^2(\mathbb{R}^2)} + \ep \|\nabla_{\ep}^4 \mathfrak{h}\|_{L^2(\mathbb{R}^2)}.
\end{align*}

\begin{lemma}\label{Lemma5-1}
    Let $\psi \in H$ be a solution of \eqref{5.3} with $\mathfrak{h} \in \mathcal{F}_5$. Then there exist positive constants $\varepsilon_0$ and $C$, independent of $\ep$, such that for all $\ep \in (0, \varepsilon_0)$,
    \begin{align}\label{5-4}
        \|\psi\|_h \leq C \|\mathfrak{h}\|_{\mathcal{F}_5}. 
    \end{align}
\end{lemma}

\begin{proof}
    Applying the Fourier transform to \eqref{5.3} gives
    \[
        \widehat{\psi}(m_1, m_2) = \frac{\widehat{\mathfrak{h}}(m_1, m_2)}{1 + \sigma \ep^2 (m_1^2 + \ep^2 m_2^2)}.
    \]
    From this representation, we derive the high-order derivative estimate:
    \[
        \ep^3 \|\nabla_{\ep}^6 \psi\|_{L^2(\mathbb{R}^2)} \leq C \ep \|\nabla_{\ep}^4 \mathfrak{h}\|_{L^2(\mathbb{R}^2)} \leq C \|\mathfrak{h}\|_{\mathcal{F}_5}.
    \]
    Similar arguments bound the $L^2$-norms of lower-order derivatives of $\psi$ by $C \|\mathfrak{h}\|_{\mathcal{F}_5}$.

    To establish the $L^\infty$-bound, define $\overline{\psi}(x_1, x_2) = \psi(\ep x_1, \ep^2 x_2)$. This satisfies
    \[
        \overline{\psi} - \sigma \Delta_{\mathbf{x}'} \overline{\psi} = \overline{\mathfrak{h}},
    \]
    where $\overline{\mathfrak{h}}(x_1, x_2) = \mathfrak{h}(\ep x_1, \ep^2 x_2)$. The Green's function for $-\Delta_{\mathbf{x}'} + 1$ on $\mathbb{R}^n$ is
    \[
        G(|\mathbf{x}' - \mathbf{y}'|) = |\mathbf{x}' - \mathbf{y}'|^{-(n-2)/2} K_{(n-2)/2}(|\mathbf{x}' - \mathbf{y}'|) > 0,
    \]
    with $K_v(z)$ denoting the modified Bessel function of order $v$. Consequently, 
    \[
        \|\psi\|_{L^{\infty}(\mathbb{R}^2)} = \|\overline{\psi}\|_{L^{\infty}(\mathbb{R}^2)} \leq C \|\overline{\mathfrak{h}}\|_{L^{\infty}(\mathbb{R}^2)} = C \|\mathfrak{h}\|_{L^{\infty}(\mathbb{R}^2)} \leq C \|\mathfrak{h}\|_{\mathcal{F}_5}.
    \]

    For the weighted Hölder norms, define $\widetilde{\psi}(x, y) = \psi(x, \ep y)$. Then
    \[
        \|\nabla^m \widetilde{\psi}\|_{L^{\infty}(\mathbb{R}^2)} = \|\nabla_{\ep}^m \psi\|_{L^{\infty}(\mathbb{R}^2)}, \quad m \geq 0.
    \]
    The $L^2$-norms satisfy
    \[
        \|\nabla^m \widetilde{\psi}\|_{L^2(\mathbb{R}^2)} = \ep^{-1/2} \|\nabla_{\ep}^m \psi\|_{L^2(\mathbb{R}^2)} \leq 
        \begin{cases} 
            C \ep^{-1/2} \|\mathfrak{h}\|_{\mathcal{F}_5}, & m = 1, 2, 3, \\
            C \ep^{3 - m - 1/2} \|\mathfrak{h}\|_{\mathcal{F}_5}, & m = 4, 5, 6.
        \end{cases}
    \]
    By the Sobolev embedding $H^2(\mathbb{R}^2) \hookrightarrow C_0(\mathbb{R}^2)$, we obtain
    \[
        \|\nabla^m \widetilde{\psi}\|_{C_0(\mathbb{R}^2)} \leq C \left( \|\nabla^m \widetilde{\psi}\|_{L^2(\mathbb{R}^2)} + \|\nabla^{m+2} \widetilde{\psi}\|_{L^2(\mathbb{R}^2)} \right).
    \]
    Combining these estimates yields
    \[
        \|\nabla^m \widetilde{\psi}\|_{C_0(\mathbb{R}^2)} \leq 
        \begin{cases} 
            C \ep^{-1/2} \|\mathfrak{h}\|_{\mathcal{F}_5}, & m = 1, \\
            C \ep^{1/2 - m} \|\mathfrak{h}\|_{\mathcal{F}_5}, & m = 2, 3, 4.
        \end{cases}
    \]
    This completes the proof, as all components of $\|\psi\|_h$ are bounded by $C \|\mathfrak{h}\|_{\mathcal{F}_5}$.
\end{proof}

\begin{lemma}\label{Lemma5-2}
    Let $\mathfrak{h} \in \mathcal{F}_5$. Then there exists a solution $\psi \in H$ to \eqref{5.3}. 
    Moreover, for any $\mathfrak{h}_1, \mathfrak{h}_2 \in \mathcal{F}_5$, the corresponding solutions $\psi_{\mathfrak{h}_1}$ and $\psi_{\mathfrak{h}_2}$ satisfy
    \begin{equation}\label{5-5}
        \left\|\psi_{\mathfrak{h}_1} - \psi_{\mathfrak{h}_2}\right\|_{h} \leq C \left\|\mathfrak{h}_1 - \mathfrak{h}_2\right\|_{\mathcal{F}_5}.
    \end{equation}
\end{lemma}

\begin{proof}
    Define the operator $L_{\varepsilon,h}$ by
    \begin{equation*}
        L_{\varepsilon,h}\psi := \psi - \sigma\varepsilon^2(\partial_1^2 + \varepsilon^2\partial_2^2)\psi.
    \end{equation*}
    Consider the space
    \begin{equation*}
        \mathcal{E}_2 := \left\{f \mid \|f\|_{\mathcal{E}_2} < +\infty\right\},
    \end{equation*}
    equipped with the inner product
    \begin{equation*}
        \langle f, g \rangle_{\mathcal{E}_2} = \int_{\mathbb{R}^2} \left( fg + \sigma \varepsilon^2 \left( \partial_1f \, \partial_1g + \varepsilon^2 \partial_2f \, \partial_2g \right) \right) dx\,dy
    \end{equation*}
    and norm $\|f\|_{\mathcal{E}_2}^2 = \langle f, f \rangle_{\mathcal{E}_2}$. Then $\mathcal{E}_2$ is a Hilbert space, and $\|f\|_{\mathcal{E}_2} = 0$ implies $f = 0$. It is straightforward to verify that $L_{\varepsilon,h}$ is self-adjoint on $\mathcal{E}_2$.

    Define the bilinear form
    \begin{equation*}
        \mathcal{L}_{\varepsilon,h}(u, v) := \int_{\mathbb{R}^2} v \, L_{\varepsilon,h}u  dx\,dy.
    \end{equation*}
    This form satisfies
    \begin{equation*}
        \left| \mathcal{L}_{\varepsilon,h}(u, v) \right| \leq C \|u\|_{\mathcal{E}_2} \|v\|_{\mathcal{E}_2},
    \end{equation*}
    so for fixed $u \in \mathcal{E}_2$, the map $v \mapsto \mathcal{L}_{\varepsilon,h}(u, v)$ is a bounded linear functional on $\mathcal{E}_2$. By the Riesz Representation Theorem, there exists a unique $w \in \mathcal{E}_2$ such that
    \begin{equation*}
        \mathcal{L}_{\varepsilon,h}(u, v) = \langle w, v \rangle_{\mathcal{E}_2}, \quad \text{for all } v \in \mathcal{E}_2.
    \end{equation*}
    Define the operator $A_h u = w$, so that
    \begin{equation*}
        \mathcal{L}_{\varepsilon,h}(u, v) = \langle A_h u, v \rangle_{\mathcal{E}_2}, \quad \text{for all } v \in \mathcal{E}_2.
    \end{equation*}
    Then $A_h: \mathcal{E}_2 \to \mathcal{E}_2$ is a bounded linear operator.

    We show $A_h$ is injective. Suppose $A_h u = 0$. Then $\langle A_h u, v \rangle_{\mathcal{E}_2} = \mathcal{L}_{\varepsilon,h}(u, v) = 0$ for all $v \in \mathcal{E}_2$, implying $L_{\varepsilon,h} u = 0$. Using the estimate \eqref{5-4}, this yields $\|u\|_{\mathcal{E}_2} = 0$, hence $u = 0$.

    Next, we prove $A_h$ is surjective. Suppose $w \in R(A_h)^\perp \cap \mathcal{E}_2$. Then $\langle A_h u, w \rangle_{\mathcal{E}_2} = 0$ for all $u \in \mathcal{E}_2$, which implies $\mathcal{L}_{\varepsilon,h}(u, w) = 0$ for all $u \in \mathcal{E}_2$. Taking $u = w$ gives $\mathcal{L}_{\varepsilon,h}(w, w) = 0$. By the coercivity result in \eqref{5-4}, this implies $\|w\|_{\mathcal{E}_2} = 0$, so $w = 0$. Thus $R(A_h) = \mathcal{E}_2$.

    Since $A_h$ is bijective, for any $\mathfrak{h} \in \mathcal{F}_5$, there exists a unique $\psi \in \mathcal{E}_2$ such that
    \begin{equation*}
        \langle A_h \psi, v \rangle_{\mathcal{E}_2} = \int_{\mathbb{R}^2} v \mathfrak{h}  dx\,dy, \quad \text{for all } v \in \mathcal{E}_2.
    \end{equation*}
    By the definition of $A_h$ and $\mathcal{L}_{\varepsilon,h}$, this is equivalent to $\psi$ satisfying \eqref{5.3}. Furthermore, the boundedness of $A_h^{-1}$ implies the estimate
    \begin{equation*}
        \|\psi\|_{\mathcal{E}_2} \leq C \|\mathfrak{h}\|_{\mathcal{F}_5}.
    \end{equation*}
    Finally, the linearity of \eqref{5.3} and \eqref{5-4} directly yield the Lipschitz estimate \eqref{5-5}.
\end{proof}

By \eqref{2-5}, we define 
\begin{equation}
    \begin{split}
        \mathfrak{h}_{\varepsilon,\phi} &:=\partial_1\phi + \sigma\varepsilon^2(\partial_1^2 + \varepsilon^2\partial_2^2)(\partial_1q) - \frac{\varepsilon^2}{2c}(\partial_1q + \partial_1\phi)^2 + \frac{\varepsilon^4}{c}\Pi_{\phi} \\
        &\ = \partial_1\phi + \sigma\varepsilon^2(\partial_1^2 + \varepsilon^2\partial_2^2)(\partial_1q) - \frac{\varepsilon^2}{2c}|\nabla_{\varepsilon} f|^2 \\
        &\quad + \frac{\varepsilon^4}{2c(1+\varepsilon^6|\nabla_{\varepsilon} h|^2)}\left(-c\partial_{1}h + \varepsilon^2\nabla_{\varepsilon} h \cdot \nabla_{\varepsilon} f\right)^2 \\
        &\quad + \frac{\sigma \varepsilon^2}{c}\left(\nabla_{\varepsilon}\cdot \left[\frac{\nabla_{\varepsilon} h}{\sqrt{1+\varepsilon^6|\nabla_{\varepsilon} h|^2}} - \nabla_{\varepsilon} h\right]\right).
    \end{split}
\end{equation}

\begin{lemma}\label{Lemma 5.3}
    Let $\phi_1, \phi_2 \in \mathcal{F}_{\phi}$ and $\psi_1, \psi_2 \in \mathcal{F}_{\psi}$. Then 
    \begin{align*}
        \|\mathfrak{h}_{\varepsilon,\phi_1} - \mathfrak{h}_{\varepsilon,\phi_2}\|_{\mathcal{F}_{5}} \leq C \varepsilon^{-\frac{1}{2}} \|\phi_1 - \phi_2\|_{*} + O(\varepsilon) \|\psi_1 - \psi_2\|_{h}.
    \end{align*}
\end{lemma}

\begin{proof}
By the definition of $\mathfrak{h}_{\varepsilon,\phi}$, we have
\begin{align*}
    \mathfrak{h}_{\varepsilon,\phi_1} - \mathfrak{h}_{\varepsilon,\phi_2} 
    &= (\partial_1\phi_1 - \partial_1\phi_2) - \frac{\varepsilon^2}{2c} \nabla_{\varepsilon}(2q + \phi_1 + \phi_2) \cdot \nabla_{\varepsilon}(\phi_1 - \phi_2) \\
    &\quad + g_1(h_1)g_2(h_1, f_1) - g_1(h_2)g_2(h_2, f_2) \\
    &\quad + \nabla_{\varepsilon} \cdot \left[ g_3(h_1) \nabla_{\varepsilon} h_1 - g_3(h_2) \nabla_{\varepsilon} h_2 \right],
\end{align*}
where the coefficient functions are defined as:
\begin{align*}
    g_1(h) &= \frac{\varepsilon^4}{2c(1 + \varepsilon^6 |\nabla_{\varepsilon} h|^2)}, \\
    g_2(h, f) &= \left( -c\partial_{1}h + \varepsilon^2 \nabla_{\varepsilon} h \cdot \nabla_{\varepsilon} f \right)^2, \\
    g_3(h) &= \frac{-\sigma \varepsilon^8 |\nabla_{\varepsilon} h|^2}{c \sqrt{1 + \varepsilon^6 |\nabla_{\varepsilon} h|^2} \left( 1 + \sqrt{1 + \varepsilon^6 |\nabla_{\varepsilon} h|^2} \right)}.
\end{align*}

To estimate the difference in $\mathcal{F}_5$-norm, we bound each component. First, 
by the definition of $\|\cdot\|_{\mathcal{F}_5}$, 
\begin{align*}
    \|\partial_1\phi_1 - \partial_1\phi_2\|_{\mathcal{F}_5} 
    &\leq \|\nabla_{\varepsilon}(\phi_1 - \phi_2)\|_{L^{\infty}(\mathbb{R}^2)} 
       + \|\nabla_{\varepsilon}(\phi_1 - \phi_2)\|_{L^{2}(\mathbb{R}^2)} \\
    &\quad + \|\nabla_{\varepsilon}^2(\phi_1 - \phi_2)\|_{L^{2}(\mathbb{R}^2)} 
       + \|\nabla_{\varepsilon}^3(\phi_1 - \phi_2)\|_{L^{2}(\mathbb{R}^2)} \\
    &\quad + \|\nabla_{\varepsilon}^4(\phi_1 - \phi_2)\|_{L^{2}(\mathbb{R}^2)} 
       + \varepsilon \|\nabla_{\varepsilon}^5(\phi_1 - \phi_2)\|_{L^{2}(\mathbb{R}^2)}.
\end{align*}
By the definition of $\|\cdot\|_*$, this implies
\begin{equation}\label{eq:partial_phi_bound}
    \|\partial_1\phi_1 - \partial_1\phi_2\|_{\mathcal{F}_5} \leq C\varepsilon^{-\frac{1}{2}} \|\phi_1 - \phi_2\|_{*}.
\end{equation}

Next, consider the second term: 
$\frac{\varepsilon^2}{2c} \nabla_{\varepsilon}(2q + \phi_1 + \phi_2) \cdot \nabla_{\varepsilon}(\phi_1 - \phi_2)$. 
Applying the definition of $\|\cdot\|_{\mathcal{F}_5}$ and Sobolev embedding, we obtain
\begin{equation}\label{eq:second_term_decomp}
    \begin{split}
        &\left\| \frac{\varepsilon^2}{2c} \nabla_{\varepsilon}(2q + \phi_1 + \phi_2) \cdot \nabla_{\varepsilon}(\phi_1 - \phi_2) \right\|_{\mathcal{F}_{5}} \\
        &\leq C_1 \varepsilon^2 \|\nabla_{\varepsilon}(2q + \phi_1 + \phi_2)\|_{L^{\infty}} \|\nabla_{\varepsilon}(\phi_1 - \phi_2)\|_{L^{\infty}} \\
        &\quad + C_2 \varepsilon^2 \left\| \nabla_{\varepsilon}(2q + \phi_1 + \phi_2) \cdot \nabla_{\varepsilon}(\phi_1 - \phi_2) \right\|_{L^2} \\
        &\quad + C_3 \varepsilon^2 \left\| \nabla_{\varepsilon}^4 \left[ \nabla_{\varepsilon}(2q + \phi_1 + \phi_2) \cdot \nabla_{\varepsilon}(\phi_1 - \phi_2) \right] \right\|_{L^2},
    \end{split}
\end{equation}
where the Sobolev embedding theorem provides
\begin{equation}\label{eq:sobolev_embed}
    \|\nabla_{\varepsilon}^2 h\|_{L^2} + \|\nabla_{\varepsilon}^3 h\|_{L^2} \leq C \left( \|h\|_{L^2} + \|\nabla_{\varepsilon}^4 h\|_{L^2} \right).
\end{equation}
The first component in \eqref{eq:second_term_decomp} is bounded by
\begin{equation}\label{eq:inf_bound}
    \varepsilon^2 \|\nabla_{\varepsilon}(2q + \phi_1 + \phi_2)\|_{L^{\infty}} \|\nabla_{\varepsilon}(\phi_1 - \phi_2)\|_{L^{\infty}} \leq C \varepsilon^{\frac{3}{2}} \|\phi_1 - \phi_2\|_{*}.
\end{equation}
For the second component, we have
\begin{equation}\label{eq:L2_bound}
    \varepsilon^2 \left\| \nabla_{\varepsilon}(2q + \phi_1 + \phi_2) \cdot \nabla_{\varepsilon}(\phi_1 - \phi_2) \right\|_{L^2} \leq C \varepsilon^2 \|\phi_1 - \phi_2\|_{*}.
\end{equation}
For the third component, apply the Leibniz rule to the fourth-order gradient:
\begin{equation}\label{eq:leibniz_app}
    \begin{split}
        &\varepsilon^2 \left\| \nabla_{\varepsilon}^4 \left[ \nabla_{\varepsilon}(2q + \phi_1 + \phi_2) \cdot \nabla_{\varepsilon}(\phi_1 - \phi_2) \right] \right\|_{L^2} \\
        &\leq \varepsilon^2 \sum_{k=0}^{4} \binom{4}{k} \left\| \left( \nabla_{\varepsilon}^{k+1} (2q + \phi_1 + \phi_2) \right) \cdot \left( \nabla_{\varepsilon}^{5-k} (\phi_1 - \phi_2) \right) \right\|_{L^2} \\
        &\leq C\varepsilon^2 \sum_{k=0}^{2} \left\| \nabla_{\varepsilon}^{k+1} (2q + \phi_1 + \phi_2) \right\|_{L^\infty} \left\| \nabla_{\varepsilon}^{5-k} (\phi_1 - \phi_2) \right\|_{L^2} \\
        &\quad + C\varepsilon^2 \sum_{k=3}^{4} \left\| \nabla_{\varepsilon}^{k+1} (2q + \phi_1 + \phi_2) \right\|_{L^2} \left\| \nabla_{\varepsilon}^{5-k} (\phi_1 - \phi_2) \right\|_{L^\infty}.
    \end{split}
\end{equation}
The required derivative norms are bounded as follows:
\begin{align}
    \left\| \nabla_{\varepsilon}^{m} (2q + \phi_1 + \phi_2) \right\|_{L^2} &\leq C, \quad m = 4,5, \label{eq:high_deriv_L2} \\
    \left\| \nabla_{\varepsilon}^{m} (2q + \phi_1 + \phi_2) \right\|_{L^\infty} &\leq 
    \begin{cases} 
        C, & m = 1,2, \\
        C \varepsilon^{-\frac{1}{2}}, & m = 3. \label{eq:deriv_inf}
    \end{cases}
\end{align}
Combining \eqref{eq:high_deriv_L2}, \eqref{eq:deriv_inf}, and \eqref{eq:leibniz_app}, the dominant term in \eqref{eq:leibniz_app} is $\mathcal{O}(\varepsilon)$, yielding
\begin{equation}\label{eq:high_deriv_bound}
    \varepsilon^2 \left\| \nabla_{\varepsilon}^4 \left[ \nabla_{\varepsilon}(2q + \phi_1 + \phi_2) \cdot \nabla_{\varepsilon}(\phi_1 - \phi_2) \right] \right\|_{L^2} \leq C \varepsilon \|\phi_1 - \phi_2\|_{*}.
\end{equation}
Synthesizing \eqref{eq:inf_bound}, \eqref{eq:L2_bound}, and \eqref{eq:high_deriv_bound} into \eqref{eq:second_term_decomp}, we conclude
\begin{equation}\label{eq:second_term_final}
    \left\| \frac{\varepsilon^2}{2c} \nabla_{\varepsilon}(2q + \phi_1 + \phi_2) \cdot \nabla_{\varepsilon}(\phi_1 - \phi_2) \right\|_{\mathcal{F}_{5}} \leq C \varepsilon \|\phi_1 - \phi_2\|_{*}.
\end{equation}

For the third term in $\mathfrak{h}_{\varepsilon,\phi_1} - \mathfrak{h}_{\varepsilon,\phi_2}$, we decompose the difference as
\begin{equation}\label{eq:g1g2_decomp}
    \begin{split}
        \| g_1(h_1) g_2(h_1, f_1) - g_1(h_2) g_2(h_2, f_2) \|_{\mathcal{F}_{5}} 
        &\leq \| g_1(h_1) \big( g_2(h_1, f_1) - g_2(h_2, f_2) \big) \|_{\mathcal{F}_{5}} \\
        &\quad + \| \big( g_1(h_1) - g_1(h_2) \big) g_2(h_2, f_2) \|_{\mathcal{F}_{5}}.
    \end{split}
\end{equation}
\noindent\textbf{Estimation of the first term in \eqref{eq:g1g2_decomp}.}
Applying the definition of $\|\cdot\|_{\mathcal{F}_5}$ and Sobolev embedding \eqref{eq:sobolev_embed}, 
\begin{equation}\label{eq:g1g2_diff_F5}
    \begin{split}
        &\| g_1(h_1) \big( g_2(h_1, f_1) - g_2(h_2, f_2) \big) \|_{\mathcal{F}_{5}} \\
        &\leq \| g_1(h_1) \|_{L^{\infty}} \| g_2(h_1, f_1) - g_2(h_2, f_2) \|_{L^{\infty}} \\
        &\quad + C \| g_1(h_1) \big( g_2(h_1, f_1) - g_2(h_2, f_2) \big) \|_{L^2} \\
        &\quad + C \| \nabla_{\varepsilon}^4 \big[ g_1(h_1) \big( g_2(h_1, f_1) - g_2(h_2, f_2) \big) \big] \|_{L^2}.
    \end{split}
\end{equation}
We establish the following bounds:
\begin{align}
    \| g_1(h_1) \|_{L^{\infty}} &\leq C \varepsilon^4, \label{eq:g1_inf_bound} \\
    \| g_2(h_1, f_1) - g_2(h_2, f_2) \|_{L^{\infty}} &\leq C \left( \| \partial_1(\psi_1 - \psi_2) \|_{L^{\infty}} + \varepsilon^2 \| \nabla_{\varepsilon} (\phi_1 - \phi_2) \|_{L^{\infty}} \right), \label{eq:g2_diff_inf} \\
    \| g_1(h_1) \big( g_2(h_1, f_1) - g_2(h_2, f_2) \big) \|_{L^2} &\leq C \varepsilon^4 \left( \| \nabla_{\varepsilon} (\psi_1 - \psi_2) \|_{L^2} + \varepsilon^2 \| \nabla_{\varepsilon} (\phi_1 - \phi_2) \|_{L^2} \right). \label{eq:g1g2_diff_L2}
\end{align}
For the high-order derivative term, the Leibniz rule yields:
\begin{equation}\label{eq:leibniz_g1g2_diff}
    \begin{split}
        &\| \nabla_{\varepsilon}^4 \big[ g_1(h_1) \big( g_2(h_1, f_1) - g_2(h_2, f_2) \big) \big] \|_{L^2} \\
        &\leq C \sum_{m=0}^{3} \| \nabla_{\varepsilon}^m g_1(h_1) \|_{L^{\infty}} \| \nabla_{\varepsilon}^{4-m} \big( g_2(h_1, f_1) - g_2(h_2, f_2) \big) \|_{L^2} \\
        &\quad + C \| \nabla_{\varepsilon}^4 g_1(h_1) \|_{L^{2}} \| g_2(h_1, f_1) - g_2(h_2, f_2) \|_{L^{\infty}}.
    \end{split}
\end{equation}
The derivatives of $g_1$ satisfy:
\begin{align}
    \|  g_1(h_1) \|_{L^{\infty}} & \leq C \varepsilon^4 \quad \| \nabla_{\varepsilon} g_1(h_1) \|_{L^{\infty}} \leq C \varepsilon^8, \quad 
    \| \nabla_{\varepsilon}^2 g_1(h_1) \|_{L^{\infty}} \leq C \varepsilon^8, \label{eq:g1_deriv_inf} \\
    \| \nabla_{\varepsilon}^3 g_1(h_1) \|_{L^{\infty}} &\leq C \varepsilon^6, \quad 
    \| \nabla_{\varepsilon}^4 g_1(h_1) \|_{L^{\infty}} \leq C \varepsilon^8. \label{eq:g1_high_deriv_inf}
\end{align}
The difference $g_2(h_1, f_1) - g_2(h_2, f_2)$ admits the estimates:
\begin{align}
\|  \big( g_2(h_1, f_1) - g_2(h_2, f_2) \big) \|_{L^\infty} &\leq C \left( \varepsilon^{-\frac{1}{2}} \| \psi_1 - \psi_2 \|_{h} + \varepsilon^{\frac{3}{2}} \| \phi_1 - \phi_2 \|_{*} \right), \label{eq:g2_diff_inft} \\
    \| \nabla_{\varepsilon} \big( g_2(h_1, f_1) - g_2(h_2, f_2) \big) \|_{L^2} &\leq C \left( \varepsilon^{-\frac{1}{2}} \| \psi_1 - \psi_2 \|_{h} + \varepsilon^{\frac{3}{2}} \| \phi_1 - \phi_2 \|_{*} \right), \label{eq:g2_diff_grad_L2} \\
    \| \nabla_{\varepsilon}^2 \big( g_2(h_1, f_1) - g_2(h_2, f_2) \big) \|_{L^2} &\leq C \left( \varepsilon^{-\frac{3}{2}} \| \psi_1 - \psi_2 \|_{h} + \varepsilon \| \phi_1 - \phi_2 \|_{*} \right), \label{eq:g2_diff_hess_L2} \\
    \| \nabla_{\varepsilon}^3 \big( g_2(h_1, f_1) - g_2(h_2, f_2) \big) \|_{L^2} &\leq C \left( \varepsilon^{-\frac{3}{2}} \| \psi_1 - \psi_2 \|_{h} + \varepsilon \| \phi_1 - \phi_2 \|_{*} \right), \label{eq:g2_diff_third_L2} \\
    \| \nabla_{\varepsilon}^4 \big( g_2(h_1, f_1) - g_2(h_2, f_2) \big) \|_{L^2} &\leq C \left( \varepsilon^{-\frac{5}{2}} \| \psi_1 - \psi_2 \|_{h} + \varepsilon^{-\frac{1}{2}} \| \phi_1 - \phi_2 \|_{*} \right). \label{eq:g2_diff_fourth_L2}
\end{align}
Combining \eqref{eq:g1_deriv_inf}--\eqref{eq:g1_high_deriv_inf} and \eqref{eq:g2_diff_inft}--\eqref{eq:g2_diff_fourth_L2}, we obtain
\begin{equation}\label{eq:high_deriv_g1g2_diff}
    \| \nabla_{\varepsilon}^4 \big[ g_1(h_1) \big( g_2(h_1, f_1) - g_2(h_2, f_2) \big) \big] \|_{L^2} \leq C \varepsilon^{\frac{3}{2}} \left( \| \psi_1 - \psi_2 \|_{h} + \varepsilon^2 \| \phi_1 - \phi_2 \|_{*} \right).
\end{equation}

\noindent\textbf{Estimation of the second term in \eqref{eq:g1g2_decomp}.}
Similarly, we have
\begin{equation}\label{eq:g1_diff_g2_F5}
    \begin{split}
        &\| \big( g_1(h_1) - g_1(h_2) \big) g_2(h_2, f_2) \|_{\mathcal{F}_{5}} \\
        &\leq \| g_2(h_2, f_2) \|_{L^{\infty}} \| g_1(h_1) - g_1(h_2) \|_{L^{\infty}} \\
        &\quad + C \| \big( g_1(h_1) - g_1(h_2) \big) g_2(h_2, f_2) \|_{L^2} \\
        &\quad + C \| \nabla_{\varepsilon}^4 \big[ \big( g_1(h_1) - g_1(h_2) \big) g_2(h_2, f_2) \big] \|_{L^2}.
    \end{split}
\end{equation}
The uniform bounds yield:
\begin{align}
    \| g_2(h_2, f_2) \|_{L^{\infty}} \| g_1(h_1) - g_1(h_2) \|_{L^{\infty}} &\leq C \varepsilon^9 \| \psi_1 - \psi_2 \|_{h}, \label{eq:g1_diff_inf} \\
    \| \big( g_1(h_1) - g_1(h_2) \big) g_2(h_2, f_2) \|_{L^2} &\leq C \varepsilon^9 \| \psi_1 - \psi_2 \|_{h}. \label{eq:g1_diff_g2_L2}
\end{align}
For the derivative term, Leibniz rule gives
\begin{equation}\label{eq:leibniz_g1_diff_g2}
    \begin{split}
        &\| \nabla_{\varepsilon}^4 \big[ \big( g_1(h_1) - g_1(h_2) \big) g_2(h_2, f_2) \big] \|_{L^2} \\
        &\leq C \sum_{m=0}^{2} \| \nabla_{\varepsilon}^m g_2(h_2, f_2) \|_{L^{\infty}} \| \nabla_{\varepsilon}^{4-m} \big( g_1(h_1) - g_1(h_2) \big) \|_{L^2} \\
        &\quad + C \sum_{m=3}^{4} \| \nabla_{\varepsilon}^m g_2(h_2, f_2) \|_{L^{2}} \| \nabla_{\varepsilon}^{4-m} \big( g_1(h_1) - g_1(h_2) \big) \|_{L^{\infty}} \\
        &\leq C \varepsilon^6 \| \psi_1 - \psi_2 \|_{h}. 
    \end{split}
\end{equation}

\noindent\textbf{Synthesis.}
Combining \eqref{eq:g1g2_diff_F5}--\eqref{eq:g1g2_diff_L2}, \eqref{eq:high_deriv_g1g2_diff}, and \eqref{eq:g1_diff_inf}--\eqref{eq:leibniz_g1_diff_g2} into \eqref{eq:g1g2_decomp}, we obtain the estimate:
\begin{equation}\label{eq:g1g2_final_bound}
    \| g_1(h_1) g_2(h_1, f_1) - g_1(h_2) g_2(h_2, f_2) \|_{\mathcal{F}_{5}} \leq C \varepsilon^{\frac{3}{2}} \left( \| \psi_1 - \psi_2 \|_{h} + \varepsilon^2 \| \phi_1 - \phi_2 \|_{*} \right).
\end{equation}

Finally, we estimate the last term in $\mathfrak{h}_{\varepsilon,\phi_1} - \mathfrak{h}_{\varepsilon,\phi_2}$. We decompose the divergence term as:
\begin{equation}\label{eq:div_decomp}
    \begin{split}
        \| \nabla_{\varepsilon} \cdot \left[ g_3(h_1) \nabla_{\varepsilon} h_1 - g_3(h_2) \nabla_{\varepsilon} h_2 \right] \|_{\mathcal{F}_5} 
        &\leq \| \nabla_{\varepsilon} \cdot \left[ g_3(h_1) \nabla_{\varepsilon} (h_1 - h_2) \right] \|_{\mathcal{F}_5} \\
        &\quad + \| \nabla_{\varepsilon} \cdot \left[ (g_3(h_1) - g_3(h_2)) \nabla_{\varepsilon} h_2 \right] \|_{\mathcal{F}_5}.
    \end{split}
\end{equation}

\noindent\textbf{Estimation of the first term in \eqref{eq:div_decomp}.}
Applying the $\mathcal{F}_5$-norm definition:
\begin{equation}\label{eq:div_term1_F5}
    \begin{split}
        \| \nabla_{\varepsilon} \cdot \left[ g_3(h_1) \nabla_{\varepsilon} (h_1 - h_2) \right] \|_{\mathcal{F}_5} 
        &\leq C \| \nabla_{\varepsilon} \cdot \left[ g_3(h_1) \nabla_{\varepsilon} (h_1 - h_2) \right] \|_{L^{\infty}} \\
        &\quad + C \| \nabla_{\varepsilon} \cdot \left[ g_3(h_1) \nabla_{\varepsilon} (h_1 - h_2) \right] \|_{L^{2}} \\
        &\quad + C \| \nabla_{\varepsilon}^5 \left[ g_3(h_1) \nabla_{\varepsilon} (h_1 - h_2) \right] \|_{L^{2}}.
    \end{split}
\end{equation}
The first two components satisfy:
\begin{align}
    &\| \nabla_{\varepsilon} \cdot \left[ g_3(h_1) \nabla_{\varepsilon} (h_1 - h_2) \right] \|_{L^{\infty}} \nonumber \\
    &\leq C \left( \varepsilon^6 \| \nabla_{\varepsilon} (\psi_1 - \psi_2) \|_{L^{\infty}} + \varepsilon^7 \| \nabla_{\varepsilon}^2 (\psi_1 - \psi_2) \|_{L^{\infty}} \right) \leq C \varepsilon^5 \| \psi_1 - \psi_2 \|_{h}, \label{eq:div_term1_inf} \\
    &\| \nabla_{\varepsilon} \cdot \left[ g_3(h_1) \nabla_{\varepsilon} (h_1 - h_2) \right] \|_{L^{2}} \nonumber \\
    &\leq C \left( \varepsilon^6 \| \nabla_{\varepsilon} (\psi_1 - \psi_2) \|_{L^{2}} + \varepsilon^7 \| \nabla_{\varepsilon}^2 (\psi_1 - \psi_2) \|_{L^{2}} \right) \leq C \varepsilon^5 \| \psi_1 - \psi_2 \|_{h}. \label{eq:div_term1_L2}
\end{align}
For the high-order derivative term, Leibniz rule yields:
\begin{equation}\label{eq:div_term1_high}
    \begin{split}
        &\| \nabla_{\varepsilon}^5 \left[ g_3(h_1) \nabla_{\varepsilon} (h_1 - h_2) \right] \|_{L^{2}} \\
        &\leq \sum_{m=0}^{3} \| \nabla_{\varepsilon}^m g_3(h_1) \|_{L^{\infty}} \| \nabla_{\varepsilon}^{6-m} (h_1 - h_2) \|_{L^{2}} \\
        &\quad + \sum_{m=4}^{5} \| \nabla_{\varepsilon}^m g_3(h_1) \|_{L^{2}} \| \nabla_{\varepsilon}^{6-m} (h_1 - h_2) \|_{L^{\infty}} \\
        &\leq C \varepsilon^5 \| \psi_1 - \psi_2 \|_{h}.
    \end{split}
\end{equation}
Combining \eqref{eq:div_term1_inf}, \eqref{eq:div_term1_L2}, and \eqref{eq:div_term1_high}:
\begin{equation}\label{eq:div_term1_final}
    \| \nabla_{\varepsilon} \cdot \left[ g_3(h_1) \nabla_{\varepsilon} (h_1 - h_2) \right] \|_{\mathcal{F}_5} \leq C \varepsilon^5 \| \psi_1 - \psi_2 \|_{h}.
\end{equation}

\noindent\textbf{Estimation of the second term in \eqref{eq:div_decomp}.}
Similarly, we have
\begin{equation}\label{eq:div_term2_F5}
    \begin{split}
        &\| \nabla_{\varepsilon} \cdot \left[ (g_3(h_1) - g_3(h_2)) \nabla_{\varepsilon} h_2 \right] \|_{\mathcal{F}_5} \\
        &\leq C \| \nabla_{\varepsilon} \cdot \left[ (g_3(h_1) - g_3(h_2)) \nabla_{\varepsilon} h_2 \right] \|_{L^{\infty}} \\
        &\quad + C \| \nabla_{\varepsilon} \cdot \left[ (g_3(h_1) - g_3(h_2)) \nabla_{\varepsilon} h_2 \right] \|_{L^{2}} \\
        &\quad + C \| \nabla_{\varepsilon}^5 \left[ (g_3(h_1) - g_3(h_2)) \nabla_{\varepsilon} h_2 \right] \|_{L^{2}}.
    \end{split}
\end{equation}
The uniform bounds yield:
\begin{align}
    &\| \nabla_{\varepsilon} \cdot \left[ (g_3(h_1) - g_3(h_2)) \nabla_{\varepsilon} h_2 \right] \|_{L^{\infty}} \nonumber \\
    &\leq C \left( \varepsilon^7 \| \nabla_{\varepsilon} (\psi_1 - \psi_2) \|_{L^{\infty}} + \varepsilon^8 \| \nabla_{\varepsilon}^2 (\psi_1 - \psi_2) \|_{L^{\infty}} \right) \leq C \varepsilon^6 \| \psi_1 - \psi_2 \|_{h}, \label{eq:div_term2_inf} \\
    &\| \nabla_{\varepsilon} \cdot \left[ (g_3(h_1) - g_3(h_2)) \nabla_{\varepsilon} h_2 \right] \|_{L^{2}} \nonumber \\
    &\leq C \varepsilon^8 \left( \| \nabla_{\varepsilon} (\psi_1 - \psi_2) \|_{L^{\infty}} + \| \nabla_{\varepsilon}^2 (\psi_1 - \psi_2) \|_{L^{2}} \right) \leq C \varepsilon^6 \| \psi_1 - \psi_2 \|_{h}. \label{eq:div_term2_L2}
\end{align}
For the high-order derivative term, note that $g_3(h)$ depends on $|\nabla_{\varepsilon} h|^2$:
\begin{equation}\label{eq:div_term2_high}
    \begin{split}
        &\| \nabla_{\varepsilon}^5 \left[ (g_3(h_1) - g_3(h_2)) \nabla_{\varepsilon} h_2 \right] \|_{L^{2}} \\
        &\leq C \varepsilon^8 \sum_{m=1}^{3} \| \nabla_{\varepsilon}^m h_2 \|_{L^{\infty}} \| \nabla_{\varepsilon}^{6-m} (|\nabla_{\varepsilon} h_1|^2 - |\nabla_{\varepsilon} h_2|^2) \|_{L^2} \\
        &\quad + C \varepsilon^8 \sum_{m=4}^{6} \| \nabla_{\varepsilon}^m h_2 \|_{L^{2}} \| \nabla_{\varepsilon}^{6-m} (|\nabla_{\varepsilon} h_1|^2 - |\nabla_{\varepsilon} h_2|^2) \|_{L^{\infty}}.
    \end{split}
\end{equation}
The required derivative bounds are:
\begin{align}
    \| \nabla_{\varepsilon} h_2 \|_{L^{\infty}} &\leq C, & 
    \| \nabla_{\varepsilon}^2 h_2 \|_{L^{\infty}} &\leq C \varepsilon^{-\frac{1}{2}}, &
    \| \nabla_{\varepsilon}^3 h_2 \|_{L^{\infty}} &\leq C \varepsilon^{-\frac{3}{2}}, \label{eq:h2_deriv_inf} \\
    \| \nabla_{\varepsilon}^4 h_2 \|_{L^{2}} &\leq C, & 
    \| \nabla_{\varepsilon}^5 h_2 \|_{L^{2}} &\leq C \varepsilon^{-1}, & 
    \| \nabla_{\varepsilon}^6 h_2 \|_{L^{2}} &\leq C \varepsilon^{-2}, \label{eq:h2_high_deriv_L2}
\end{align}
and
\begin{align}
    \| \nabla_{\varepsilon}^m (|\nabla_{\varepsilon} h_1|^2 - |\nabla_{\varepsilon} h_2|^2) \|_{L^{\infty}} 
    &\leq C \varepsilon^{-\frac{2m+1}{2}} \| \psi_1 - \psi_2 \|_{h}, \quad m=0,1,2, \label{eq:gradh_diff_inf} \\
    \| \nabla_{\varepsilon}^3 (|\nabla_{\varepsilon} h_1|^2 - |\nabla_{\varepsilon} h_2|^2) \|_{L^2} 
    &\leq C \varepsilon^{-\frac{3}{2}} \| \psi_1 - \psi_2 \|_{h}, \label{eq:gradh_diff_third_L2} \\
    \| \nabla_{\varepsilon}^5 (|\nabla_{\varepsilon} h_1|^2 - |\nabla_{\varepsilon} h_2|^2) \|_{L^2} 
    &\leq C \varepsilon^{-3} \| \psi_1 - \psi_2 \|_{h}. \label{eq:gradh_diff_fifth_L2}
\end{align}
Combining \eqref{eq:h2_deriv_inf}--\eqref{eq:gradh_diff_fifth_L2} into \eqref{eq:div_term2_high}:
\begin{equation}\label{eq:div_term2_high_final}
    \| \nabla_{\varepsilon}^5 \left[ (g_3(h_1) - g_3(h_2)) \nabla_{\varepsilon} h_2 \right] \|_{L^{2}} \leq C \varepsilon^4 \| \psi_1 - \psi_2 \|_{h}.
\end{equation}
Synthesizing \eqref{eq:div_term2_inf}, \eqref{eq:div_term2_L2}, and \eqref{eq:div_term2_high_final}:
\begin{equation}\label{eq:div_term2_final}
    \| \nabla_{\varepsilon} \cdot \left[ (g_3(h_1) - g_3(h_2)) \nabla_{\varepsilon} h_2 \right] \|_{\mathcal{F}_5} \leq C \varepsilon^4 \| \psi_1 - \psi_2 \|_{h}.
\end{equation}

\noindent\textbf{Conclusion.}
Substituting \eqref{eq:div_term1_final} and \eqref{eq:div_term2_final} into \eqref{eq:div_decomp}, we obtain the final estimate for the divergence term:
\begin{equation}\label{eq:div_final}
    \| \nabla_{\varepsilon} \cdot \left[ g_3(h_1) \nabla_{\varepsilon} h_1 - g_3(h_2) \nabla_{\varepsilon} h_2 \right] \|_{\mathcal{F}_5} \leq C \varepsilon^4 \| \psi_1 - \psi_2 \|_{h}.
\end{equation}
This completes the proof of the required estimates for all components of $\mathfrak{h}_{\varepsilon,\phi_1} - \mathfrak{h}_{\varepsilon,\phi_2}$.
\end{proof}

Now we are able to give the proof of Theorem \ref{Thm1.1}.

\begin{proof}[ Proof of Theorem \ref{Thm1.1}]

With Lemmas \ref{Lemma5-2} and \ref{Lemma 5.3}, we have established that for a given function $\phi$, a solution to \eqref{5-2} exists. Consequently, the original problem is equivalent to finding a solution to \eqref{5-1}. We rewrite \eqref{5-1} as
\begin{equation}
    \label{5-7}
    \begin{split}
        \mathcal{L}_{\ep}\phi &= \hat{P}_1 + \hat{P}_2 + P_3 + \hat{P}_4,
    \end{split}
\end{equation}
where
\begin{align*}
    \hat{P}_1 &:= P_1 + \frac{3}{2c} \partial_1 \bigl( (\partial_1 \phi)^2 \bigr) - \left(2A + \tfrac{1}{3}\right) \ep^2 \partial_1^2 \partial_2^2 q, \\
    \hat{P}_2 &:= P_2 - \sigma (1 + \ep^2) \ep^4 \partial_2^4 q, \\
    \hat{P}_4 &:= P_4 - \mathcal{L}_2 q.
\end{align*}
We shall find a solution to \eqref{5-7} using a fixed-point argument.

For any $\phi \in \mathcal{F}_\phi$, we first establish the estimate
\begin{equation}\label{5-9}
    \big\|\partial_1^{-1}\hat{P}_1\big\|_{**} \leq C \varepsilon.
\end{equation}    
To prove this, we verify that all terms in $\partial_1^{-1}\hat{P}_1$ are bounded by $C \varepsilon$ in the $\|\cdot\|_{**}$ norm. We illustrate the process with $(\partial_1 \phi)^2$ and $\varepsilon^2 \mathcal{Q}\Pi$, as representative cases; the remaining terms can be handled similarly, and we omit the details for brevity. 

For $(\partial_1 \phi)^2$, we compute
\begin{align*}
    \|(\partial_1 \phi)^2\|_{L^2(\mathbb{R}^2)} &\leq \|\partial_1\phi\|_{L^2}\|\partial_1\phi\|_{L^\infty} \leq C\|\phi\|_{*}, \\
    \|\partial_1((\partial_1 \phi)^2)\|_{L^2} &\leq C\|\partial_1^2\phi\|_{L^2}\|\partial_1\phi\|_{L^\infty} \leq C\|\phi\|_{*}, \\
    \|(\partial_1 \phi)^2\|_{L^{4/3}} &= \|\partial_1\phi\|_{L^{8/3}}^2 \leq \|\partial_1\phi\|_{L^2}^{3/2}\|\partial_1\phi\|_{L^\infty}^{1/2} \leq C\|\phi\|_{*}^{3/2}.
\end{align*}
This yields
\[
\|(\partial_1 \phi)^2\|_{**} \leq C \varepsilon.
\]

For $\varepsilon^2 \mathcal{Q}\Pi = \ep^2\Pi + \frac{1}{3}\ep^4\mathcal{P}\Pi$, where 
\begin{align*}
    \Pi &= -\frac{1}{2}(\partial_2 f)^2 + \frac{1}{2(1 + \ep^6|\nabla_{\ep} h|^2)}(-c\partial_1h + \ep^2\nabla_{\ep} h \cdot \nabla_{\ep} f)^2 \\
    &\quad + \sigma \ep^4 \nabla_{\ep} \cdot \left[ \frac{-|\nabla_{\ep}h|^2\nabla_{\ep} h}{\sqrt{1+ \ep^6|\nabla_{\ep} h|^2}(1+\sqrt{1+ \ep^6|\nabla_{\ep} h|^2})} \right],
\end{align*}
we begin with $(\partial_2f)^2$:  
\begin{align*}
    \ep^2\|(\partial_2f)^2\|_{**} &\leq C\ep^2(\|(\partial_2q)^2\|_{**} + \|(\partial_2\phi)^2\|_{**}) \\
    &\leq C\ep^2(1 + \|(\partial_2\phi)^2\|_{**}),
\end{align*}
where 
\begin{align*}
    \ep^2\|(\partial_2\phi)^2\|_{**} &= \ep^2 \left( \|(\partial_2\phi)^2\|_{L^2} + \|\partial_1((\partial_2\phi)^2)\|_{L^2} + \|(\partial_2\phi)^2\|_{L^{4/3}} \right) \\
    &\leq C \|\phi\|_{*} \leq C\ep.
\end{align*}
Similarly,
\begin{align*}
    \ep^4\|\mathcal{P}(\partial_2f)^2\|_{**} &\leq C\ep^4(\|\mathcal{P}(\partial_2q)^2\|_{**} + \|\mathcal{P}(\partial_2\phi)^2\|_{**}) \\
    &\leq C\ep^4(1 + \|\mathcal{P}(\partial_2\phi)^2\|_{**}),
\end{align*}
with
\begin{align*}
    \ep^4\|\mathcal{P}(\partial_2\phi)^2\|_{**} &\leq C\ep^2 \left( \|\phi\|_{*} + \|\nabla_{\ep}^3\phi\|_{L^2}\|\nabla_{\ep}\phi\|_{L^4} + \|\nabla_{\ep}^2\phi\|_{L^2}\|\nabla_{\ep}^2\phi\|_{L^4} \right) \\
    &\leq C\ep^2.
\end{align*}

Next, consider $\ep^2g_4(h)g_2(h,f)$ and $\ep^2\nabla_{\ep}\cdot [g_5(h)\nabla_{\ep}h]$, where 
\begin{align*}
    g_2(h,f) &= (-c\partial_1h + \ep^2\nabla_{\ep}h \cdot \nabla_{\ep}f)^2, \\
    g_4(h) &= \frac{1}{2(1+\ep^6|\nabla_\ep h|^2)}, \\
    g_5(h) &= \frac{\sigma \ep^4 |\nabla_{\ep}h|^2}{\sqrt{1+\ep^6|\nabla_{\ep}h|^2}(1+\sqrt{1+\ep^6|\nabla_{\ep}h|^2})}.
\end{align*}
We obtain
\begin{align*}
    \ep^2\|g_4(h)g_2(h,f)\|_{**} 
        &= \ep^2 \left( \|g_4g_2\|_{L^2} + \|\partial_1(g_4g_2)\|_{L^2} + \|g_4g_2\|_{L^{4/3}} \right) \\
        &\leq C \ep,
\end{align*}
and
\begin{align*}
    \ep^2\|\nabla_{\ep}\cdot[g_5\nabla_{\ep}h]\|_{**} 
        &= \ep^2 \left( \|\nabla_{\varepsilon}\cdot[g_5\nabla_{\varepsilon} h]\|_{L^2} + \|\partial_1\nabla_{\varepsilon}\cdot[g_5\nabla_{\varepsilon} h]\|_{L^2} + \|\nabla_{\varepsilon}\cdot[g_5\nabla_{\varepsilon} h]\|_{L^{4/3}} \right) \\
        &\leq C\ep^2.
\end{align*}
For the higher-order terms:
\begin{align*}
    \ep^4\|\mathcal{P}(g_4g_2)\|_{**} 
        &\leq \ep^4 \left( \|\nabla_{\ep}^2(g_4g_2)\|_{L^2} + \|\nabla_{\ep}^2(g_4g_2)\|_{L^{4/3}} + \|\partial_1\nabla_{\ep}^2(g_4g_2)\|_{L^2} \right) \\
        &\leq C\ep^2, \\
    \ep^4\|\mathcal{P}\nabla_{\ep}\cdot[g_5\nabla_{\ep}h]\|_{**}
        &= \ep^4 \left( \|\mathcal{P}\nabla_{\ep}\cdot[g_5\nabla_{\ep} h]\|_{L^2} + \|\partial_1\mathcal{P}\nabla_{\ep}\cdot[g_5\nabla_{\ep} h]\|_{L^2} + \|\mathcal{P}\nabla_{\ep}\cdot[g_5\nabla_{\ep} h]\|_{L^{4/3}} \right) \\
        &\leq C\ep^2.
\end{align*}
Combining these estimates yields
\[
\|\ep^2\mathcal{Q}\Pi\|_{**} \leq C \varepsilon.
\]

Next, we estimate $\|\partial_2^{-1}\hat{P}_2\|_{***} \leq C\varepsilon$ using representative terms $\varepsilon^4\mathcal{P}(h\partial_2f)$ and $\varepsilon^2h\partial_2f$ from $\partial_2^{-1} \hat{P}_2$. We compute:
\begin{align*}
    \varepsilon^4\|\mathcal{P}(h\partial_2 f)\|_{***} 
        &= \varepsilon^4 \left( \|\mathcal{P}(h\partial_2f)\|_{L^2} + \|\partial_2 \mathcal{P}(h\partial_2f)\|_{L^2} + \|\mathcal{P}(h\partial_2f)\|_{L^{4/3}} \right) \\
        &\leq C\varepsilon, \\
    \varepsilon^2\|h\partial_2 f\|_{***} 
        &= \varepsilon^2 \left( \|h\partial_2f\|_{L^2} + \|\partial_2(h\partial_2f)\|_{L^2} + \|h\partial_2f\|_{L^{4/3}} \right) \\
        &\leq C\varepsilon.
\end{align*}
This establishes the bound $\|\partial_2^{-1}\hat{P}_2\|_{***} \leq C\varepsilon$.

For $P_3$, we recall the operator representation $|D| = -\mathscr{R}_1(D)\partial_{x_1} - \mathscr{R}_2(D)\partial_{x_2}$, where $D = (-i\partial_{x_1}, -i\partial_{x_2})$ and the Riesz transforms $\mathscr{R}_j(D) = i\frac{D_j}{|D|}$ are bounded on $W^{s,q}$ for $1 < q < \infty$ and $s \in \mathbb{Z}_{\geq 0}$. The expression for $P_3$ is:
\begin{align*}
    P_3 &= -\frac{\varepsilon^{-4}}{c^2} G_2(\varepsilon^2 h) f - \frac{\varepsilon^{-4}}{c^2} R_3(\varepsilon^2 h) f - \frac{\varepsilon^{-2}}{c^2} \left( \sigma + \tfrac{1}{3} \right) \mathcal{P} G_2(\varepsilon^2 h) f \\
        &\quad - \frac{\varepsilon^{-2}}{c^2} \left( \sigma + \tfrac{1}{3} \right) \mathcal{P} R_3(\varepsilon^2 h) f - \frac{1}{3c^2} \sigma \mathcal{P}^2 R_1(\varepsilon^2 h) f.
\end{align*}

To establish the estimates $\|\mathscr{R}_1(D)|D|^{-1}P_3\|_{**} \leq C \varepsilon$ and $\|\mathscr{R}_2(D)|D|^{-1}P_3\|_{***} \leq C \varepsilon$, we analyze representative terms: $\varepsilon^{-4}R_3(\varepsilon^2 h)f$, $\varepsilon^{-2}\mathcal{P} R_3(\varepsilon^2 h)f$, and $\mathcal{P}^2 R_1(\varepsilon^2 h)f$.

\noindent\textbf{Term 1:} $\varepsilon^{-4} R_3(\varepsilon^2 h) f = -\varepsilon^{-4} \partial_1 \mathscr{R}_1(D) |D|^{-1} R_3(\eta)\xi - \varepsilon^{-3} \partial_2 \mathscr{R}_2(D) |D|^{-1} R_3(\eta)\xi$\\
Using Theorem \ref{Thm |D| R3} and Remark \ref{LemmaA1}, we estimate:
\begin{align*}
    \varepsilon^{-4} \| \mathscr{R}_1(D) |D|^{-1} R_3(\eta)\xi \|_{**} 
        &= \varepsilon^{-4} \left( \| \cdot \|_{L^2} + \| \partial_1 \cdot \|_{L^2} + \| \cdot \|_{L^{4/3}} \right) \\
    \varepsilon^{-4} \| \mathscr{R}_1(D) |D|^{-1} R_3(\eta)\xi \|_{L^2} 
        &= \varepsilon^{-5/2} \| \cdot \|_{L^2(\mathbb R^2;{\bf x'})} \\
        &\leq C \varepsilon^{-5/2} \| |D|^{-1} R_3(\eta)\xi \|_{L^2(\mathbb R^2;{\bf x'})} \\
        &\leq C \varepsilon^{-5/2} \| \eta \|_{C^1}^2 \| \eta \|_{W^{1,4}(\mathbb R^2;{\bf x'})} \| \xi \|_{W^{1,4}(\mathbb R^2;{\bf x'})} \\
        &\leq C \varepsilon^3, \\
    \varepsilon^{-4} \| \mathscr{R}_1(D) |D|^{-1} R_3(\eta)\xi \|_{L^{4/3}} 
        &= \varepsilon^{-7/4} \| \cdot \|_{L^{4/3}(\mathbb R^2;{\bf x'})} \\
        &\leq C \varepsilon^{-7/4} \| |D|^{-1} R_3(\eta)\xi \|_{L^{4/3}(\mathbb R^2;{\bf x'})} \\
        &\leq C \varepsilon^{-7/4} \| \eta \|_{C^1}^2 \| \eta \|_{H^1(\mathbb R^2;{\bf x'})} \| \xi \|_{W^{1,4}(\mathbb R^2;{\bf x'})} \\
        &\leq C \varepsilon^3.
\end{align*}
By Theorem \ref{Thm R_3}, we have  
\begin{align*}
    \ep^{-4}\|\partial_1 \mathscr{R}_1(D) |D|^{-1}{R}_3(\eta)\xi\|_{L^2(\mathbb R^2)}&= \ep^{-7/2}\|\mathscr{R}_1(D)\partial_{x_1}|D|^{-1}R_3(\eta)\xi\|_{L^2(\mathbb{R}^2; {\bf x'})} \\
            &\le  C\ep^{-7/2}\|R_3(\eta)\xi\|_{L^2(\mathbb{R}^2; {\bf x'})}\\
			&\le C\ep^{-7/2}\|\eta\|_{C^1}^2\|\partial_{{\bf x'}}\eta\|_{W^{1,4}(\mathbb{R}^2; {\bf x'})}\|\xi\|_{W^{1,4}(\mathbb{R}^2; {\bf x'})}\\
			&\quad+C\ep^{-7/2}\|\eta\|_{C^1}^2\|\eta\|_{W^{1,4}(\mathbb{R}^2; {\bf x'})}\| \partial_{{\bf x'}}\xi\|_{W^{1,4}(\mathbb{R}^2; {\bf x'})}\\
			&\le C\ep^3.
\end{align*}
Combining these yields:
\[
\varepsilon^{-4} \| \mathscr{R}_1(D) |D|^{-1} R_3(\eta)\xi \|_{**} \leq C \varepsilon^3.
\]
Similarly, $\varepsilon^{-3} \| \mathscr{R}_2(D) |D|^{-1} R_3(\eta)\xi \|_{***} \leq C \varepsilon^3$.

\noindent\textbf{Term 2:} $\varepsilon^{-2} \mathcal{P} R_3(\varepsilon^2 h) f = -\varepsilon^{-2} \partial_1^2 R_3(\varepsilon^2 h) f - \partial_2^2 R_3(\varepsilon^2 h) f$\\
Theorem \ref{Thm R_3} provides:
\begin{align*}
    \varepsilon^{-2} \| \partial_1 R_3(\varepsilon^2 h) f \|_{**} 
        &= \varepsilon^{-2} \left( \| \partial_1 R_3 \|_{L^2} + \| \partial_1 R_3 \|_{L^{4/3}} + \| \partial_1^2 R_3 \|_{L^2} \right) \\
    \varepsilon^{-2} \| \partial_1 R_3(\varepsilon^2 h) f \|_{L^2} 
        &= \varepsilon^{-5/2} \| \partial_{x_1} R_3(\eta)\xi \|_{L^2(\mathbb R^2;{\bf x'})} \\
        &\leq \varepsilon^{-5/2} \| R_3(\eta)\xi \|_{H^1(\mathbb R^2;{\bf x'})} \\
        &\leq C \varepsilon^{-5/2} \| \eta \|_{C^2}^2 \| \partial_{{\bf x'}} \eta \|_{W^{2,4}(\mathbb R^2;{\bf x'})} \| \xi \|_{W^{2,4}(\mathbb R^2;{\bf x'})} \\
        &\quad + C \varepsilon^{-5/2} \| \eta \|_{C^2}^2 \| \eta \|_{W^{2,4}(\mathbb R^2;{\bf x'})} \| \partial_{{\bf x'}} \xi \|_{W^{2,4}(\mathbb R^2;{\bf x'})} \\
        &\leq C \varepsilon^4, \\
    \varepsilon^{-2} \| \partial_1 R_3(\varepsilon^2 h) f \|_{L^{4/3}} 
        &= \varepsilon^{-7/4} \| \partial_{x_1} R_3(\eta)\xi \|_{L^{4/3}(\mathbb R^2;{\bf x'})} \\
        &\leq \varepsilon^{-7/4} \| R_3(\eta)\xi \|_{W^{1,4/3}(\mathbb R^2;{\bf x'})} \\
        &\leq C \varepsilon^{-7/4} \| \eta \|_{C^2}^2 \| \partial_{{\bf x'}} \eta \|_{H^2(\mathbb R^2;{\bf x'})} \| \xi \|_{W^{2,4}(\mathbb R^2;{\bf x'})} \\
        &\quad + C \varepsilon^{-7/4} \| \eta \|_{C^2}^2 \| \eta \|_{H^2(\mathbb R^2;{\bf x'})} \| \partial_{{\bf x'}} \xi \|_{W^{2,4}(\mathbb R^2;{\bf x'})} \\
        &\leq C \varepsilon^4, \\
    \varepsilon^{-2} \| \partial_1^2 R_3(\varepsilon^2 h) f \|_{L^2} 
        &= \varepsilon^{-7/2} \| \partial_{x_1}^2 R_3(\eta)\xi \|_{L^2(\mathbb R^2;{\bf x'})} \\
        &\leq \varepsilon^{-7/2} \| R_3(\eta)\xi \|_{H^2(\mathbb R^2;{\bf x'})} \\
        &\leq C \varepsilon^{-7/2} \| \eta \|_{C^3}^2 \| \partial_{{\bf x'}} \eta \|_{W^{3,4}(\mathbb R^2;{\bf x'})} \| \xi \|_{W^{3,4}(\mathbb R^2;{\bf x'})} \\
        &\quad + C \varepsilon^{-7/2} \| \eta \|_{C^3}^2 \| \eta \|_{W^{3,4}(\mathbb R^2;{\bf x'})} \| \partial_{{\bf x'}} \xi \|_{W^{3,4}(\mathbb R^2;{\bf x'})} \\
        &\leq C \varepsilon^3.
\end{align*}
Thus, 
\[
\| \varepsilon^{-2} \partial_1 R_3(\varepsilon^2 h) f \|_{**} \leq C \varepsilon^3.
\]
Similarly, $\| \partial_2 R_3(\varepsilon^2 h) f \|_{***} \leq C \varepsilon^3$.

\noindent\textbf{Term 3:} $\mathcal{P}^2 R_1(\varepsilon^2 h) f = -\partial_1^2 \mathcal{P} R_1(\varepsilon^2 h) f - \varepsilon^2 \partial_2^2 \mathcal{P} R_1(\varepsilon^2 h) f$\\
We compute:
\begin{align*}
    \| \partial_1 \mathcal{P} R_1(\varepsilon^2 h) f \|_{**} 
        &= \sum_{m=1}^2 \| \partial_1^m \mathcal{P} R_1 \|_{L^2} + \| \partial_1 \mathcal{P} R_1 \|_{L^{4/3}}.
\end{align*}
From \eqref{2-5}:
\begin{align*}
    \| \partial_1 \mathcal{P} R_1(\varepsilon^2 h) f \|_{L^2} 
        &= \| \partial_1 \mathcal{P} (-c \varepsilon^2 \partial_1 h - G_0 f) \|_{L^2} \\
        &\leq C \varepsilon^2 \| \partial_1^2 \nabla_{\varepsilon}^2 h \|_{L^2} + C \varepsilon \| \partial_1 \nabla_{\varepsilon}^3 f \|_{L^2} \\
        &\leq C \varepsilon, \\
    \| \partial_1^2 \mathcal{P} R_1(\varepsilon^2 h) f \|_{L^2} 
        &= \| \partial_1^2 \mathcal{P} (-c \varepsilon^2 \partial_1 h - G_0 f) \|_{L^2} \\
        &\leq C \varepsilon^2 \| \partial_1 \nabla_{\varepsilon}^4 h \|_{L^2} + C \varepsilon \| \nabla_{\varepsilon}^5 f \|_{L^2} \\
        &\leq C \varepsilon.
\end{align*}
For the $L^{4/3}$ component, using the decomposition $R_1(\eta)\xi = (B_1 G_0 + A_1)\xi + R_2(\eta)\xi$ and Theorems \ref{ThmA2'}, \ref{ThmA1'}, \ref{Thm R_2}:
\begin{align*}
    &\| B_1 G_0 \xi \|_{W^{3,4/3}(\mathbb R^2;{\bf x'})} + \| A_1 \xi \|_{W^{3,4/3}(\mathbb R^2;{\bf x'})} \\
    &\leq C \left( \| \partial_{{\bf x'}} \eta \|_{W^{3,4}(\mathbb R^2;{\bf x'})} \| \partial_{{\bf x'}} \xi \|_{H^3(\mathbb R^2;{\bf x'})} + \| \eta \|_{W^{3,4}(\mathbb R^2;{\bf x'})} \| \partial_{{\bf x'}}^2 \xi \|_{H^3(\mathbb R^2;{\bf x'})} \right) \\
    &\leq C \varepsilon^{11/4}, \\
    \| R_2(\eta)\xi \|_{W^{3,4/3}(\mathbb R^2;{\bf x'})} 
        &\leq C (1 + \| \eta \|_{C^4}) \| \eta \|_{C^4}\\
        &\quad \times\left( \| \partial_{{\bf x'}} \eta \|_{H^4(\mathbb R^2;{\bf x'})} \| \xi \|_{W^{4,4}(\mathbb R^2;{\bf x'})} + \| \eta \|_{H^4(\mathbb R^2;{\bf x'})} \| \partial_{{\bf x'}} \xi \|_{W^{4,4}(\mathbb R^2;{\bf x'})} \right) \\
        &\leq C \varepsilon^{15/4}.
\end{align*}
Thus,
\begin{align*}
    &\| \partial_1 \mathcal{P} R_1(\varepsilon^2 h) f \|_{L^{4/3}} \\
        &= \varepsilon^{-7/4} \| \partial_{x_1} (-\Delta) R_1(\eta)\xi \|_{L^{4/3}(\mathbb R^2;{\bf x'})} \\
        &\leq \varepsilon^{-7/4} \| R_1(\eta)\xi \|_{W^{3,4/3}(\mathbb R^2;{\bf x'})} \\
        &\leq C \varepsilon^{-7/4} \left( \| B_1 G_0 \xi \|_{W^{3,4/3}(\mathbb R^2;{\bf x'})} + \| A_1 \xi \|_{W^{3,4/3}(\mathbb R^2;{\bf x'})} + \| R_2(\eta)\xi \|_{W^{3,4/3}(\mathbb R^2;{\bf x'})} \right) \\
        &\leq C \varepsilon.
\end{align*}
Combining these results gives
\[
\| \partial_1 \mathcal{P} R_1(\varepsilon^2 h) f \|_{**} \leq C \varepsilon.
\]
Applying analogous estimates to the remaining terms yields the desired bounds:
\begin{equation}\label{5-11}
    \| \mathscr{R}_1(D) |D|^{-1} P_3 \|_{**} \leq C \varepsilon, \quad 
    \| \mathscr{R}_2(D) |D|^{-1} P_3 \|_{***} \leq C \varepsilon.
\end{equation}

For $\hat{P}_4$, we analyze the representative term $\mathcal{P}G_0hG_0f$ to establish the estimates
\begin{equation}\label{5-12}
    \|\mathscr{R}_1(D)|D|^{-1}\hat{P}_4\|_{**} \leq C\varepsilon, \quad 
    \|\mathscr{R}_2(D)|D|^{-1}\hat{P}_4\|_{***} \leq C\varepsilon.
\end{equation}
Using the decomposition $\mathcal{P}G_0hG_0f = -\partial_1^2G_0hG_0f - \varepsilon^2\partial_2^2G_0hG_0f$, we estimate $\|\partial_1G_0hG_0f\|_{**}$ and $\varepsilon^2\|\partial_2G_0hG_0f\|_{***}$. For the first component:
\begin{align*}
    \|\partial_1G_0hG_0f\|_{**} 
        &= \|\partial_1G_0hG_0f\|_{L^2} + \|\partial_1^2G_0hG_0f\|_{L^2} + \|\partial_1G_0hG_0f\|_{L^{4/3}}.
\end{align*}
The $L^2$-norm of the second derivative satisfies:
\begin{align*}
    \|\partial_1^2G_0hG_0f\|_{L^2} 
        &\leq C\varepsilon \|\nabla_{\varepsilon}^3(hG_0f)\|_{L^2} \\
        &\leq C\varepsilon \Big( \|\nabla_{\varepsilon}^3 h\|_{L^4} \|G_0f\|_{L^4} + \|\nabla_{\varepsilon}^2h\|_{L^4} \|\nabla_{\varepsilon}G_0f\|_{L^4} \\
        &\quad + \|\nabla_{\varepsilon} h\|_{L^\infty} \|\nabla_{\varepsilon}^2G_0f\|_{L^2} + \|h\|_{L^\infty} \|\nabla_{\varepsilon}^3G_0f\|_{L^2} \Big) \\
        &\leq C\varepsilon^2 \Big( \|\nabla_{\varepsilon}^3 h\|_{L^4} \|\nabla_{\varepsilon}f\|_{L^4} + \|\nabla_{\varepsilon}^2h\|_{L^4} \|\nabla_{\varepsilon}^2f\|_{L^4} \\
        &\quad + \|\nabla_{\varepsilon} h\|_{L^\infty} \|\nabla_{\varepsilon}^3f\|_{L^2} + \|h\|_{L^\infty} \|\nabla_{\varepsilon}^4f\|_{L^2} \Big).
\end{align*}
Applying the Gagliardo-Nirenberg inequality $\|u\|_{L^4} \leq C\|\nabla u\|_{L^2}^{1/2}\|u\|_{L^2}^{1/2}$ yields
\begin{align*}
    \|\nabla_{\varepsilon}^2h\|_{L^4} \leq C\varepsilon^{-1/4}, \quad 
    \|\nabla_{\varepsilon}^3h\|_{L^4} \leq C\varepsilon^{-1/4}, \\
    \|\nabla_{\varepsilon}f\|_{L^4} \leq C\varepsilon^{-1/4}, \quad 
    \|\nabla_{\varepsilon}^2f\|_{L^4} \leq C\varepsilon^{-1/4}.
\end{align*}
Substituting these bounds gives $\|\partial_1^2G_0hG_0f\|_{L^2} \leq C\varepsilon$. Similarly, $\|\partial_1G_0hG_0f\|_{L^2} \leq C\varepsilon$.

For the $L^{4/3}$ component:
\begin{align*}
    \|\partial_1G_0hG_0f\|_{L^{4/3}} 
        &\leq C\varepsilon \|\nabla_{\varepsilon}^2(hG_0f)\|_{L^{4/3}} \\
        &\leq C\varepsilon \Big( \|\nabla_{\varepsilon}^2 h\|_{L^2} \|G_0f\|_{L^4} + \|\nabla_{\varepsilon}h\|_{L^2} \|\nabla_{\varepsilon} G_0f\|_{L^4} \\
        &\quad + \|h\|_{L^{2}} \|\nabla_{\varepsilon}^2G_0f\|_{L^4} \Big) \\
        &\leq C\varepsilon.
\end{align*}
Thus, $\|\partial_1G_0hG_0f\|_{**} \leq C\varepsilon$. Analogously, $\varepsilon^2\|\partial_2G_0hG_0f\|_{***} \leq C\varepsilon$.

Combining estimates \eqref{5-9}, \eqref{5-11}, \eqref{5-12} with Proposition \ref{prop4-2}, we obtain a solution $\mathcal{N}(\phi)$ to \eqref{5-7} satisfying
\begin{align*}
    \|\mathcal{N}(\phi)\|_* 
        &\leq C \Big( \|\partial_1^{-1}\hat{P}_1\|_{**} + \|\partial_2^{-1}\hat{P}_2\|_{***} + \|\mathscr{R}_1(D)|D|^{-1}P_3\|_{**} \\
        &\quad + \|\mathscr{R}_2(D)|D|^{-1}P_3\|_{***} + \|\mathscr{R}_1(D)|D|^{-1}\hat{P}_4\|_{**} + \|\mathscr{R}_2(D)|D|^{-1}\hat{P}_4\|_{***} \Big) \\
        &\leq C \varepsilon.
\end{align*}
Consequently, the map $\phi \mapsto \mathcal{N}(\phi)$ sends $\mathcal{F}_{\phi}$ to itself when $C$ is sufficiently large.

Next, we prove that the map $\phi \mapsto \mathcal{N}(\phi)$ is a contraction. Specifically, for any $\phi_1, \phi_2 \in \mathcal{F}_{\phi}$, 
\begin{equation}\label{5-13}
    \begin{aligned}
        \|\partial_1^{-1}\hat{P}_1(\phi_1) - \partial_1^{-1}\hat{P}_1(\phi_2)\|_{**} 
            &\leq C \varepsilon^{1/2} \|\phi_1 - \phi_2\|_*, \\
        \|\partial_2^{-1}\hat{P}_2(\phi_1) - \partial_2^{-1}\hat{P}_2(\phi_2)\|_{***} 
            &\leq C \varepsilon^{1/2} \|\phi_1 - \phi_2\|_*, \\
        \|\mathscr{R}_1(D)|D|^{-1} P_3(\phi_1) - \mathscr{R}_1(D)|D|^{-1} P_3(\phi_2)\|_{**} 
            &\leq C \varepsilon^{1/2} \|\phi_1 - \phi_2\|_*, \\
        \|\mathscr{R}_2(D)|D|^{-1} P_3(\phi_1) - \mathscr{R}_2(D)|D|^{-1} P_3(\phi_2)\|_{***} 
            &\leq C \varepsilon^{1/2} \|\phi_1 - \phi_2\|_*, \\
        \|\mathscr{R}_1(D)|D|^{-1} \hat{P}_4(\phi_1) - \mathscr{R}_1(D)|D|^{-1} \hat{P}_4(\phi_2)\|_{**} 
            &\leq C \varepsilon^{1/2} \|\phi_1 - \phi_2\|_*, \\
        \|\mathscr{R}_2(D)|D|^{-1} \hat{P}_4(\phi_1) - \mathscr{R}_2(D)|D|^{-1} \hat{P}_4(\phi_2)\|_{***} 
            &\leq C \varepsilon^{1/2} \|\phi_1 - \phi_2\|_*.
    \end{aligned}
\end{equation}

We establish these estimates by analyzing representative terms from $\hat{P}_1$, $\hat{P}_2$, $P_3$, and $\hat{P}_4$. To emphasize the dependence on $\phi$, we denote $h = h_\phi$ and $f = f_\phi$.

\noindent\textbf{Term from $\hat{P}_1$: $\partial_1((\partial_1\phi)^2)$}. 
We analyze the term $\partial_1((\partial_1\phi)^2)$. Direct computation yields the following estimates:
\begin{align*}
    \|(\partial_1\phi_1)^2 - (\partial_1\phi_2)^2\|_{L^2} 
        &\leq \|\partial_1(\phi_1 + \phi_2)\|_{L^\infty} \|\partial_1(\phi_1 - \phi_2)\|_{L^2} \\
        &\leq C \varepsilon^{1/2} \|\phi_1 - \phi_2\|_*, \\
    \|(\partial_1\phi_1)^2 - (\partial_1\phi_2)^2\|_{L^{4/3}} 
        &\leq \|\partial_1(\phi_1 + \phi_2)\|_{L^4} \|\partial_1(\phi_1 - \phi_2)\|_{L^2} \\
        &\leq C \varepsilon \|\phi_1 - \phi_2\|_*, \\
    \|\partial_1((\partial_1\phi_1)^2) - \partial_1((\partial_1\phi_2)^2)\|_{L^2} 
        &\leq \|\partial_1^2(\phi_1 + \phi_2)\|_{L^\infty} \|\partial_1(\phi_1 - \phi_2)\|_{L^2} \\
        &\quad + \|\partial_1(\phi_1 + \phi_2)\|_{L^\infty} \|\partial_1^2(\phi_1 - \phi_2)\|_{L^2} \\
        &\leq C \varepsilon^{1/2} \|\phi_1 - \phi_2\|_*.
\end{align*}
Combining these estimates gives
\[
\|(\partial_1\phi_1)^2 - (\partial_1\phi_2)^2\|_{**} \leq C \varepsilon^{1/2} \|\phi_1 - \phi_2\|_*.
\]

\noindent\textbf{Term from $\hat{P}_2$: $\varepsilon^2 \partial_2(h \partial_2 f)$}.
By Lemmas \ref{Lemma5-2} and \ref{Lemma 5.3}, we have
\begin{align*}
    \varepsilon^2 \|h_{\phi_1} \partial_2 f_{\phi_1} - h_{\phi_2} \partial_2 f_{\phi_2} \|_{L^2} 
        &\leq \varepsilon^2 \|(h_{\phi_1} - h_{\phi_2}) \partial_2 f_{\phi_1} \|_{L^2}  \\
        &\quad + \varepsilon^2 \|h_{\phi_2} \partial_2 (f_{\phi_1} - f_{\phi_2}) \|_{L^2} \\
        &\leq C\varepsilon^2 \|\psi_1 - \psi_2\|_{L^2} \|\partial_2(q + \phi_1)\|_{L^\infty} \\
        &\quad + C\varepsilon^2 \|h_{\phi_2}\|_{L^\infty} \|\partial_2(\phi_1 - \phi_2)\|_{L^2} \\
        &\leq C\varepsilon \|\phi_1 - \phi_2\|_*, \\[5pt]
    \varepsilon^2 \|h_{\phi_1} \partial_2 f_{\phi_1} - h_{\phi_2} \partial_2 f_{\phi_2} \|_{L^{4/3}} 
        &\leq \varepsilon^2 \|(h_{\phi_1} - h_{\phi_2}) \partial_2 f_{\phi_1} \|_{L^{4/3}} \\
        &\quad + \varepsilon^2 \|h_{\phi_2} \partial_2 (f_{\phi_1} - f_{\phi_2}) \|_{L^{4/3}} \\
        &\leq C\varepsilon \|\phi_1 - \phi_2\|_*, \\[5pt]
    \varepsilon^2 \| \partial_2 (h_{\phi_1} \partial_2 f_{\phi_1}) - \partial_2 (h_{\phi_2} \partial_2 f_{\phi_2}) \|_{L^2} 
        &\leq \varepsilon^2 \| \partial_2 [(h_{\phi_1} - h_{\phi_2}) \partial_2 f_{\phi_1}] \|_{L^2} \\
        &\quad + \varepsilon^2 \| \partial_2 [h_{\phi_2} \partial_2 (f_{\phi_1} - f_{\phi_2})] \|_{L^2} \\
        &\leq \varepsilon^2 \| \partial_2 (h_{\phi_1} - h_{\phi_2}) \|_{L^2} \| \partial_2 f_{\phi_1} \|_{L^\infty} \\
        &\quad + \varepsilon^2 \| h_{\phi_1} - h_{\phi_2} \|_{L^2} \| \partial_2^2 f_{\phi_1} \|_{L^\infty} \\
        &\quad + \varepsilon^2 \| \partial_2 h_{\phi_2} \|_{L^\infty} \| \partial_2 (f_{\phi_1} - f_{\phi_2}) \|_{L^2} \\
        &\quad + \varepsilon^2 \| h_{\phi_2} \|_{L^\infty} \| \partial_2^2 (f_{\phi_1} - f_{\phi_2}) \|_{L^2} \\
        &\leq C\varepsilon^{1/2} \|\phi_1 - \phi_2\|_*.
\end{align*}
Thus, combining the components of the $\| \cdot \|_{***}$ norm:
\[
\varepsilon^2 \| h_{\phi_1} \partial_2 f_{\phi_1} - h_{\phi_2} \partial_2 f_{\phi_2} \|_{***} \leq C \varepsilon^{1/2} \|\phi_1 - \phi_2\|_*.
\]

\noindent\textbf{Term from $P_3$: $\ep^{-4}R_3(\ep^2h)f$}.
Consider the decomposition:
\[
\ep^{-4}{R}_3(\ep^2h)f = -\ep^{-4}\partial_1 \mathscr{R}_1(D) |D|^{-1}{R}_3(\eta)\xi - \ep^{-3}\partial_2 \mathscr{R}_2(D) |D|^{-1}{R}_3(\eta)\xi.
\]
For the first term $\ep^{-4}\partial_1 \mathscr{R}_1(D) |D|^{-1}{R}_3(\eta)\xi$, Theorem \ref{Thm |D| R3} and Corollary \ref{cor D R3} imply:
\begin{align*}
        &\|\ep^{-4} \mathscr{R}_1(D) |D|^{-1}\big({R}_3(\eta_1)\xi_1-{R}_3(\eta_2)\xi_2\big)\|_{L^2(\mathbb R^2)}\\
            &\le C\ep^{-\frac{5}{2}}\||D|^{-1}\big({R}_3(\eta_1)\xi_1-{R}_3(\eta_2)\xi_2\big)\|_{L^2(\mathbb R^2; \mathbf{x'})}\\
            &\le C\ep^{-\frac{5}{2}} \Big(\| |D|^{-1}\big({R}_3(\eta_1)-{R}_3(\eta_2)\big)\xi_1\|_{L^2(\mathbb R^2; \mathbf{x'})} \\
            &\qquad+ \||D|^{-1}{R}_3(\eta_2)(\xi_1-\xi_2)\|_{L^2(\mathbb R^2; \mathbf{x'})}\Big)\\
            &\le C\ep^{-\frac{5}{2}}\Big(\sup_{t\in[0,1]}\| |D|^{-1}\tfrac{d}{dt}{R}_3(\eta_t)\xi_1\|_{L^2(\mathbb R^2; \mathbf{x'})}
             \\
             &\qquad+ \| |D|^{-1}{R}_3(\eta_2)(\xi_1-\xi_2)\|_{L^2(\mathbb R^2; \mathbf{x'})}\Big)\\
            &\le C \ep^{\frac{7}{4}}\|\eta_2-\eta_1\|_{W^{1,4}(\mathbb R^2; \mathbf{x'})}  + C\ep^{\frac{11}{4}}\|\xi_2-\xi_1\|_{W^{1,4}(\mathbb R^2; \mathbf{x'})}\\
            &\le C\ep^{3}\|\psi_2-\psi_1\|_{L^4(\mathbb{R}^2)} + C\ep^4\|\nabla_{\ep}(\psi_2-\psi_1)\|_{L^{4}(\mathbb{R}^2)}\\
            &\hspace{2em} + C\ep^3\|\phi_2-\phi_1\|_{L^4(\mathbb R^2)} + C\ep^4\|\nabla_{\ep}(\phi_2-\phi_1)\|_{L^4(\mathbb R^2)}\\
            &\le C\ep^3\|\psi_2-\psi_1\|_{h} + C\ep^3\|\phi_2-\phi_1\|_{*} \\
            &\le C\ep^2 \|\phi_2-\phi_1\|_{*},
\end{align*}
where $\eta_t = \eta_1 + t(\eta_2 - \eta_1)$ for $t \in [0, 1]$. Similarly,
\begin{align*}
    &\|\ep^{-4} \mathscr{R}_1(D) |D|^{-1}\big({R}_3(\eta_1)\xi_1-{R}_3(\eta_2)\xi_2\big)\|_{L^{\frac{4}{3}}(\mathbb R^2)}\\
    &\le C\ep^{-\frac{7}{4}}\||D|^{-1}\big({R}_3(\eta_1)\xi_1-{R}_3(\eta_2)\xi_2\big)\|_{L^{\frac{4}{3}}(\mathbb R^2; \mathbf{x'})}\\
    &\le C\ep^{-\frac{7}{4}} \Big( \| |D|^{-1}\big({R}_3(\eta_1)-{R}_3(\eta_2)\big)\xi_1\|_{L^{\frac{4}{3}}(\mathbb R^2; \mathbf{x'})}  \\
    &\qquad+ \||D|^{-1}{R}_3(\eta_2)(\xi_1-\xi_2)\|_{L^{\frac{4}{3}}(\mathbb R^2; \mathbf{x'})} \Big)\\
    &\le C \ep^{-\frac{7}{4}} \Big( \sup_{t\in [0,1]} \| |D|^{-1} \tfrac{d}{dt}{R}_3(\eta_t) \xi_1 \|_{L^{\frac{4}{3}}(\mathbb R^2; \mathbf{x'})}  \\
    &\qquad+ \| |D|^{-1}{R}_3(\eta_2)(\xi_1-\xi_2)\|_{L^{\frac{4}{3}}(\mathbb R^2; \mathbf{x'})} \Big)\\
    &\le C\ep^{\frac{5}{2}} \|\eta_2 - \eta_1\|_{H^1(\mathbb R^2; \mathbf{x'})} + C\ep^{\frac{11}{4}} \|\xi_1 - \xi_2\|_{W^{1,4}(\mathbb R^2; \mathbf{x'})}\\
    &\le C\ep^{3} \|\psi_2 - \psi_1\|_{h} + C\ep^3 \|\phi_1 - \phi_2\|_{*}\\
    &\le C\ep^2 \|\phi_1 - \phi_2\|_{*}.
\end{align*}
By Corollary \ref{CorR3} and Theorem \ref{Thm R_3}:
\begin{align*}
    &\ep^{-4} \| \partial_1 \mathscr{R}_1(D) |D|^{-1} (R_3(\eta_1)\xi_1 - R_3(\eta_2)\xi_2) \|_{L^2(\mathbb R^2)} \\
    &\le C\ep^{-\frac{7}{2}} \| {R}_3(\eta_1)\xi_1 - {R}_3(\eta_2)\xi_2 \|_{L^2(\mathbb R^2; \mathbf{x'})} \\
    &\le C\ep^{-\frac{7}{2}}\Big( \| \big( {R}_3(\eta_1) - {R}_3(\eta_2) \big) \xi_1 \|_{L^2(\mathbb R^2; \mathbf{x'})} + \| {R}_3(\eta_2) (\xi_1 - \xi_2) \|_{L^2(\mathbb R^2; \mathbf{x'})}\Big) \\
    &\le C\ep^{-\frac{7}{2}} \Big(\sup_{t\in[0,1]} \| \tfrac{d}{dt} {R}_3(\eta_t) \xi_1 \|_{L^2(\mathbb R^2; \mathbf{x'})}  +  \| {R}_3(\eta_2) (\xi_1 - \xi_2) \|_{L^2(\mathbb R^2; \mathbf{x'})} \Big)\\
    &\le C \ep^{\frac{3}{4}} \left(\| \eta_2 - \eta_1 \|_{W^{2,4}(\mathbb R^2; \mathbf{x'})} +  \| \xi_1 - \xi_2 \|_{W^{2,4}(\mathbb R^2; \mathbf{x'})}\right) \\
    &\le C\ep^{2} \big( \| \psi_1 - \psi_2 \|_{L^4(\mathbb R^2)} + \ep \| \nabla_{\ep}(\psi_1 - \psi_2) \|_{L^4(\mathbb R^2)} \\
    &\hspace{2em} + \ep^2 \| \nabla_{\ep}^2 (\psi_1 - \psi_2) \|_{L^4(\mathbb R^2)} \big) \\
    &\hspace{1em} + C\ep \big( \| \phi_1 - \phi_2 \|_{L^4(\mathbb R^2)} + \ep \| \nabla_{\ep}(\phi_1 - \phi_2) \|_{L^4(\mathbb R^2)} \\
    &\hspace{2em} + \ep^2 \| \nabla_{\ep}^2 (\phi_1 - \phi_2) \|_{L^4(\mathbb R^2)} \big) \\
    &\le C\ep \| \phi_2 - \phi_1 \|_{*}.
\end{align*}
Therefore,
\[
\ep^{-4} \| \mathscr{R}_1(D) |D|^{-1} \left( {R}_3(\ep^2h_{\phi_1})f_{\phi_1} - {R}_3(\ep^2h_{\phi_2})f_{\phi_2} \right) \|_{**} \le C\ep \| \phi_2 - \phi_1 \|_{*}.
\]
Similarly, 
\[
\ep^{-3} \| \mathscr{R}_2(D) |D|^{-1} \left( {R}_3(\ep^2h_{\phi_1})f_{\phi_1} - {R}_3(\ep^2h_{\phi_2})f_{\phi_2} \right) \|_{**} \le C\ep \| \phi_2 - \phi_1 \|_{*}.
\]

\noindent\textbf{Term from $\hat{P}_4$: $\varepsilon^{-2} (G_0 h G_0 f)$}.
We decompose the expression as follows:
\begin{align*}
    \varepsilon^{-2} (G_0 h G_0 f) 
        &= \ep^{-5} G_0(\eta G_0 \xi) \\
        &= -\ep^{-5} \partial_{x_1} \Bigl[ \mathscr{R}_1(D) \tanh(|D|) (\eta G_0 \xi) \Bigr] - \ep^{-5} \partial_{x_2} \Bigl[ \mathscr{R}_2(D) \tanh(|D|) (\eta G_0 \xi) \Bigr] \\
        &= -\ep^{-4} \partial_1 \Bigl[ \mathscr{R}_1(D) \tanh(|D|) (\eta G_0 \xi) \Bigr] - \ep^{-3} \partial_2 \Bigl[ \mathscr{R}_2(D) \tanh(|D|) (\eta G_0 \xi) \Bigr].
\end{align*}
From Remark \ref{LemmaA1} and Lemmas \ref{Lemma5-2}, \ref{Lemma 5.3}:
\begin{align*}
    & \varepsilon^{-4} \| \mathscr{R}_1(D) \tanh{(|D|)} (\eta_1 G_0 \xi_1 - \eta_2 G_0 \xi_2) \|_{L^2(\mathbb R^2)} \\
    &= \varepsilon^{-\frac{5}{2}} \| \mathscr{R}_1(D) \tanh{(|D|)} (\eta_1 G_0 \xi_1 - \eta_2 G_0 \xi_2) \|_{L^2(\mathbb R^2; \mathbf{x'})} \\
    &\le C \varepsilon^{-\frac{5}{2}} \big( \| (\eta_1 - \eta_2) G_0 \xi_1 \|_{L^2(\mathbb R^2; \mathbf{x'})} \\
    &\hspace{2em} + \| \eta_2 G_0 (\xi_1 - \xi_2) \|_{L^2(\mathbb R^2; \mathbf{x'})} \big) \\
    &\le C\ep^{-\frac{5}{2}} \big( \| \eta_1 - \eta_2 \|_{L^{\infty}(\mathbb R^2; \mathbf{x'})} \| G_0 \xi_1 \|_{L^2(\mathbb R^2; \mathbf{x'})} \\
    &\hspace{2em} + \| \eta_2 \|_{L^{\infty}(\mathbb R^2; \mathbf{x'})} \| G_0 (\xi_1 - \xi_2) \|_{L^2(\mathbb R^2; \mathbf{x'})} \big) \\
    &\le C\ep^{-\frac{5}{2}} \big( \| \eta_1 - \eta_2 \|_{L^{\infty}(\mathbb R^2; \mathbf{x'})} \| \partial_{\mathbf{x}'}^2 \xi_1 \|_{L^2(\mathbb R^2; \mathbf{x'})} \\
    &\hspace{2em} + \| \eta_2 \|_{L^{\infty}(\mathbb R^2; \mathbf{x'})} \| \partial_{\mathbf{x}'}^2 (\xi_1 - \xi_2) \|_{L^2(\mathbb R^2; \mathbf{x'})} \big) \\
    &\le C\ep \| \psi_1 - \psi_2 \|_{L^{\infty}(\mathbb{R}^2)} + C\ep \| \nabla_{\ep}^2 (\phi_1 - \phi_2) \|_{L^2(\mathbb R^2)} \\
    &\le C\ep^{\frac{1}{2}} \| \phi_1 - \phi_2 \|_{*},
\end{align*}
where we used the inequality:
\[
\| G_0 \xi \|_{L^2(\mathbb{R}^2; \mathbf{x'})} = \| |\mathbf{k}| \tanh (|\mathbf{k}|) \widehat{\xi} \|_{L^2(\mathbb{R}^2; \mathbf{x'})} \le C \| |\mathbf{k}|^2 \widehat{\xi} \|_{L^2(\mathbb{R}^2; \mathbf{x'})} \le C \| \partial_{\mathbf{x}'}^2 \xi \|_{L^2(\mathbb{R}^2; \mathbf{x'})}.
\]
For the $L^{4/3}$ norm:
\begin{align*}
    & \varepsilon^{-4} \| \mathscr{R}_1(D) \tanh{(|D|)} (\eta_1 G_0 \xi_1 - \eta_2 G_0 \xi_2) \|_{L^{\frac{4}{3}}(\mathbb R^2)} \\
    &= \varepsilon^{-\frac{7}{4}} \| \mathscr{R}_1(D) \tanh{(|D|)} (\eta_1 G_0 \xi_1 - \eta_2 G_0 \xi_2) \|_{L^{\frac{4}{3}}(\mathbb R^2; \mathbf{x'})} \\
    &\le C \varepsilon^{-\frac{7}{4}} \big( \| (\eta_1 - \eta_2) G_0 \xi_1 \|_{L^{\frac{4}{3}}(\mathbb R^2; \mathbf{x'})}  + \| \eta_2 G_0 (\xi_1 - \xi_2) \|_{L^{\frac{4}{3}}(\mathbb R^2; \mathbf{x'})} \big) \\
    &\le C\ep^{-\frac{7}{4}} \big( \| \eta_1 - \eta_2 \|_{L^{4}(\mathbb R^2; \mathbf{x'})} \| G_0 \xi_1 \|_{L^2(\mathbb R^2; \mathbf{x'})} \\
    &\hspace{2em} + \| \eta_2 \|_{L^{4}(\mathbb R^2; \mathbf{x'})} \| G_0 (\xi_1 - \xi_2) \|_{L^2(\mathbb R^2; \mathbf{x'})} \big) \\
    &\le C\ep^{-\frac{7}{4}} \big( \| \eta_1 - \eta_2 \|_{L^{4}(\mathbb R^2; \mathbf{x'})} \| \partial_{\mathbf{x}'}^2 \xi_1 \|_{L^2(\mathbb R^2; \mathbf{x'})} \\
    &\hspace{2em} + \| \eta_2 \|_{L^{4}(\mathbb R^2; \mathbf{x'})} \| \partial_{\mathbf{x}'}^2 (\xi_1 - \xi_2) \|_{L^2(\mathbb R^2; \mathbf{x'})} \big) \\
    &\le C\ep \| \psi_1 - \psi_2 \|_{L^{4}(\mathbb{R}^2)} + C\ep \| \nabla_{\ep}^2 (\phi_1 - \phi_2) \|_{L^2(\mathbb R^2)} \\
    &\le C\ep^{\frac{1}{2}} \| \phi_1 - \phi_2 \|_{*}.
\end{align*}
Moreover,
\begin{align*}
    & \ep^{-4} \| \partial_1 \mathscr{R}_1(D) \tanh{(|D|)} \left( \eta_1 G_0 \xi_1 - \eta_2 G_0 \xi_2 \right) \|_{L^2(\mathbb R^2)} \\
    &= \ep^{-1} \| \partial_1 \mathscr{R}_1(D) \tanh{(|D|)} \left( h_1 G_0 f_1 - h_2 G_0 f_2 \right) \|_{L^2(\mathbb R^2)} \\
    &\le C\ep^{-1} \| \partial_1  \tanh{(|D|)} \left( h_1 G_0 f_1 - h_2 G_0 f_2 \right) \|_{L^2(\mathbb R^2)} \\
    &\le C\ep^{-1} \| \partial_1 |D| \left( h_1 G_0 f_1 - h_2 G_0 f_2 \right) \|_{L^2(\mathbb R^2)} \\
    &\le C \Big( \| \nabla_{\ep}^2 (\psi_1 - \psi_2) \|_{L^4(\mathbb R^2)} \| G_0 f_1 \|_{L^4(\mathbb R^2)} \\
    &\hspace{1em} + \| \nabla_{\ep} (\psi_1 - \psi_2) \|_{L^4(\mathbb R^2)} \| \nabla_{\ep} G_0 f_1 \|_{L^4(\mathbb R^2)} \\
    &\hspace{1em} + \| \psi_1 - \psi_2 \|_{L^{\infty}(\mathbb R^2)} \| \nabla_{\ep}^2 G_0 f_1 \|_{L^2(\mathbb R^2)} \\
    &\hspace{1em} + \| \nabla_{\ep}^2 h_2 \|_{L^{\infty}(\mathbb R^2)} \| G_0 (\phi_1 - \phi_2) \|_{L^2(\mathbb R^2)} \\
    &\hspace{1em} + \| \nabla_{\ep} h_2 \|_{L^{\infty}(\mathbb R^2)} \| \nabla_{\ep} G_0 (\phi_1 - \phi_2) \|_{L^2(\mathbb R^2)} \\
    &\hspace{1em} + \| h_2 \|_{L^{\infty}(\mathbb R^2)} \| \nabla_{\ep}^2 G_0 (\phi_1 - \phi_2) \|_{L^2(\mathbb R^2)} \Big).
\end{align*}
By the Gagliardo-Nirenberg inequality:
\begin{align*}
    \| \nabla_{\ep}^2 (\psi_1 - \psi_2) \|_{L^4(\mathbb R^2)}
        & \le C\ep^{-\frac{1}{4}} \| \nabla_{\ep}^3 (\psi_1 - \psi_2) \|_{L^2}^{1/2} \| \nabla_{\ep}^2 (\psi_1 - \psi_2) \|_{L^2}^{1/2} \\
        & \le C\ep^{-\frac{1}{4}} \| \psi_1 - \psi_2 \|_{h}, \\
    \| \nabla_{\ep} (\psi_1 - \psi_2) \|_{L^4(\mathbb R^2)}
        & \le C\ep^{-\frac{1}{4}} \| \nabla_{\ep}^2 (\psi_1 - \psi_2) \|_{L^2}^{1/2} \| \nabla_{\ep} (\psi_1 - \psi_2) \|_{L^2}^{1/2} \\
        & \le C\ep^{-\frac{1}{4}} \| \psi_1 - \psi_2 \|_{h},
\end{align*}
and 
\begin{align*}
     \|\nabla_{\ep}G_0f_1\|_{L^4(\mathbb R^2)}&\le C\ep^{-\frac{1}{4}}\|\nabla_{\ep}^2G_0f_1\|_{L^2(\mathbb R^2)}^{1/2}\|\nabla_{\ep}G_0f_1\|_{L^2(\mathbb R^2)}^{1/2}\\
    &\le C\ep^{\frac{7}{4}}\|\nabla_{\ep}^4f_1\|_{L^2(\mathbb R^2)}^{1/2}\|\nabla_{\ep}^3f_1\|_{L^2(\mathbb R^2)}^{1/2}\le C \ep^{\frac{7}{4}},\\
    \|G_0f_1\|_{L^4(\mathbb R^2)}&\le C \ep^{\frac{7}{4}}.
\end{align*}
Thus,
\begin{align*}
    & \ep^{-4} \| \partial_1 \mathscr{R}_1(D) \tanh{(|D|)} \left( \eta_1 G_0 \xi_1 - \eta_2 G_0 \xi_2 \right) \|_{L^2} \\
    & \le C\ep^{\frac{3}{2}} \left( \| \psi_1 - \psi_2 \|_{h} + \| \phi_1 - \phi_2 \|_{*} \right) \\
    & \le C\ep \| \phi_1 - \phi_2 \|_{*}.
\end{align*}
Analogous estimates hold for $\ep^{-3} \partial_2 \mathscr{R}_2(D) \tanh (|D|) (\eta G_0 \xi)$, yielding:
\begin{align*}
    \ep^{-4} \| \mathscr{R}_1(D) \tanh{(|D|)} (\eta_1 G_0 \xi_1 - \eta_2 G_0 \xi_2) \|_{**} & \le C\ep^{\frac{1}{2}} \| \phi_1 - \phi_2 \|_{*}, \\
    \ep^{-3} \| \mathscr{R}_2(D) \tanh{(|D|)} (\eta_1 G_0 \xi_1 - \eta_2 G_0 \xi_2) \|_{**} & \le C\ep^{\frac{1}{2}} \| \phi_1 - \phi_2 \|_{*}.
\end{align*}

\noindent\textbf{Conclusion}
The remaining terms in $\hat{P}_1$, $\hat{P}_2$, $P_3$, and $\hat{P}_4$ admit similar estimates. Applying \eqref{5-13} and Proposition \ref{prop4-2}:
\[
\left\| \mathcal{N}(\phi_1) - \mathcal{N}(\phi_2) \right\|_* \leq C \varepsilon^{\frac{1}{2}} \| \phi_1 - \phi_2 \|_*.
\]
The existence of a solution $\phi$ to \eqref{5-7} follows from the contraction mapping principle, which establishes the main theorem.
\end{proof}

\begin{center}
{\bf Acknowledgments}
\end{center}

The research of C.F. Gui is supported by University of Macau research grants CPG2024-00016-FST, CPG202500032-FST, SRG2023-00011-FST, MYRG-GRG2023-00139-FST-UMDF, UMDF Professorial Fellowship of Mathematics, Macao SAR FDCT0003/2023/RIA1 and Macao SAR FDCT0024/2023/RIB1. The research of Y. Liu is supported by  NSFC No. 12471204. The research of J.C. Wei is supported by National Key R\&D Program of China 2022YFA1005602, and Hong Kong General Research Fund “New frontiers in singular limits of nonlinear partial differential equations”. The research of W. Yang is partially supported by National
Key R\&D Program of China 2022YFA1006800, NSFC No.
12171456 and No. 12271369, FDCT No. 0070/2024/RIA1, UMDF No. TISF/2025/006/FST, No. MYRG-GRG2024-00082-FST, No. MYRG-GRG2025-00051-FST and Startup
Research Grant No. SRG2023-00067-FST.

\medskip

\noindent
{\bf Note added after completion:} After the completion of our work—presented by the fourth author at the conference ``Partial Differential Equations and Fluid Mechanics" held at the University of Bath, September 2–4, 2025—we learned from Professor Mark Groves that he and Wahlen have independently obtained similar results. See \cite{Groves2025-1, Groves2025-2}

\appendix 
\renewcommand{\thetheorem}{A.\arabic{theorem}}
\setcounter{theorem}{0}
\renewcommand{\thecorollary}{A.\arabic{corollary}}
\setcounter{corollary}{0}
\begin{appendices}
	\section*{Appendix}
	
	In this appendix, we present some 
	$W^{k,p}$-boundedness results for operators employed earlier, with proofs detailed in \cite{Grafakos2008, Grafakos2014}.

	\begin{theorem}[Mikhlin Multiplier Theorem]
		Let $m: \mathbb{R}^d \backslash\{0\} \rightarrow \mathbb{C}$ be a measurable function. Assume there exist constants $C_\alpha>$ 0 such that the following conditions holds:
		For all multi-indices $\alpha=\left(\alpha_1, \ldots, \alpha_d\right)$ with $|\alpha|=\alpha_1+\cdots+\alpha_d \leq k$, where $k>\frac{d}{2}$ is an integer,
		$$
		\left|\partial^\alpha m(\xi)\right| \leq C_\alpha|\xi|^{-|\alpha|}, \quad \forall \xi \neq 0
		$$
		Here, $\partial^\alpha=\frac{\partial^{|\alpha|}}{\partial \xi_1^{\alpha_1} \ldots \partial \xi_d^{\alpha_d}}$ denotes the mixed partial derivative. Then $m$ is an $L^p$-Fourier multiplier for all $1<p<\infty$. Specifically, the operator $T_m$ defined by
		$$
		T_m f:=\mathcal{F}^{-1}(m \cdot \widehat{f}), \quad f \in \mathcal{S}\left(\mathbb{R}^d\right)
		$$
		extends to a bounded operator on $L^p\left(\mathbb{R}^d\right)$, satisfying
		$$
		\left\|T_m f\right\|_{L^p\left(\mathbb{R}^d\right)} \leq C_{p, d}\Bigg(\sum_{|\alpha| \leq\left\lfloor\frac{d}{2}\right\rfloor+1} C_\alpha\Bigg)\|f\|_{L^p\left(\mathbb{R}^d\right)}
		$$
		where $C_{p, d}>0$ depends only on $p$ and $d$.
	\end{theorem}

	By the Mikhlin multiplier theorem, we can readily establish the boundedness of the following operators:

	\begin{corollary}
		The operators $(1 + e^{-2|D|})^{-1}$, $R_i = i \frac{D_i}{|D|}$, and $\tanh(|D|)$ are bounded in $W^{k,p}(\mathbb{R}^d)$, where $D_i=-i\partial_{x_i}$ and $|D|=(-\Delta)^{1/2}$.
	\end{corollary}
	
	To establish $L^p$ norm estimates, we need the following well-known inequalities.
	\begin{theorem}[Hardy-Littlewood-Paley theorem on $\mathbb{R}^2$]
		Assume $p \in[1,2]$. Then
		$$
		\|\widehat f\|_{L^p\left(|x|^{-2(2-p)}\right)} \leqslant C_p\|f\|_{L^p}
		$$
		where  we denote by $L^p(w)$ the $L^p$-space with respect to $w(x) d x$. The norm on $L^p(w)$ is
		$$
		\|f\|_{L^p(w)}=\left(\int_{\mathbb{R}^2}|f(x)|^p w(x) d x\right)^{1 / p}.
		$$
	\end{theorem}
	\begin{proof}
		We consider the mapping $(T f)(\xi)=|\xi|^n \widehat{f}(\xi)$. By Plancherel theorem, we have
		$$
		\|T f\|_{L_*^2\left(|\xi|^{-2 n}\right)} \leqslant\|T f\|_{L^2\left(|\xi|^{-2 n}\right)}=\|\widehat{f}\|_2=\|f\|_2
		$$
		which implies that $T$ is of weak type $(2,2)$. We now work towards showing that $T$ is of weak type $(1,1)$. Thus, the Marcinkiewicz interpolation theorem implies the theorem.
		Now, consider the set $E_\alpha=\left\{\xi:|\xi|^n|\widehat{f}(\xi)|>\alpha\right\}$. For simplicity, we let $\nu$ denote the measure $|\xi|^{-2 n} d \xi$ and assume that $\|f\|_1=1$. Then, $|\hat{f}(\xi)| \leqslant\left(\frac{|\omega|}{2 \pi}\right)^{n / 2}$. For $\xi \in E_\alpha$, we therefore have $\alpha<\left(\frac{|\omega|}{2 \pi}\right)^{n / 2}|\xi|^n$. Consequently,
		$$
		(T f)_*(\alpha)=\nu\left(E_\alpha\right)=\int_{E_\alpha}|\xi|^{-2 n} d \xi \leqslant \int_{|\xi|^n>\left(\frac{|\omega|}{2 \pi}\right)^{-n / 2}}|\xi|^{-2 n} d \xi \leqslant C \alpha^{-1}
		$$
		Thus, we prove that
		$$
		\alpha \cdot(T f)_*(\alpha) \leqslant C\|f\|_1
		$$
		which implies $T$ is of weak type $(1,1)$. Therefore, we complete the proof.
	\end{proof}

\end{appendices}

\end{document}